\documentclass[reqno,10pt]{amsart}

\usepackage{mathtools}
\usepackage{float}
\usepackage{amsmath}
\usepackage{amsfonts}
\usepackage{amssymb}
\usepackage{graphicx}
\usepackage{hyperref}
\usepackage{mathrsfs}
\usepackage{pb-diagram}
\usepackage{epstopdf}
\usepackage{bm}
\usepackage{epsfig,epsf,xypic,epic}
\usepackage{caption}
\usepackage{subcaption}
\usepackage{geometry}
\usepackage{enumitem}

\usepackage[all, cmtip]{xy}
\usepackage{amstext}
\usepackage[all]{xy}
\usepackage{yhmath}
\usepackage{mathrsfs}
\usepackage{bbding}
\usepackage{orcidlink}

\usepackage{color}

\def\E{\ifmmode{\mathbb E}\else{$\mathbb E$}\fi} 
\def\N{\ifmmode{\mathbb N}\else{$\mathbb N$}\fi} 
\def\R{\ifmmode{\mathbb R}\else{$\mathbb R$}\fi} 
\def\Q{\ifmmode{\mathbb Q}\else{$\mathbb Q$}\fi} 
\def\C{\ifmmode{\mathbb C}\else{$\mathbb C$}\fi} 
\def\H{\ifmmode{\mathbb H}\else{$\mathbb H$}\fi} 
\def\Z{\ifmmode{\mathbb Z}\else{$\mathbb Z$}\fi} 
\def\P{\ifmmode{\mathbb P}\else{$\mathbb P$}\fi} 
\def\T{\ifmmode{\mathbb T}\else{$\mathbb T$}\fi} 
\def\SS{\ifmmode{\mathbb S}\else{$\mathbb S$}\fi} 
\def\DD{\ifmmode{\mathbb D}\else{$\mathbb D$}\fi} 
\def\R{\ifmmode{\mathbb R}\else{$\mathbb R$}\fi} 

\newcommand{\del}{\partial}

\newcommand{\ben}{\begin{enumerate}}
\newcommand{\een}{\end{enumerate}}
\newcommand{\be}{\begin{equation}}
\newcommand{\ee}{\end{equation}}
\newcommand{\bea}{\begin{eqnarray}}
\newcommand{\eea}{\end{eqnarray}}
\newcommand{\bc}{\begin{center}}
\newcommand{\ec}{\end{center}}
\newcommand{\beastar}{\begin{eqnarray*}}
\newcommand{\eeastar}{\end{eqnarray*}}

\theoremstyle{theorem}
\newtheorem{thm}{Theorem}[section]
\newtheorem{cor}[thm]{Corollary}
\newtheorem{lem}[thm]{Lemma}
\newtheorem{prop}[thm]{Proposition}

\theoremstyle{definition}
\newtheorem{defn}[thm]{Definition}
\newtheorem{rem}[thm]{Remark}

\newtheorem{nota}[thm]{Notation}
\newtheorem{warn}[thm]{Warning}
\newtheorem{prob}[thm]{Problem}

\newtheorem{ques}[thm]{Question}

\newtheorem{exm}[thm]{Example}

\numberwithin{equation}{section}

\def\R{{\mathbb R}}

\def\E{{\mathbb E}}
\def\Z{{\mathbb Z}}
\def\C{{\mathbb C}}
\def\R{{\mathbb R}}

\def\N{{\mathbb N}}

\def\SS{{\mathcal S}}

\def\DD{{\mathcal D}}

\def\11{{\mathbb I}}

\def\C{\mathbb{C}}
\def\Z{\mathbb{Z}}

\def\T{\mathbb{T}}

\def\Q{\mathbb{Q}}

\def\E{\ifmmode{\mathbb E}\else{$\mathbb E$}\fi} 
\def\N{\ifmmode{\mathbb N}\else{$\mathbb N$}\fi} 
\def\R{\ifmmode{\mathbb R}\else{$\mathbb R$}\fi} 
\def\Q{\ifmmode{\mathbb Q}\else{$\mathbb Q$}\fi} 
\def\C{\ifmmode{\mathbb C}\else{$\mathbb C$}\fi} 
\def\H{\ifmmode{\mathbb H}\else{$\mathbb H$}\fi} 
\def\Z{\ifmmode{\mathbb Z}\else{$\mathbb Z$}\fi} 
\def\P{\ifmmode{\mathbb P}\else{$\mathbb P$}\fi} 
\def\SS{\ifmmode{\mathbb S}\else{$\mathbb S$}\fi} 
\def\DD{\ifmmode{\mathbb D}\else{$\mathbb D$}\fi} 

\def\R{{\mathbb R}}

\def\E{{\mathbb E}}
\def\Z{{\mathbb Z}}
\def\C{{\mathbb C}}
\def\R{{\mathbb R}}

\def\N{{\mathbb N}}

  \def\P{g}

\def\CC{{\mathcal C}}

\def\CE{{\mathcal E}}

\def\CG{{\mathcal G}}

\def\CI{{\mathcal I}}
\def\CJ{{\mathcal J}}

\def\CL{{\mathcal L}}

\def\CU{{\mathcal U}}
\def\CV{{\mathcal V}}
\def\CW{{\mathcal W}}

\def\CZ{{\mathcal Z}}
\def\darr#1{\raise1.5ex\hbox{$\leftrightarrow$}
\mkern-16.5mu #1}

\def\roughly#1{\raise.3ex\hbox{$#1$\kern-.75em
\lower1ex\hbox{$\sim$}}}

\def\opname#1{\mathop{\kern0pt{\rm #1}}\nolimits}

\def\dim{\opname{dim}}

\def\supp{\opname{supp}}

\DeclareMathOperator {\Symp} {Symp}

\DeclareMathOperator {\Diff} {Diff}

\DeclareMathOperator{\image}{\mathrm{Image}}

\DeclareMathOperator{\Aut}{\mathrm{Aut}}

\DeclareMathOperator{\id}{\mathrm{id}}

\DeclareMathOperator{\Cont}{\mathrm{Cont}}

\DeclareMathOperator{\Hameo}{\mathrm{Hameo}}
\DeclareMathOperator{\Sympeo}{\mathrm{Sympeo}}
\DeclareMathOperator{\Homeo}{\mathrm{Homeo}}

\begin{document}

\quad \vskip1.375truein

\title[simplicity of contactomorphism group]{Simplicity of the contactomorphism group of
finite regularity}

\author{Yong-Geun Oh\, \orcidlink{0000-0003-0333-4308}}
\address{Center for Geometry and Physics, Institute for Basic Sciences (IBS), Pohang, Korea \&
Department of Mathematics, POSTECH, Pohang, Korea}
\email{yongoh1@postech.ac.kr}

\begin{abstract} For a given coorientable contact manifold $(M^{2n+1},\xi)$, 
we consider the group $ \Cont_c^{(r,\delta)}(M,\alpha)$ 
consisting of $C^{r,\delta}$ contactomorphisms with compact support which is 
equipped  with $C^{r,\delta}$-topology of H\"older regularity $(r,\delta)$ for $r \geq 1$ 
and $0 <\delta \leq 1$.
We prove that for all H\"older class exponents 
with $r > n + 2$ or $r = n+1, \, \frac12 < \delta \leq 1$ (resp. $r < n+1$ or $r = n+1$ and 
$ 0< \delta <\frac12$),
the group is a perfect  (and so a simple) group. In particular,
$\Cont_c^r(M,\xi)$ is simple for all integer $r \geq 1$. 
For the case of $\Cont_c^{(r,\delta)}(M,\alpha)$ of general H\"older regularity, we prove the simplicity for all pairs $(r,\delta)$
leaving \emph{only} the case of $(r,\delta) = (n+1,\frac12)$ open.
\end{abstract}

\thanks {This work is supported by the IBS project \# IBS-R003-D1.}
\keywords{$C^r$ contactomorphism group, simplicity, Mather-Rybicki's constructions, contact
product, Legendrianization, contact potential, H\"older regularity threshold}
 \maketitle

\def\mq{\mathfrak{q}}
\def\mp{\mathfrak{p}}
\def\mH{\mathfrak{H}}
\def\mh{\mathfrak{h}}
\def\ma{\mathfrak{a}}
\def\ms{\mathfrak{s}}
\def\mm{\mathfrak{m}}
\def\mn{\mathfrak{n}}

\def\Hoch{{\tt Hoch}}
\def\mt{\mathfrak{t}}
\def\ml{\mathfrak{l}}
\def\mT{\mathfrak{T}}
\def\mL{\mathfrak{L}}
\def\mg{\mathfrak{g}}
\def\md{\mathfrak{d}}

\tableofcontents

\section{Introduction}

Let $(M,\xi)$ be a connected smooth contact manifold. The set
\be\label{eq:weak-contact}
\{f \in \Diff^r(M) \mid df (\xi) \subset \xi\} =: \Cont^r(M,\xi)
\ee
is a subgroup of $\Diff^r(M)$ for all $r \geq 1$, even for the H\"older regularity
$(r,\delta)$ with $r \geq 1$, $0 < \delta \leq1$. 
The general topology of this group is not well-behaved. For example, it is not
known whether the group is locally contractible to the knowledge of the present author.

On the other hand, when $(M,\xi)$ is coorientable and
 equipped with a contact form $\alpha$, any smooth contactomorphism $f$ satisfies
$$
f^*\alpha = \lambda_f\, \alpha
$$
for a nowhere vanishing smooth function $\lambda_f:M \to \R$, which we call the
\emph{conformal factor} of $f$. Then we consider the logarithm $\ell_f = \log \lambda_f$
which we call the \emph{conformal exponent} following the practice exercised in 
\cite{oh:contacton-Legendrian-bdy,oh:entanglement1,oh:shelukhin-conjecture}.

The following subset of $\Cont^r(M,\xi)$
is a subgroup of $\Diff^r(M)$ which is more suitable e.g., for the simplicity study of 
the contactomorphism group of finite regularity. Following Tsuboi \cite{tsuboi3}, we adopt
the following definition.

\begin{defn}[$C^r$ contact diffeomorphism]\label{defn:Contrc} A $C^r$ diffeomorphism $f: M \to M$ is called
a $C^r$ contactomorphism with respect to $\alpha$ if $\lambda_f = \lambda_f^\alpha$ is a positive $C^r$ function. 
We denote by $\Cont^r(M,\alpha)$ the set of $C^r$ contactomorphisms.
\end{defn} 
It is straightforward to see that this definition does not depend on the choice of
\emph{smooth} contact form $\alpha$ and that the set of $C^r$ contactomorphisms forms a
subgroup of $\Diff_c^r(M)$ which is locally contractible. (See the discussion in 
\cite{lychagin}, \cite{banyaga:book}, \cite[Section 2]{tsuboi3} and 
Sections \ref{sec:contactomorphisms}, \ref{sec:legendrianization}, especially 
the identitiy \eqref{eq:halphaalpha'},
 of the present paper.) 

\begin{rem}
One may call an element of $\Cont^r(M,\xi)$
a \emph{weakly-$C^r$ contactomorphism} but we do not concern the group
\eqref{eq:weak-contact} in the present paper except when we discuss the case of 
contact homeomorphisms later in \ref{subsec:discussion}. Its definition a priori does 
not involve the choice of a 
contact form in its definition and so defined even for non-coorientable contact manifold.
This is usually denoted by $\Cont(M,\xi)$  in the literature when $r = \infty$. 
Obviously when $\xi$ is coorientable,
we have $\Cont^\infty(M,\xi) = \Cont^\infty(M,\alpha)$ for any choice of smooth contact form $\alpha$,
and hence two definitions coincide for the smooth, i.e., for the $C^\infty$ case.
\end{rem}

\subsection{Statement of main results}

We denote  the set of compactly supported $C^r$ contactomorphisms by 
$$
\Cont_c^r(M,\alpha)
$$
and  its identity component by $\Cont_c^r(M,\xi)_0$. We also denote by
$B\overline{\Cont}_c^r(M,\alpha)$ Haefliger's classifying space \cite{haefliger}, \cite{tsuboi3} of 
the group $\Cont_c^r(M,\alpha)$.

The following two results have been previously
known concerning the simplicity of contactomorphism groups:
\begin{itemize}
\item {(Tsuboi \cite{tsuboi3})} For $1 \leq r < n + \frac32$, 
$H_1(B\overline{\Cont}_c^r(M,\alpha);\Z) = 0$. 
\item {(Rybicki \cite{rybicki2})}  For $r = \infty$, 
$H_1(B\overline{\Cont}_c^r(M,\alpha);\Z) = 0$. 
\end{itemize}
In particular, $\Cont_c^r(M,\alpha)$ is a perfect (and so simple) group for the corresponding $r$.
In their papers, the following contact version of the fragmentation lemma
is an important ingredient. The proof follows from the fact that any contactomorphism contact isotopic to the identity is
generated by a contact Hamiltonian and so the fragmentation lemma can be proved by the same
argument as that of the symplectic case \cite{banyaga}.  (See \cite{banyaga:book}.)

\begin{lem}[Fragmentation Lemma]\label{lem:fragmentation-intro}
Let $f \in \Cont_c(M,\alpha)_0$ and let $\{U_i\}_{i=1}^k$ be an open cover
of $M$. Then there exists $f_j \in \Cont_c(M,\alpha)_0$, $j = 1, \ldots, \ell$ with 
$f = f_1 \circ f_2 \cdots \circ f_\ell$
such that $\supp(f_j) \subset U_{i(j)}$ for all $j$. The same holds for contact isotopies of 
contactomorphisms.
\end{lem}

In the present paper, we prove the following.

\begin{thm}\label{thm:perfect} For any integer $r \geq n+2$, $H_1(B\overline{\Cont}_c^r(M,\alpha);\Z) = 0$. In particular, 
$\Cont_c^r(M,\alpha)$ is a simple group.
\end{thm}
Therefore combining the above three results, we obtain the following 
complete answer to the simplicity question
for the group of contactomorphisms of $C^r$ regularity with integer $r$ (including the case of $r = \infty$ \cite{rybicki2}).

\begin{cor}\label{cor:integer} Assume $\dim M = 2n+1$ with $n \geq 1$. 
Then for any integer $r \geq 1$ including $r = \infty$,
$H_1(B\overline{\Cont}_c^r(M,\alpha);\Z) = 0$. In particular,
$\Cont_c^r(M,\alpha)$ is a perfect group for all integer $r \geq 1$ and $r = \infty$.
\end{cor}

The above result can be further extended to the H\"older class of regularites $(r,\delta)$ with
$r \in \N$ and $0 < \delta \leq 1$. 
 (See also Section \ref{sec:contactomorphisms} for the precise definition
 of $C^{r,\beta}$ contactomorphisms and the set $\Cont^{(r,\delta)}(M,\alpha)$
 consisting thereof.)
By definition, this set $\Cont^{(r,\delta)}(M,\alpha)$ 
forms a subgroup of $\Cont^r(M,\alpha)$ containing $\Cont^{r+1}(M,\alpha)$.

As in \cite{mather}, we derive Theorem \ref{thm:perfect} as a consequence of
the following general result for the case of H\"older regularity class.
We recommend the readers, who want to immediately see 
where the threshold $(n+1,\frac12)$ comes from, now visiting Section \ref{sec:threshold}.

\begin{thm}\label{thm:hoelder-up}
Assume $\dim M = 2n+1$ with $n \geq 1$. Then
$H_1(B\overline{\Cont}_c^{(r, \delta)}(M,\alpha);\Z) = 0$, 
and hence $\Cont_c^{(r,\delta)}(M,\alpha)$ is a perfect group for all 
pairs $(r,\delta)$ with $r > n+2$ or $r = n+1$ and $\frac12 < \delta \leq 1$.
\end{thm}

By reversing the direction of the construction as in \cite{mather2}, we obtain:

\begin{thm}\label{thm:hoelder-down}
Assume $\dim M = 2n+1$ with $n \geq 1$. Then
$H_1(B\overline{\Cont}_c^{(r, \delta)}(M,\alpha);\Z) = 0$
and hence $\Cont_c^{(r,\delta)}(M,\alpha)$ is a perfect group for all 
pairs $(r,\delta)$ with $r < n+1$ or $r = n+1$ and $0 \leq \delta < \frac12$.
\end{thm}
Appearance of the half integer threshold of $\delta$ has its origin from the
asymmetry of the orders of power of $A$ in that when $A \geq 1$, the norm of the contact scaling $(z,q,p) \mapsto (A^2 z, A q, A p)$
is $A^2$ but the norm of its inverse $(z,q,p) \mapsto (A^{-2} z, A^{-1} q, A^{-1} p)$
 is $A^{-1}$. 

These leave the following question open which is the contact analog to the celebrated
open question \cite{mather,mather2} on simplicity of the $C^r$ diffeomorphism 
group with $r = n+1$ for an $n$-manifold. We would like to compare this open problem with 
that of the diffeomorphism case: Mather proved in \cite{mather,mather2} the corresponding result for the diffeomorphism group
$\Diff_c(M^m)^r$ of connected $m$-manifold $M$, \emph{if $r \neq m+1$}.
In the mean time, Theorem \ref{thm:hoelder-down} recovers Tsuboi's 
result \cite{tsuboi3} whose proof is in the same spirit as Theorem \ref{thm:perfect}
similarly as in \cite{mather,mather2}. On the other hand,
Epstein \cite{epstein:simplicity} proved the simplicity of the homeomorphism group
$\Homeo_c(M)$.

\begin{ques} Suppose $\dim M = 2n+1$. Is $\Cont_c^{(r, \delta)}(M,\xi)$ simple when 
$r= n+1$ and $\delta = \frac12$?
\end{ques}

\subsection{Rybicki's contactization of Mather's construction}

The main methodology of the proof is again those introduced by Mather \cite{mather}, 
\cite{epstein:commutators} which also relies on certain fragmentation lemma
and the application of Schauder-Tychonoff's fixed point theorem, which has been also applied 
by Tsuboi \cite{tsuboi3} and Rybicki \cite{rybicki2} to the simplicity problem of
contactomorphisms. Using the fragmentation lemma and Epstein's reduction
\cite{epstein:simplicity}, the proof of simplicity (or rather perfectness) is  reduced to the
case of Euclidean space $\R^m$ ($m = 2n+1$), which in turn crucially relies on the `linear structure' of 
the $\R^m$. Some fundamental properties of the Euclidean space (or the torus)
 used in Mather's proof are the simple facts:
 \begin{itemize}
 \item They carry the abelian group structure induced by the
 linear addition operator $+$ thereon.
 \item Any diffeomorphism $C^1$-close to the identity can be written as
$f = \id + v$ for $v$ is a $\R^m$-valued function that is $C^1$-close to the zero function.
\end{itemize}
(See \cite{mather4} for a detailed analysis of what obstructs the method of \cite{mather,mather2}
applied to the case of $r = \dim M+1$ for the general diffeomorphism case.)
\emph{Such a simple linear description of $C^1$ neighborhood fails to hold for the contactomorphisms.}
This prevents one from directly borrowing Mather's construction of \emph{rolling-up operators}
to the case of contactomorphisms.

New ingredients introduced by Rybicki \cite{rybicki2} in this regard are the following:
\begin{enumerate}
\item Usage of the local parametrization of $C^1$ neighborhood of the identity via the 
Legendrianization and the generating functions, which we name the \emph{contact potential},
of the relevant contactomorphisms in his
construction of contact version of unfolding-fragmentation operators. This space of \emph{real-valued} functions
is the domain of the function space where his application of
Schauder-Tychonoff's fixed point theorem is made.
\item  Usage of a new fragmentation lemma based on this contact potential 
and the \emph{contact cylinders} $(\CW_k^{2n+1},\alpha_k)$ of the form
\bea\label{eq:contact-cylinder}
\CW_k^{2n+1} &: = & S^1 \times T^*(T^k \times \R^{n-k}) \cong
S^{k+1}  \times \R^{n-k} \times \R^n, \nonumber\\
\alpha_k  &=  & d\xi_0 - \sum_{i=1} p_i \, d\xi_i
\eea
where $T^k = (S^1)^k$ and $(\xi_0, \ldots,\xi_k)$ are standard coordinates 
($S^1$-valued) of $(S^1)^{k+1}$ and $(\xi_{k+1}, \ldots, \xi_n)$ are those of $\R^{n-k}$
and $p = (p_1, \ldots, p_n)$ the conjugate coordinates of $(\xi_1, \cdots, \xi_n)$.
\end{enumerate}

We closely follow the scheme of Rybicki  which is used for the $C^\infty$ case.
However we need to make both geometric constructions and derivative estimates optimal in
all the steps of Rybicki's proof which deals with the $C^\infty$ case by suitably adapting  
Epstein's simplification of the simpleness proof of $\Diff^\infty(M)$ exercised in \cite{epstein:commutators}
  to the contact case (without using
the Nash-Moser implicit function theorem originally used in \cite{thurston}, 
\cite[Appendix]{mather4}): 
\begin{enumerate}
\item  We need to package the contact
 geometry elements employed in \cite{rybicki2} systematically in the framework of 
 contact Hamiltonian geometry and calculus of 
 \cite{oh:contacton-Legendrian-bdy,oh:entanglement1,oh:shelukhin-conjecture}.
\item We need to identify the optimal form of \emph{contact homothetic transformations} 
for the definition of the rolling-up operator and the unfolding-fragmentation operators.
(See Section \ref{sec:shifting-supports}.)
\item Using this optimal geometric package, we derive the \emph{optimal version} of 
many of the rough estimates carried out in \cite{rybicki2} by making finer choices of  various numerical constants 
 appearing in the construction. 
\item Both the statement of \cite[Lemma 8.6 (2)]{rybicki2} and its proof
 are imprecise and need to be made precise and then proved. (See Remark \ref{rem:rybicki-error}
 below for the reason why.)
 Because of this, we need to provide its details with some corrections and
 amplifications of the construction of the unfolding-fragmentation operator
 associated the $N$-fragmentation with $N > 2$ as given in Section \ref{sec:unfolding-fragmentation} and Part III
 of the present paper. 
\end{enumerate}
 
Concerning the necessity of the optimal estimates, they are not needed
for the $C^\infty$ case studied by Rybicki but only
some rough estimates are enough as done in \cite{rybicki2}. But in our study of
finite regularity, especially to determine the lower threshold $r = n+ 2$ and 
the upper threshold $r = n+1$, it is essential for us to first make the optimal choice of the homothetic transformation 
and then make the optimal estimates  
that appear in the course of studying the $C^r$ (or $C^{(r,\delta)}$) norms of 
various contactomorphisms  and of their products.  Our estimates then also crucially
rely on the systematic contact Hamiltonian calculus involving the conformal exponents 
and other basic contact Hamiltonian geometry as exercised in our study of contact instantons in
\cite{oh:entanglement1,oh:shelukhin-conjecture}, for example.

\begin{warn}\label{warn:notation}
We adopt the notations used in \cite{mather}-\cite{mather3}, \cite{epstein:commutators} and \cite{rybicki2}, 
especially those from \cite{rybicki2} so that the readers can easily compare 
the details of the present paper with those in \cite{rybicki2} corresponding thereto.
However we warn the readers that 
even though we adopt the same notations for the purpose of comparison, the detailed numerics appearing in the definitions are 
almost never the same as
those from \cite{rybicki2}, since we make the optimal choices of various numerical constants 
and orders of powers.  The systematic framework  of contact Hamiltonian geometry and calculus 
developed in  \cite{oh:contacton-Legendrian-bdy,oh:shelukhin-conjecture} enables us to 
find these optimal choices  for the various constructions
which is  crucial  in our determination of the threshold 
\be\label{eq:r-threshold}
r = n + 2
\ee
for the lower threshold, $r = n+1$ for the upper threshold.
(Also the latter threshold corresponds to Tusboi's upper threshold $r = n + \frac32$ in \cite{tsuboi3}.)
These are the counterparts of Mather's thresholds $r = n+1$ \cite{mather} and $r = n$ 
\cite{mather2}  respectively  for the case of diffeomorphisms.
\end{warn}

\subsection{Discussion and open problems}
\label{subsec:discussion}

\subsubsection{Relationship with \cite{rybicki2} and \cite{tsuboi3}}

Once the estimate for $r = n+2$ 
$$
\|u\|_{n+2} \leq \varepsilon_{n+2}
$$ 
is given, we can inductively obtain a sequence $\varepsilon_r$ for $r \geq n+2$, 
adapting the argument 
of Epstein \cite[p.121]{epstein:simplicity} for the case of $\Diff_c(M)$, 
such that the map $\vartheta: \CU \to \CU$ is defined on
$$
\CU = \{u \in C_c^\infty(\R^{2n+1}) \mid \|u\|_r \leq \varepsilon_r\}
$$
which is a convex closed subset of the Frechet space $C_c^\infty(\R^{2n+1})$. 
Unlike the $C^\infty$ case considered \cite{rybicki2},  we study the case of finite regularity and
 obtain the precise threshold $(n+1,\frac12)$. For this purpose
we need to employ a contact scaling transformation $\rho_{A,t}$ that is
different from that of \cite{rybicki2} but similar to the one used by Tsuboi in
\cite{tsuboi3}. (See Section \ref{sec:shifting-supports} and \ref{sec:threshold} for the definitions and
comparison between the two scaling transformations.)

Once we identify the correct scaling transformation, we again apply
the Schauder-Tychonoff fixed point theorem and conclude that $\Cont_c^{(r,\delta)}(M,\alpha)$ 
is a perfect group \emph{as long as $(r,\delta) \neq (n+1, \frac12)$}. Rybicki employed the same strategy in \cite{rybicki2} to prove 
the perfectness of $\Cont_c^\infty(M,\alpha)$. 
Even for this case, our proof clarifies the presentation of  relevant contact geometry and
simplifies the estimates to the optimal level of those given in \cite{rybicki2}. 

 \begin{rem}\label{rem:rybicki-error}
 \begin{enumerate}
\item  Both the homological identity $[g] = [g^{2^{n+2}}]$ 
 and Rybicki's statement 
 \emph{``Observe the above procedure may be repeated for any integer
 $a > 2$ by making use of $\eta_a$ and suitable translations $\tau_{i,t}$. As a result
 there exists $g_a \in \Cont_c(\R^n,\alpha_{st})_0$ such that $\widetilde 
 \Theta^{(n)}(g_a) = f^*$ and $[g^{a^{n+2}}] = [g_a]$. Moreover by (1) we have 
 $[g_a] = [g]$.''}
appearing in the course of the proof \cite[Lemma 8.6 (2)]{rybicki2} are not precise
and contain gaps in its proof.
\item  To make validity of the relevant homological statement hold true,
 one needs to generalize the construction of the operator $\Xi_A^{(k)}$
 given in \cite[p.3313]{rybicki2} associated to the 2-fragmentation to arbitrary $N$-fragmentations
 as given in  Section \ref{sec:unfolding-fragmentation} of the present paper.  
 Probably the author of \cite{rybicki2}  might have had
 this whole process in his mind.  However, this is not even mentioned explicitly,
while this homological identity is  one of the crucial ingredients in his proof.
In the present author's opinion,  the details of this should have been provided in much more
details.
\end{enumerate}
 \end{rem}

Rybicki also suspected that the perfectness may hold at least for large $r$.
Our paper affirmatively confirms this in  the optimal way to the level of precisly locating 
the lower threshold $r = n + 2$ and  the upper threshold $r = n+1$ for contact manifolds of
dimension $2n+1$, similarly as  Mather 
\cite{mather,mather4} did for the diffeomorphism group $\Diff_c(M^n)$ in which case the
corresponding thresholds are $n+1$ and $n$.
Rybicki also asked whether the contact analogs to the 
Thurston-Mather type isomorphism from \cite{mather3}, \cite{tsuboi2,tsuboi3}, \cite{banyaga:book} 
and \cite{rybicki1} can be proved, and regards  as a hard problem.  
We hope that our  systematic study of the background geometry and 
of the optimal estimates will help  making the future researches of such questions easier.

On the other hand, for the case of $1 \leq r < n+\frac32$ in the opposite direction, 
Tsuboi \cite{tsuboi3} previously proved
the simplicity of $\Cont_c(M,\xi)$ by utilizing some construction of infinite repetition 
which has been used in the study of topology of diffeomorphism groups.
(We refer readers to \cite{tsuboi1} for a detailed exposition on such construction with 
many illuminating illustrations.)  
The dual version of the method laid out in the present paper also gives a somewhat
different proof of Tsuboi's result for $1 \leq r < n +\frac32$. We would like to compare 
it with  the way how Mather's proofs for the case with $r > n+2$ \cite{mather} and $r < n+1$ 
\cite{mather2} work.

\begin{rem} In our earlier works on contact instantons,
\cite{oh:contacton-Legendrian-bdy}-\cite{oh:shelukhin-conjecture} and others,
the Greek letter $\psi$ and the associated notation $g_\psi$ as the ones of a contact diffeomorphism and of
its conformal exponent were used respectively. In the present article, we replace them by
the Roman letter $f$ and $\ell_f$respectivley
  to be in more close contact with the literature related to the study of
simplicity problem such as \cite{mather,mather2,mather3}, \cite{epstein:commutators} and \cite{rybicki2}.
We also mention that  the conformal factor $\lambda_f$ will be used
at all in the present paper neither in our constructions related to the Legendrianization nor in 
the definitions or the estimates of the norms of contactomorphisms. Only the conformal exponent
$\ell_f$ will be used in those matters.
\end{rem}

\subsubsection{Towards topological contact dynamics of 
M\"uller-Spaeth} 

Another interesting direction of research is towards the direction of regularity lower than
$C^1$ similarly as in the case of Hamiltonian homeomorphisms (hameomorphisms) as 
done in
\cite{oh:hameo1}, \cite{CHS}. In fact such a study has been carried out by M\"uller and Spaeth
in their series of papers \cite{mueller-spaeth1}--\cite{mueller-spaeth3}. They
in particular introduced the notions of \emph{topological contact automorphisms} 
(\cite[Definition 6.8]{mueller-spaeth1}) and 
and of \emph{contact homeomorphisms} (\cite[Definitioin 6.7]{mueller-spaeth1}). 
They denote them respectively by
\be\label{eq:contact-homeomorphism}
 \Homeo(M,\xi), \quad  \Aut(M,\xi).
\ee
By their definition $ \Homeo(M,\xi)$ is a subgroup of  $ \Aut(M,\xi)$.
The group $\Aut(M,\xi)$ is the analogue of Eliashberg-Gromov's
symplectic homeomorphism group, the $C^0$-closure $\overline{\Symp}(M,\omega)$, 
which  was also denoted by  $\Sympeo(M,\omega)$ in \cite{oh:hameo1}.

We prefer to reserve the notation $\Homeo(M,\xi)$ for the set of elements 
from their $\Aut(M,\xi)$ and to
reserve $\Hameo(M,\xi)$ for the set of those from their $\Homeo(M,\xi)$ and 
call an element therefrom a \emph{contact hameomorphisms} since
in the $C^0$-level, there is  a priori clear difference between the notions of `contact' and 
`contact Hamiltonian' unlike the smooth case. In this vein  we will adopt the notations
$$
\Hameo(M,\xi) \subset \Homeo(M,\xi)
$$
instead of \eqref{eq:contact-homeomorphism} of M\"uller and Spaeth, to emphasize the fact that
the Hamiltonians enter into the definition of the group. One may compare this pair 
in the contact case with the pair of notations
$$
\Hameo(M,\omega) \subset  \Sympeo(M,\omega)
$$
introduced in \cite{oh:hameo1} in the symplectic case.
(In the same vein, one might prefer to replace the notation $\Sympeo(M,\omega)$ by
$\Homeo(M,\omega)$.)

These groups are defined in \cite{mueller-spaeth1} by taking some suitable completions of
$\Cont_c^\infty(M,\alpha)$ similarly as in the case of (symplectic) hameomorphisms \cite{oh:hameo1}.
They also showed that $\Hameo(M,\xi)$  is a normal subgroup of $\Homeo(M,\xi)$.

Two natural questions to ask in this regard are the following:
\begin{ques}\label{ques:question}Let $(M,\xi)$ be a coorientable contact manifold.
\begin{enumerate}
\item Is the inclusion $\Hameo(M,\xi) \subset \Homeo(M,\xi)$ proper?
\item Is $\Homeo(M,\xi)$ a perfect (or a simple) group? How about $\Hameo(M,\xi)$?
\end{enumerate}
\end{ques}
One might go further by specializing to the 3 dimensional case of contact manifold as the 
contact counter part of the area-preserving dynamics on 2 dimensional surface.

Based on the recent developments \cite{CHS} of $C^0$ Hamiltonian 
dynamics in symplectic geometry, the following problem seems to be a very interesting 
doable open problem.

\begin{prob}[Open problem]  
Find the answers to the questions asked in Question \ref{ques:question} for the 3 dimensional 
contact manifold $(M,\xi)$, and see if any of the existing Floer-type analytical
machinery can be used as in the 2 dimensional area-preserving dynamics.
\end{prob}
It is worthwhile to recall the readers that there is one
big difference of the contactomorphism group from the symplectomorphism group for the contact case:
there cannot be any Calabi-type invariant by the simpleness of $\Cont_c(M,\xi)$. This seems to prevent one from 
easily guessing the direction of the answers to the questions, unlike the 
area-preserving case \cite{oh:hameo1,oh:hameo2}.

The organization of the paper is now in order. In Section 2, we set the notations and
various conventions we adopt in contact/symplectic geometry. These are not all the same
as those used in \cite{rybicki2}. We also recall various cubical objects and operators 
defined on $\R^{2n+1}$ appearing in \cite{rybicki2} but none of these cubical objects are
the same as the corresponding ones in \cite{rybicki2} in their numerics although the same notations are 
used. Then in Section 3 - 4 we explain the basic background contact geometry emphasizing
our usages of conformal exponent, contact product, the Legendrianization and of an equivariant
Darboux-Weinstein theorem. After then we provide the geometric part of various  
constructions employed in \cite{rybicki2} with some new definitions, refinements, amplifications and
corrections, and give the proofs of the main theorems as an application of 
Schauder-Tychonoff's theorem, \emph{assuming} the
existence of the map $\vartheta: \CL(\varepsilon,A) \to  \CL(\varepsilon,A)$ 
on some closed convex subset $\CL(\varepsilon,A) \to \CL(\varepsilon,A)$ of
the Banach space $C^r_c(\R^{2n+1},\R)$ of real-valued functions. (See Section \ref{sec:wrap-up}
for the definition of $\vartheta$.)

Then in Part II, we prove all the derivative estimates on the various operators entering in
all relevant constructions. All these estimates have their precedents in \cite{rybicki2}, some of which 
 in turn have their precedents also in \cite{mather} and \cite{epstein:commutators}, but our
optimal versions thereof enable us to define the operator $\vartheta: \CL(\varepsilon,A) \to  \CL(\varepsilon,A)$
for a sufficiently small $\varepsilon$ and sufficiently large $A > 1$, 
\emph{provided} $r + \delta > n+ 2$. Then Part III combines the geometry of Part I and the estimates of
Part II to complete the proofs of all main theorems.
In Appendix A, we derive an implication of the equivariance of the Darboux-Weinstein chart 
in terms of the independence thereof on the first factor of the contact product.
\bigskip

\noindent{\bf Acknowledgement:} We would like to thank Leonid Polterovich for
attracting our attention to Rybicki's work on the simplicity of $C^\infty$ contactomorphism
group in Luminy in the summer of 2023, while the author had been investigating on
the simplicity question of the contact diffeomorphism group. We also thank him for his interest in
 the present work and useful comments on the exposition thereof.
We also thank Sanghyun Kim
for a useful discussion on Mather's thresholds and on the modulus of continuity at the final
stage of finishing the paper which has motivated  the author to add the content of 
Theorem \ref{thm:hoelder-down} and Subsection \ref{subsec:hoelder-down}.

\section{Notations and conventions}

In this section, we gather the notations and conventions that we adopt for the various geometric
constructions, and compare them with those used by Mather \cite{mather} and by Epstein \cite{epstein:commutators}.

\subsection{Notations and conventions for general contact geometry}

We start with the notations for the general contact geometry.
\begin{enumerate}
\item $(z,q,p) = (z,q_1,\ldots, q_n, p_1, \ldots, p_n)$; The canonical coordinates of the 1-jet bundle 
$J^1\R^n \cong \R \times T^*\R^n$ \emph{with the $z$-coordinates written first},
\item $M_Q: =  Q \times Q \times \R$; The contact product of contact manifold $(Q,\alpha)$ equipped with contact form
$\mathscr A = -e^\eta \pi_1^*\alpha + \pi_2^*\alpha$ \cite{oh:shelukhin-conjecture}, 
\item $(x,X,\eta)$: a point in the contact product $Q \times Q \times \R$,
\item  $x = (z,q,p)$ and $X = (Z,Q,P)$ are the `coordinate system' for the
contact product with $Q = J^1\R^n$. (Note that \emph{the $\R$-factor is written at the last spot} unlike the case of $J^1Q$.)
\item $\Phi_U: U \to V, \, \Phi_{U;A}: U_A \to V_A$; Legendrian Darboux-Weinstein charts.
\end{enumerate}

We also set up our conventions for the definitions of Hamiltonian vector fields both in
symplectic and in contact geometry so that they are compatible in some natural sense.
We briefly summarize basic calculus of contact Hamiltonian dynamics to set up our conventions
on their definitions and signs following \cite{oh:contacton-Legendrian-bdy}. 

\begin{defn} Let $\alpha$ be a contact form of $(M,\xi)$.
The associated function $H$ defined by
\be\label{eq:contact-Hamiltonian}
H = - \alpha(X)
\ee
is called the \emph{($\alpha$-)contact Hamiltonian} of $X$. We also call $X$ the
\emph{($\alpha$-)contact Hamiltonian vector field} associated to $H$.
\end{defn}
We alert readers that under our sign convention under which the Reeb vector field $R_\lambda$
as a contact vector field becomes the constant function $H = -1$.
We denote by $R_\alpha$ the Reeb vector field of $\alpha$.

Here are more general conventions that we use in the present article.
\begin{itemize}
\item
The symplectic form on the cotangent bundle $T^*N$ is given by
\be\label{eq:omega0}
\omega_0 = \sum_{i=1}^n dq_i \wedge dp_i = d(-\theta)
\ee
in the canonical coordinates of $T^*N$ associated to a coordinate system $(x_1, \ldots, x_n)$
of $N$, where $\theta = \sum_i p_i\, dq_i$ is the Liouville one-form on $T^*N$.
\item The Hamiltonian vector field associated a real-valued function $H$ on a symplectic
manifold $(M,\omega)$ is given by the equation
\be\label{eq:symp-XH}
dH = X_H \rfloor \omega.
\ee
\item The standard contact form on the 1-jet bundle $J^1N$ is given by
\be\label{eq:alpha0}
\alpha_0 = dz - \sum_{i=1} p_i \, dq_i
\ee
in the canonical coordinates $(z, q_1, \ldots, q_n,p_1, \ldots, p_n)$ of $J^1N$.
\item The contact Hamiltonian vector field $X = X_H$ associated to a real-valued function $H$
on general contact manifold $(M,\alpha)$ is uniquely determined by the equation
\be\label{eq:cont-XH}
\begin{cases}
X \rfloor \alpha = -H, \\
X \rfloor d\alpha  = dH - R_\alpha[H] \alpha.
\end{cases}
\ee
\end{itemize}
With these conventions, the contact Hamiltonian vector field $X_H$ has the decomposition
$$
X_H = X_H^\parallel \oplus (-H \, R_\alpha) \in \xi \oplus \R\langle R_\alpha \rangle.
$$
In the canonical coordinates $(q,p,z)$ on $\R^{2n+1}$, this expression is reduced to
the well-known coordinate formula for the contact Hamiltonian vector field below.
 (See \cite[Appendix 4]{arnold}, \cite[Lemma 2.1]{oh-wang3}, \cite[Lemma 4.1]{bhupal}, 
 \cite[Equation (2.3)]{rybicki2}, for example but with different sign conventions.)

\begin{exm} 
Let $H: \R^{2n+1} \to \R$ be a smooth function on $\R^{2n+1}$.
Then the contact Hamiltonian vector field $X = X_H$ is given by
\bea\label{eq:contact-XH}
X_H & =&  X_H^\parallel -  H R_\alpha \\
& = &  \sum_{i=1}^n \left(\frac{\del H}{\del p_i} \frac{D}{\del q_i} - \frac{D H}{\del q_i} \frac{\del}{\del p_i}\right) - H \frac{\del}{\del z} \nonumber \\
& = & \sum_{i=1}^n \frac{\del H}{\del p_i} \frac{\del}{\del q_i} -
\left(\frac{\del H}{\del q_i} + p_i \frac{\del H}{\del z}\right)\frac{\del}{\del p_i}
+ \left(\left\langle p, \frac{\del H}{\del p}\right \rangle  - H\right)\frac{\del}{\del z}
\eea
\end{exm} 
 
\subsection{Notations of Rybicki in \cite{rybicki2} and their variations}

To make it easier for readers to compare various 
constructions appearing in the present article with those from 
\cite{rybicki2},  we mostly adopt the notations of various objects appearing in Rybicki's constructions
in \cite{rybicki2}. We  follow the same notations of Rybicki as closely as possible 
but with many changes of numerical constants and orders of power 
which are necessary to be able to obtain that optimal estimates 
that give rise to the threshold $r = n+ 2$ mentioned in Warning \ref{warn:notation}.
\begin{enumerate}
\item $T^k = (S^1)^k$; The $k$-torus,
\item {[Contact cylinders]} $\CW_k^m = S^1 \times T^*(T^{k-1} \times \R^{m-k})$ for $m = 2n+1$ and
$1 \leq k\leq n$; the \emph{circular contactization} of 
symplectic manifold $T^*(T^k \times \R^{m-k})$,
equipped with the canonical (partially circular) coordinates $(\xi_0, \xi, p) = (\xi_0, \xi_1, \ldots, \xi_n, p_1, \cdots, p_n)$. 
Throughout the paper, we will use either $m$ or $2n+1$ interchangeably 
as we feel more proper to use.
\item $\alpha_0 = d\xi_0 - \sum_{i=1} p_i \, d\xi_i$; The \emph{contactization contact form} $\alpha_0= d\xi_0 - \sum_{i=1} p_i \, d\xi_i$ on $\CW^m$,
\item For $A\geq 1$, $k = 0,\ldots,n$, we define the \emph{reference rectangulapid}
$$
I_A = [-2,2]\times [-2,2]^n \times [-2A,2A]^n.
$$
\item We define
\beastar
J_A^{(k)}& = & S^1 \times T^{k-1} \times [-2A^2,2A^2]^{n-k} \times [-2A^3,2A^3]^n  \\
K_A^{(k)} & = & S^1 \times T^{k-1}  \times [-2,2] \times[-2A^2,2A^2]^{n-k+1} \times [-2A^3,2A^3]^n
\eeastar
for $k = 1, \ldots, n$, and
\beastar 
J_A^{(0)} & = & J_A =   [-3A^5,3A^5] \times [-2A^2,2A^2]^n \times [-2A^3,2A^3]^n\\
K_A^{(0)} & = & K_A = [-2,2] \times [-2A^2,2A^2]^n \times [-2A^3,2A^3]^n.
\eeastar
(Compare these definitions with those appearing in \cite[p.3312]{rybicki2} with the
same notations.)
\item $E_A = [-A,A]^m$, \, $E_A^k =  S^1 \times T^k \times [-A,A]^{m-k}$,
\item {[Subinterval]} A closed subset $E \subset E_A$ of type $E_A$ is called a \emph{subinterval},
\item {[Contact scaling]} $\chi_A$: the map defined by $\chi_A(z,q,p) = (A^2 z, A q, A p)$,
\item {[twisted $p$-translations]} $\sigma_i^t$: the contact cut-off of the map $S_i^t$ defined by 
$S_i^t(z,q,p) = (z + t q_i,q, p + t\frac{t}{\del p_i})$,
\item {[Contact scaling preceded by $p$-translations]}
$$
\rho_{A;{\bf t}} = \chi_A^2  \circ \sigma^{\bf t}, \quad \sigma^{\bf t} 
= \sigma_1^{t_1} \circ \cdots \circ \sigma_n^{t_n}
$$
for ${\bf t} = (t_1, \ldots, t_n)$. 
(This map is different from the one given in \cite[p.3307]{rybicki2}
with the same notation.)
\item {[Mather-Rybicki's rolling-up operators]}
$$
\Theta_A^{(k)}: \Cont_{J_A^{(k+1)}}(\CW_k^{2n+1}, \alpha_0)_0 \cap \CU_1 \to 
\Cont_{K_A^{(k)}}(\CW_{k+1}^{2n+1},\alpha_0)_0,
$$
\item {[Mather-Rybicki's unfolding-fragmentation operators]}
$$
\Xi_{A;N}^{(k)}: \Cont_{J_A^{(k+1)}}(\CW_{k+1}^{2n+1}, \alpha_0)_0 \cap \CU_1 \to 
\Cont_{J_A^{(k)}}(\CW_k^{2n+1},\alpha_0)_0,
$$
\item {[Rybicki's rolling-up operators]}
$$
\Psi_A^{(k)}: \Cont_{J_A^{(k)}}(\CW_k^{2n+1}, \alpha_0)_0 \cap \CU_1 \to 
\Cont_{J_A^{(k+1)}}(\CW_k^{2n+1},\alpha_0)_0
$$
where the rolling occurs in the $q_k$-coordinate direction for $k = 1,\ldots, n$ and for $z$ for $k=0$.
\end{enumerate}

We now organize the domains and the codomains entering into the definitions of 
these operators in the following diagram: For the study of the lower threshold, we consider
the diagram
$$
\xymatrix{
&&[-2,2]^{2n+1} \ar[dll]_{p_0 \rho_{A,\sigma_{t_0}}} \ar[dl]^{p_1\rho_{A,\sigma_{t_1}} }
\ar[d]^{p_*\rho_{A,\sigma_{t_*}}}\ar[dr]^{p_n \rho_{A,\sigma_{t_n}}}&\\
J_A^{(0)} \ar[d]_{\iota_0}  \ar[r] & J_A^{(1)} \ar[d]_{\iota_1}  \ar[r] & \cdots \cdots 
\ar[d]_{\iota_{*}}\ar[r] &J_A^{(n)} \ar[d]_{\iota_n} 
\\
 \CW_0^{2n+1} \ar[r]_{\pi_0} & \CW_1^{2n+1} \ar[r]_{\pi_1} & \cdots \cdots \ar[r]_{\pi_n} &\CW_n^{2n+1}.
}
$$
Here $\pi_k$ are the covering projections induced by $\R^{k+1} \to T^{k+1}$
for $k = 0, \ldots, n$, the map induced by the identity map of $T^{k+1}$ and the inclusion map
$$
  [-2A^2,2A^2]^{n-k} \times [-2A^3,2A^3]^n  \hookrightarrow \R^{2n-k}.
 $$
 The map $p_i: \R^{2n+1} \to T^{i+1} \times \R^{2n-i}$ are the obvious covering projections. 
 Finally, the maps $\rho_{A,t_i}: \R^{2n+1} \to \R^{2n+1}$ are the translation map followed by
 the contact rescaling map given by
 $$
 \rho_{A,t_i}: = \chi_A^2 \circ \sigma_i^{t_i}.
 $$
 (See \eqref{eq:rhoAt} for the precise definition thereof.) Compare this sequence
 $$
\R^{2n+1} \supset J_A^{(0)} \to J_A^{(1)} \cdots \to J_A^{(n)} \subset J^1\R^n
 $$ 
 with the sequence 
 $$
 [-2,2]^n = D_n \subset D_{n-1} \subset \cdots \subset D_0 = [-2A,2A]^n
 $$
 from \cite[Section 3]{mather}, \cite[Section 2]{epstein:commutators}, and
 observe that \emph{both sequences have $n$ terms in them}, which is
 essential for the determination of the lower threshold $r = n+2$ for the contact case and 
 $r =n+1$ for the diffeomorphism case, respectively here and in \cite{mather,mather2}.
 
 Similarly as in \cite{mather2}, we reverse the horizontal arrows for the study of the
 upper threshold.
 
\section{Conformal exponents, contact product and Legendrianization}
\label{sec:contactomorphisms}

Now let   $(M,\xi)$ be a contact manifold of dimension $m = 2n+1$, which is coorientable.
We denote by $\Cont_+(M,\xi)$ the set of orientation preserving contactomorphisms
and by $\Cont_0(M,\xi)$ its identity component.  
Equip $M$ with a contact form $\alpha$ with $\ker \alpha = \xi$.

\subsection{Conformal exponents of contactomorphisms}

For any coorientation-preserving contactomorphism $g$ we have
$$
g^*\alpha = \lambda_g^\alpha\,  \alpha.
$$
We adopt the convention of systematically
calling the function $\lambda_g^\alpha$ the conformal factor, and 
\be\label{eq:exponent}
\ell_g^\alpha : = \log(\lambda_g^\alpha)
\ee
the \emph{conformal exponent} of $g$, following the practice of 
\cite{oh:contacton-Legendrian-bdy,oh:entanglement1,oh:shelukhin-conjecture}. 
Since the contact form $\alpha$ will be fixed
throughout the paper, we omit $\alpha$ from notations by writing $\lambda_f$
and $\ell_f$ respectively from now on.
The following lemma is well-known and straightforward to check.
\begin{lem}\label{lem:ell-composition} For any two contactomorphisms $g, \, f$, we have
\be\label{eq:ell-gf}
\ell_{gf}^\alpha = \ell_g^\alpha \circ f+ \ell_f^\alpha.
\ee
\end{lem}
For each contact form $\alpha$ of the given contact manifold
$(M,\xi)$, we consider the function $\varphi_\alpha: G:= \Cont(M,\xi) \to C^\infty(M)$,
defined by
\be\label{eq:varphi}
\varphi_\alpha(f) = \ell_f^\alpha.
\ee
 Regard it as a one-cochain in the group cohomology complex
$(C^1(G, C^\infty(M)), \delta)$ with the coboundary map $\delta: C^1(G, C^\infty(M)) \to C^2(G, C^\infty(M))$
on the right $\Z[G]$-module $C^1(G, C^\infty(M))$ with the right composition
$$
(f, \ell) \mapsto \ell \circ f
$$
as the action of $G$ on $C^1(G, C^\infty(M))$. 
Then the above lemma can be interpreted as
\be\label{eq:gphipsi}
\ell_{gf} = (\delta \varphi_\alpha)(g,f).
\ee
Furthermore for a different choice of contact form $\alpha'$ of the same contact structure $\xi$
in the same orientation class, we make $\alpha$-dependence on $\ell_f$ explicit by writing $\ell_f^\alpha$
and $\ell_f^{\alpha'}$. Furthermore we have 
\be\label{eq:halphaalpha'}
\alpha' = e^{h_{(\alpha'\alpha)}}\alpha
\ee
for some smooth function $h_{(\alpha'\alpha)}$ depending on $\alpha, \, \alpha'$. Then we have
\begin{prop} Consider the two zero-cochains $\varphi_{\alpha'},  \, \varphi_{\alpha}$
in the group cohomology complex $C^*(G, C^\infty(M))$. Then $\varphi_{\alpha'},  \, \varphi_{\alpha}$
are cohomologus to each other.
\end{prop}
\begin{proof} A straightforward calculation leads to 
\be\label{eq:cohomologous}
\ell_f^{\alpha'} = h_{\alpha'\alpha}\circ f + \ell_f^\alpha. 
\ee
This itself can be written as 
$$
\varphi_{\alpha'} - \varphi_{\alpha} = \delta h_{(\alpha'\alpha)}
$$ 
where we regard $h_{(\alpha'\alpha)}$ as a zero-cochain. This finishes the proof.
\end{proof}

An iteration of \eqref{eq:ell-gf} gives rise to the following suggestive form of the identity
\be\label{eq:gphin}
\ell_{g_m\circ \cdots \circ g_1} = \sum_{k=0}^{m-1} \ell_{g_k} \circ (g_{k-1} \circ \cdots \circ g_1)
\ee
 for all $g_i \in \Cont(M,\xi)$ with $i = 1, \cdots, m$.
 
\subsection{Defintion of $C^{r,\delta}$ contactomorphisms}
 
Now, we give the precise definition of $C^r$ (resp. $C^{r,\beta}$) contactomorphisms.

\begin{defn}[$C^r$ contact diffeomorphism]\label{defn:Cr-topology}
 A $C^r$ diffeomorphism $f: M \to M$ is called
a $C^r$ contactomorphism if $f^*\alpha = \lambda_f\, \alpha$ for some positive function $\lambda_f$,  
or equivalently $\ell_F$, is a positive $C^r$ function. 
\end{defn} 

This definition can be extended to the H\"older regularity classes.
More precisely,  let $\beta$ be a modulus of continuity of the form $\beta:[0,\infty) \to [0,1]$ and
Denote by $\Diff^{(r, \beta)}(M,\alpha)$ be the set of $C^{r, \beta}$ diffeomorphisms.

We specialize to the H\"older regularity $\beta(x) = x^\delta$ and
define the set of $C^{r, \delta}$ contactomorphism as follows.
\begin{defn}[{$\Cont^{(r, \delta)}(M,\alpha)$}]
We define the set of $C^{r, \delta}$ contactomorphisms to be the intersection
\be\label{eq:cont-rbeta}
\Cont^{(r, \delta)}(M,\alpha) : =  \Cont^r(M,\alpha) \cap \Diff^{(r, \delta)}(M) \subset \Diff^{(r, \delta)}(M).
\ee
\end{defn}

An immediate corollary of \eqref{eq:ell-gf} is the following.

\begin{cor}\label{cor:composition} For any $1 \leq r \leq \infty$, the set $\Cont_c^r(M,\alpha)$
(resp. $\Cont_c^{(r,\delta)}(M,\alpha)$)
is a topological subgroup of $\Diff_c^r(M)$.
\end{cor}
\begin{proof} We have only to prove that for any contact diffeomorphisms $f, \, g$
such that they are of class $C^r$ as well as $\ell_f$, $\ell_g$ are $C^r$, 
$\ell_{gf}$ are $C^r$. But this is apparent by the formula \eqref{eq:ell-gf}.

For the case of $\Cont_c^{(r,\delta)}(M,\alpha)$, it is well-known \cite[Section 2]{mather} that
$\Diff^{(r, \delta)}(M)$ is a subgroup of $\Diff^r(M)$, which finishes the proof.
\end{proof}

It is straightforward to derive from the identity \eqref{eq:halphaalpha'} that the definition of
$\Cont^{(r,\delta)}(M,\alpha)$ and its topology do not depend on the choice of the contact form 
$\alpha$. 

\begin{rem} However, we alert the readers that the product operation is not
closed in the intersection
$$
\Cont_c^{r+1}(M,\alpha) \cap \Diff_c^r(M),
$$
according to the definition given in Definition \ref{defn:Contrc},
and hence the intersection does not form a subgroup of $\Diff_c^r(M)$ because 
the conformal exponent of an element may not be $C^r$. Therefore
$\Cont_c^r(M,\xi)$ is \emph{not}
a closed subgroup of $\Diff_c^r(M)$, because a priori one loses a regularity by 1
to define the conformal exponent. As a result, the $C^r$-convergence of 
contactomorphisms $f_i$ does not guarantee convergence $\ell_{f_i}$.
The following inclusion is a proper inclusion
$$
\overline{\Cont_c^{r+1}(M,\xi)} \subset \Cont_c^r(M,\xi)
$$
where the closure is the one of $\Cont_c^{r+1}(M,\xi) \subset \Diff_c^r(M)$ taken
in $\Diff_c^r(M)$. The case $r < 1$ is particularly interesting which may be regarded as
a contact counterpart of the group $\Hameo(M,\omega)$
of \emph{hameomorphisms} in symplectic geometry \cite{oh:hameo1}. (It was conjectured in
\cite{oh:hameo1} that $\Hameo(M,\omega)$ is a \emph{proper} subgroup of the area-preserving homeomorphism
group in the 2 dimensional case such as $M = D^2, \, S^2$, and the conjecture has been recently proved 
by Cristofaro-Gardiner,  Humili\`ere and  Seyfaddini \cite{CHS}.) Some related researches 
in this direction have been carried out by M\"uller and Spaeth in their series of 
works \cite{mueller-spaeth1,mueller-spaeth2,mueller-spaeth3} and others.
\end{rem}

 \subsection{Contact product and Legendrianization}
 
We now consider the product $(M_Q,\Xi)$
\be\label{eq:MQ}
M_Q: = Q \times Q \times \R, \quad \Xi: = \ker \mathscr A
\ee
with contact distribution $\Xi = \ker \mathscr A$ for a specifically chosen contact form
\be\label{eq:legendrianization-form}
\mathscr A: = -e^\eta\, \pi_1^*\alpha + \pi_2^*\alpha.
\ee
Here we follow the sign convention of \cite{oh:shelukhin-conjecture}.
We then consider the operation of \emph{Legendrianization} of contactomorphisms,
which is the contact analog to the \emph{Lagrangianization} of
canonically associating to each symplectomorphism the Lagrangian submanifold
in the product which is nothing but its graph. 

We summarize basic properties of this contact product  and the Legendrianization,
some of which are  well-known, e.g.,  can be found from  \cite{lychagin,bhupal}. 

\begin{prop} \label{prop:contactgraph} Let $(Q,\xi)$ be any contact manifold and
a contact form $\alpha$ be given. We define $(M_Q,\mathscr A)$ as above.
Denote by $\pi_i: M_Q \to Q$ the projection to the $i$-th factor of the product for $i=1,2$,
and $\eta: M_Q \to \R$ the projection to $\R$.  Then the following hold:
\begin{enumerate}
\item The fibers of the projection maps $(\pi_i,\eta)$, $i = 1,\, 2$ are Legendrian in $M_Q$.
\item The Reeb vector field $R_{\mathscr A}$ is given by $(0, R_\alpha, 0)$.
\item For any contactomorphism $g$ of $(Q,\xi)$ with $g^*\alpha = e^{\ell_g} \alpha$, the map
\be\label{eq:tildepsi}
j_g(y) = (y,g(y), \ell_g(y))
\ee
is a Legendrian embedding of $Q$ into $(M_Q, \mathscr A)$.
\end{enumerate}
We call the image of the map $j_g$  the \emph{Legendrianization} of $g$ 
\cite{oh:shelukhin-conjecture}.
\end{prop}
One can utilize this to give a local parametrization of a neighborhood of the
identity in $\Cont(M,\xi)$, which also shows local contractibility of the $\Cont_c^r(M,\xi)$.
(See \cite{lychagin}, \cite{banyaga:book}.)

We will  apply this construction to the case 
when $Q = J^1\R^n \cong \R^{2n+1}$ equipped with the standard
contact form $\alpha_0= dz - p dq$ in the canonical coordinates 
$$
(z, q_1, \cdots, q_n, p_1, \cdots, p_n).
$$
In terms of the coordinate system $(x,X,\eta)$, which we call the contact product
coordinate system,  of the contact product
$$
M_{\R^{2n+1}} = \R^{2n+1} \times \R^{2n+1} \times \R
$$
with $x = (z,q,p)$ and $X = (Z,Q,P)$, we have
$$
\pi_1^*\alpha_0 = dz - p\, dq, \quad \pi_2^*\alpha_0 = dZ - P\, dQ.
$$
Then  the form $\mathscr A$ is given by
\be\label{eq:A-R2n+1}
\mathscr A_{\R^{2n+1}} = -e^\eta (dz - p\, dq) + (dZ - P\, dQ)
\ee
in the given coordinate system. We highlight the location of the $\R$-coordinate $\eta$
that is written \emph{at the last spot}, while the $\R$-coordinate in the 1-jet coordinate system
of $J^1\R^{2n+1}$ is written \emph{at the first spot.} \emph{This practice will be consistently
used throughout the present paper.}

\begin{rem}[Locations of the $\R$ coordinates]
 Here is the reason why we exercise the aforementioned practice. There are quite a 
few constructions  carried out on 
$$
\R^{2(2n+1)+1} \cong J^1\R^{2n+1} \cong M_{\R^{2n+1}}.
$$
Some of them are more natural to consider in $J^1\R^{2n+1}$ but others are in $M_{\R^{2n+1}}$.
Some of the constructions may be the most natural 
even in $\R^{2(2n+1)+1}$. We would like to make it clear
on which space we perform
the constructions by distinguishing the ways of representing points in $M_{\R^{2n+1}}$
and in $J^1\R^{2n+1}$ at least by putting the $\R$ coordinates differently.
Luckily, the articles \cite{bhupal} and the present author's preprints \cite{oh:entanglement1}, \cite{oh:shelukhin-conjecture}
have already used the same convention putting the $\R$-coordinate in the last spot as $Q \times Q \times \R = M_Q$
which makes the current practice consistent with them. Furthermore we take the same practice 
as \cite{rybicki2} put the $\R$ coordinate in the 1-jet bundle $J^1\R^{2n+1}$ in the first spot so that comparing the notations and the details from  \cite{rybicki2} and from the current paper is hoped to
become easier.
\end{rem}

 \section{Basic contact vector fields and contactomorphisms of $\R^{2n+1}$}
 
 We take the (global)  frame
 $$
 \left\{\frac{\del}{\del z}, \frac{D}{\del q_i},  \, \frac{\del} {\del p_i}\right\}, 
 $$
 on $\R^{2n+1} = J^1 \R^n$ which is equipped with the standard contact form
 $$
 \alpha_0 = dz - \sum_{i=1}^n p_i \, dq_i.
 $$
 Here we write 
 $$
 \frac{\del}{\del q_i} + q_i \frac{\del}{\del z}  = : \frac{D}{\del q_i}
 $$
 following the notation from \cite{LOTV}.
 We mention that $\{\frac{D}{\del q_i},  \, \frac{\del} {\del p_i}\}$ is a Darboux frame
of the contact distribution $\xi$. More specifically, they are tangent to the distribution $\xi$ and
satisfies
$$
d\alpha\left(\frac{D}{\del q_i}, \frac{\del}{\del p_j}\right) = \delta_{ij}.
$$
 Note that the collection $\{ \frac{\del}{\del q_i} + q_i \frac{\del}{\del z},  \, \frac{\del} {\del p_i}\}$
 provides a global Darboux frame of the contact distribution $\xi$ of $J^1 \R^n$. 
 We mention that except the Reeb vector field $\frac{\del}{\del z}$ (whose Hamiltonian is the constant function $-1$)
 none of these vector fields are contact.
 
 \subsection{Basic contact Hamiltonian vector fields}
 
 Now we consider the flows of various basic Hamiltonian vector fields: 
 \begin{enumerate}
 \item Consider the Hamiltonian $H = - q_i$ whose Hamiltonian vector field is given by
 $$
 X_{- q_i} = \frac{\del}{\del p_i} + q_i \frac{\del}{\del z} = \vec f_i + q_i \vec e_z.
 $$
It generates the translation flow in the $p_i$-direction
$$
\psi_{-q_i}^t(z,q,p) = (z + tq_i,q_1, \cdots, q_n, p_1, \cdots, p_{i-1}, p_i + t, p_{i+1}, \cdots, p_n) =:T_i^t.
$$
In particular, we write 
$$
T_i := \psi_{-q_i}^1
$$ 
the translation by $1$ in the $p_i$-direction.  
\item Next, we  consider $H = p_i$ and its Hamiltonian vector field
$$
X_{p_i}  = \frac{\del}{\del q_i}  = \vec e_i 
$$ 
whose flow is given by 
$$
\psi_{p_i}^t = (z, q_1, \cdots, q_i + t, \cdots, q_n, p_1, \cdots, p_n) =: S_i^t
$$
In particular, we have
$$
\psi_{p_i}^1(z,q,p) = (z, q + \vec e_i,p) =: S_i.
$$
\item 
We also consider the \emph{contact Euler vector field} $E^c$ defined by
\be\label{eq:Ec}
E^c = 2z \frac{\del}{\del z} + \sum_{i=1}^n \left(q_i \frac{\del}{\del q_i} + p_i \frac{\del}{\del p_i}\right)
\ee
which is a contact vector field associated to the Hamiltonian
$
H^c : = 2z - \sum_{i=1}^n q_i p_i,
$
and generates the flows the \emph{contact rescaling} $\eta_t$ given by
\be\label{eq:etat} 
\psi_{H^c} ^t(z, q,p) =  (e^{2t} z,\, e^t \, q,  e^t\, p).
\ee 
\item  The \emph{front Euler vector field}
\be\label{eq:Ef}
E^f = \sum_{i=1}^n q_i \frac{\del}{\del q_i} + z \frac{\del}{\del z}
\ee
is associated to  the Hamiltonian $H^f : = z - \sum_{i=1}^n q_i p_i$ and
generate and the \emph{front rescaling} $\chi_t$ given by
\be\label{eq:chit}
\psi_{H^f}(q,p,t) = (e^t,\, e^t q, p).
\ee
\end{enumerate}
We denote the time-one maps of $E^c$ and  $E^f$ by $R^c$ and $R^f$ respectively.

\begin{rem}\label{rem:rybicki2}
We would like to warn the readers that our definition is different from that of \cite{rybicki2} for the signs
in that the Hamiltonian associated to a contact vector field $X$ is given by $H: = -\alpha(X)$, while
\cite{rybicki2} adopts its negative $H = \alpha(X)$. 
\end{rem}

\subsection{Contact cut-off of basic Hamiltonian vector fields}
\label{subsec:contact-cutoff}

In summary, we will consider the following collection of contact diffeomorphisms
\be\label{eq:TSRcRf}
\{T_i, \, S_i\}_{i=1}^n \cup \{R^c, \, R^f\}
\ee
throughout the paper, which we call the \emph{basic contact transformations} 
on $\R^{2n+1}$. 
Obviously, none of these contactomorphisms are compactly supported on
$\R^{2n+1}$ while our concern lies in the compactly supported contactomorphisms.
However all of them are generated by contact Hamiltonian vector fields and so we can 
cut down their supports by multiplying a bump function \emph{to the associated 
Hamiltonian functions} as done in \cite{banyaga} for the symplectic case and in
\cite{rybicki2} for the contact case, when there is given a compact subset $K \subset 
\R^{2n+1}$.  

It deserves some explanations to explain this contact cut-off of 
basic Hamiltonian vector fields and their induced contactomorphisms,
which we feel is somewhat counter intuitive.
We will focus on the cut-off of the case of transformation $T_j = \psi_{p_i}^1$
which is generated by the Hamiltonian vector field
$$
 \frac{\del}{\del p_i} + q_i \frac{\del}{\del z}
$$
on $\R^{2n+1} \cong J^1\R^n$. It generates the flow
$$
T^t (z,q,p) = (z + q_i t,q, p + t \vec f_i).
$$
If we multiply a cut-off function of the form 
$$
\kappa =  \kappa_0 \times \left(\prod_{i=1}^n \kappa_i\right)  \times \left(
\prod_{j=1}^n   \kappa_{n+j} \right)
$$
supported in the interior of a compact set 
$$
E = \{(z,q,p) \mid |z| \leq 2a_0, \, |q_i| \leq 2a_i, \, |p_j| \leq 2b_j\}
$$
that also satisfy $\kappa_i \equiv 1$ on $[-a_i,a_i]$ and $\kappa_{n+j} \equiv b_j$ on
$[-b_j,b_j]$ for some positive numbers $a_i , \, b_j > 0$ for $i = 0, \ldots, n$ and $j = 1, \ldots, n$

We write $\rho_i : = \kappa/\kappa_i$ which does not depend on $q_i$ by definition. 
Then we compute
the Hamiltonian vector field of $H =  q_j \cdot \kappa  =  (q_j  \chi_j) \cdot \rho_j$. We substitute 
$H$ into \eqref{eq:contact-XH} and obtain
$$
X_H = - \left((q_i \kappa_i)' \frac{\del}{\del p_i} + q_i(q_i \kappa_i )' \frac{\del}{\del z}\right)
+ X'
$$
where $X'$ is the vector field that does not contain components of $\frac{\del}{\del q_i}$ 
and $\frac{\del}{\del z}$ which is still supported in $E$. 
We now examine the components in the two directions of $q_i$ and $z$, which is
$$
(q_i \kappa_i)' \frac{\del}{\del p_i} + q_i (q_i)' \frac{\del}{\del z} 
= (\kappa_i + q_i ') \frac{\del}{\del p_i} + q_i (\kappa_i + q_i \kappa_i' 
- q_i \kappa_i \rho_i) \frac{\del}{\del z}.
$$
Therefore the $p_i$ component of its flow is given by
$$
t \mapsto p_i + t ( (\kappa_i (q_i)+ q_i \kappa_i'(q_i))  =: p_i(t)
$$
and the $z$ component is given by
$$
t \mapsto z + q_i (\kappa_i + q_i \kappa_i' - q_i \kappa_i \rho_i) t =: z(t).
$$
The upshot of the above calculation is to show that \emph{the coordinate function
$-q_i$ does not generate a pure $p_i$-translation but only the one coupled with
a $z$-translation, while $p_i$ generates the pure $q_i$-translation.}

\begin{lem}\label{eq:cutoff-support} The cut-off flow $\tau^t$ maps $E$ into $E$ and
is a translation with constant speed $q_i$ in the direction of $p_i$ on $(-b_i,b_i)$.
\end{lem}
\begin{proof} The first statement immediately follows from the observation 
that the vector field $X_{p_i }$  vanishes outside $E$, and the second follows from
the property $ \equiv 1$ on $\frac12 E \subset E$.
\end{proof}

These preparations on the cut-off being made and mentioned,
 we will omit the process throughout the paper, without further mentioning of converting 
 \eqref{eq:TSRcRf} to the relevant compactly supported counterparts see 
 \cite[Construction in p. 519]{mather} for the precise description of 
the process in the case of $\Diff_c^r(M)$):
We denote the compactly supported contactomorphisms obtained this way by
\be\label{eq:tausigmachieta}
\{\tau_i, \sigma_i\} \cup \{\chi, \, \eta\}.
\ee
We will also consistently utilize the following notations for the corresponding collection
\be\label{eq:logA}
\chi_A : = \psi_{H^c}^{\log A}, \quad \eta_A: = \psi_{H^f}^{\log A}, \quad A > 1
\ee
respectively, where we also denote by $\chi_A$, $\eta_A$ the associated
cut-off version thereof by an abuse of notations.

\part{Mather-Rybicki's constructions for $\Cont_c(\R^{2n+1},\alpha_0)_0$}
  
\section{An equivariant Darboux-Weinstein chart}
\label{sec:Darboux-Weinstein}

An immediate corollary of the standard Darboux-Weinstein theorem
 is that any $C^1$-small perturbation $R'$ of 
the given Legendrian submanifold $R$ can be uniquely written as
$$
R' = \Phi_U^{-1}(j^1u(R))
$$
for some smooth function $u: R \to \R$ where $j^1u$ is the 1-jet graph of $f$ 
via the chart $\Phi_U: U \to V$.
Conversely for any $C^1$-small (resp. $C^k$-small) smooth function $u$, its 1-jet
graph corresponds to is a $C^0$-small (resp. $C^{k-1}$-small) smooth perturbation of $R$.
However we want some additional equivariant property  for the chart 
with respect to some actions by the group
$$
G = \R^{n+1} \quad \text{or} \quad T^k \times \R^{n+1-k}.
$$
with respect to some actions  for our purpose of the present paper. 
This leads us to involve a variation of  the \emph{contact product} \cite{lychagin,
banyaga:book,bhupal,oh:shelukhin-conjecture,rybicki2}, when we apply it to the 
standard `linear' contact manifold $\R^{2n+1}$.

\subsection{Coordinate transformations on $\R^{2(2n+1)+1}$}

Following Rybicki \cite{rybicki2}, we consider the contact cylinders denoted by
$$
\CW_k^{2n+1} = T^k \times \R^{2n+1-k} \cong (T^k \times \R^{n-k}) \times \R^n.
$$ 
We first consider its 1-jet bundle
$$
J^1\CW_k^{2n+1} \cong \R \times T^*\CW_k^{2n+1} \cong \R \times \CW_k^{2n+1} 
\times \R^{2n+1}, \quad \alpha_0 = d\xi_0 -\sum_{i=1}^{2n+1} p_i d\xi_i.
$$
(For the simplicity of notations, we sometime write $2n+1 =: m$.)
We write the associated canonical coordinates as $(t, x, {\bf p}_x)$ 
with $x = (z, q, p)$ where  ${\bf p}_x = (z_x,q_x,p_x)$ is the conjugate coordinates of $x$.
Here we regard $q_i$ 
as $S^1$-valued for $i = 1,\ldots, k$ and real-valued for $i = k+1, \ldots, m$
when we consider $\CW_k^m$ instead of $\R^{2n+1}$.
Consider the Legendrian submanifold
\be\label{eq:R}
R : = \CZ_{\CW_k^m} \cong \{0\}_{\R} \times \CW_k^m \times \{0\} _{(\R^n)^*}
 \subset J^1 \CW_k^m \cong \R \times T^*\CW_k^m,
\ee
the zero section of $J^1\CW_k^m$. 
By the definition of  $C^r$-topology of $\Cont(M,\alpha)$ (Definition \ref{defn:Cr-topology}) in general, 
it follows that any  contactomorphism of $\CW_k^m$ $C^1$-close the identity can be uniquely
lifted to a contactomorphism $\widetilde f$ $C^1$-close to the identity on $\R^{2n+1}$:
$$
\xymatrix{\R^{2n+1} \ar[d]^\pi \ar[r]^{\widetilde f} & \R^{2n+1} \ar[d]^\pi \\
\CW_k^m \ar[r]^f & \CW_k^m.
}
$$

\begin{rem}[Notational abuse]\label{rem:notation-abuse}
\begin{enumerate}
\item  We can uniquely lift  the chart $\Phi_U$ to  $\widetilde \Phi_U: \widetilde U \to \widetilde V$
under the covering projection $\R^{2n+1} \to \CW_k^{2n+1}$ so that the following
diagram commutes:
$$
\xymatrix{\widetilde U \ar[d]_\pi \ar[r]^{\widetilde \Phi_U} & \widetilde V\ar[d]_\pi \\
U \ar[r]_{\Phi_U} & V
}
$$ 
where $\widetilde U = \pi^{-1}(U)$ and $\widetilde  V = \pi^{-1}(V)$.
\item 
Since we will be interested in $C^1$-small
diffeomorphisms, we may and will always assume that this unique lifting
is fixed. Once this is being said, we will drop the tilde from
notations and abuse the notation $\Phi_U: U \to V$ also to denote
the latter. Since the covering projections of $\Pi$ are linear isometries with respect to
the standard linear structure and the metric of $\R^{2(2n+1) + 1}$ and 
$\delta$ is a linear map, all the estimates we will perform will be done
on $\R^{2n+1}$ and $\R^{2(2n+1)+1}$. This enables us to freely use the
global coordinates of $\R^{2(2n+1)+1}$. We will
adopt this approach in the rest of the paper, especially when we do the estimates in Part 2 of the present paper,
unless there is a danger of confusion.
(See \cite[p.525]{mather} for the similar practice laid out for the same purpose of study of 
derivative estimates in the similar setting of covering projections $\R^n \to \mathscr C_i$
with $\mathscr C_i \cong S^1 \times \R^{n-1}$.
\end{enumerate}
\end{rem}

\subsection{Contact product of $\CW_k^m$ and equivariant Darboux-Weinstein chart}

Recall the general definition of the \emph{contact product}
 $(M_Q, \mathscr A)$ of a contact manifold $(Q,\alpha)$,
$M_Q := Q \times Q \times \R$
equipped with the contact form
$$
\mathscr A = - e^\eta \pi_1^*\alpha_0 + \pi_2^* \alpha_0.
$$
Then the \emph{contact diagonal} is given by
\be\label{eq:contact-diagonal}
\{(x, x, 0) \in M_Q \mid x \in Q\} = : \Gamma_{\id}.
\ee
By construction, it follows that $\Gamma_{\id}$ is
Legendrian with respect to $\mathscr A$, which is diffeomorphic to $R$ given in \eqref{eq:R}.
More generally, recall that the \emph{ graph}
\be\label{eq:Gammaf}
\Gamma_f: Q \to M_Q; \quad \Gamma_f(x): = (x,f(x),\ell_f(x))
\ee
is a Legendrian embedding of $(M_Q,\mathscr A)$ which is called the
\emph{Legendrianization} of $f$ in \cite{oh:shelukhin-conjecture}. With a slight abuse of
notations, we will denote by $\Gamma_f$ both the map and its image.

For the purpose of applying a Mather-type construction on $Q = \R^{2n+1}$, we would like to parameterize
each contactomorphism $f$ of $\CW_k^m$
\emph{sufficiently $C^1$-close to the identity map} by a \emph{real-valued function} $u_f$ on $\R^{2n+1}$.
For this purpose, we want to associate to the map 
$$
\widetilde f - id: \R^{2n+1} \to \R^{2n+1},
$$
not $\widetilde f$ itself, a Legendrian submanifold in $(J^1\R^{2n+1},\alpha_0)$ so that it becomes the graph of 1-jet map
$j^1u_f: \R^{2n+1} \to J^1\R^{2n+1}$ for some real-valued function $u_f$.
To achieve this goal, we consider the map
$\delta: M_{\R^{2n+1}}  \to M_{\R^{2n+1}}$
defined by
\be\label{eq:deltakm}
\delta(x,X,\eta) = (x,X+x, \eta).
\ee
Obviously its inverse is given by
\be\label{eq:delta-1}
\delta^{-1}(x,X,\eta) = (x,X-x,\eta).
\ee
(See \cite[p. 3300]{rybicki2} for a similar consideration.)
\begin{rem} This operation of taking the difference $\widetilde f - id$ is natural on $M_{\R^{2n+1}}$
while it is very unnatural  as one made on $J^1\R^{2n+1}$. Therefore the operation $\delta$ is
defined as one on $M_{\R^{2n+1}}$. However it is not a contact diffeomorphism of $(M_{\R^{2n+1}}, \mathscr A)$ and hence
the graph of $\widetilde f - \id$ is not Legendrian for contactomorphism $f$ unlike $\Gamma_f$.
\end{rem}
 Here
we use the collective coordinate system  $(x,X,\eta)$ on $\R^{2(2n+1) +1}$ with
$$
x = (z, q, p), \quad X = (Z,Q,P)
$$
as the coordinate system on the contact product $M_{\R^{2n+1}}$
by identifying $R^{2(2n+1) +1}$ with $\R^{2n+1} \times \R^{2n+1} \times \R$.

\begin{rem} We emphasize the fact that this kind of linear contactomorphism
is a construction that exists only for the linear contact manifold $\R^{2n+1}$, not for
for general contact manifold $Q$, while the contact product is a general
functorial construction applied to an arbitrary contact manifold $Q$.
\end{rem}

We pull-back $\mathscr A$ by the diffeomorphism $\delta$ and set a new contact form
\be\label{eq:hatA}
\widehat{\mathscr A}|_{(t,x,X)} := \delta^*\mathscr A|_{(t,x,X)} =-e^t (dz - pdq) 
+d(z+Z)  - \sum_{i=1}^n (p + P) d(q + Q)
\ee
on $M_{\R^{2n+1}}$. We also consider the tautological map
\be\label{eq:Pi}
\Pi: J^1\R^{2n+1} \to M_{\R^{2n+1}} = \R^{2n+1} \times \R^{2n+1}\times \R
\ee
defined by $\Pi(t,x,X) = (x,X,t)$ where we identify $X$ with the conjugate $p_x$ of $x$
on the domain of the map $\Pi$, and have the diagram
\be\label{eq:diagram}
\xymatrix{
M_{\R^{2n+1}} \ar[r]^{\delta} & M_{\R^{2n+1}} \ar[d]_{\text{\rm pr}_2}\\
J^1 \R^{2n+1}\ar[u]_\Pi \ar[r]_\pi & \R^{2n+1}.
}
\ee
Obviously the map
$$
\Delta_f: = (\delta\Pi)^{-1} \circ \Gamma_f: \R^{2n+1} \to J^1 \R^{2n+1}
$$
 is not Legendrian with respect to $\alpha_0$ but becomes \emph{Legendrian} 
 with respect to the contact form $\widehat{\mathscr A}$ for all contactomorphism $f$ of $\R^{2n+1}$.
 Similarly as we did for $\Gamma_f$, we denote by $\Delta_f$ both the map and its 
 image.

The following is by now obvious and reflects the `linear structure', which plays a significant role in Mather's construction,
on the  contact manifold $\R^{2n+1}$, and is utilized in 
Rybicki's contact version \cite{rybicki2} of the rolling-up and homothetic transformations
of Mather \cite{mather}.

\begin{lem} \label{lem:Gammaf} Let $f \in \Cont(\R^{2n+1},\alpha_0)$ and consider the 
graph of the map $\Delta_f: \R^{2n+1} \to  J^1 \R^{2n+1}$ given by
\be\label{eq:Deltaf}
\Delta_f: = (\delta\Pi)^{-1} \circ \Gamma_f = \left\{\left(\ell_f(x),x,f(x) - x\right) \in J^1 \R^{2n+1}
\mid x \in R^{2n+1}
\right\}.
\ee
Then for any $f$ $C^1$-close to the identity map, $\Delta_f$ is a section of the projection 
$\pi: J^1\R^{2n+1} \to \R^{2n+1}$, which is also Legendrian with respect to the contact form $\widehat{\mathscr A}$. 
For $f=id$, we have
\be\label{eq:Deltaid}
\Delta_{\id} = \CZ_{\CW_k^{2n+1}}.
\ee
\end{lem}
\begin{proof} By definition, we have $\pi \circ \Delta_f$ is bijective on $\R^{2n+1}$ which shows that $\Delta_f$ is
associated to a section of the projection $\pi$. By definition, we have
\be\label{eq:deltaPi-Gamma}
(\delta\Pi) \circ \Delta_f = \Gamma_f
\ee
and hence
$$
 \Delta_f^*(\widehat{\mathscr A}) = (\delta^{-1} \circ \Gamma_f)^*
 \widehat{\mathscr A} = \Gamma_f^*\mathscr A
$$
This shows that $ \Delta_f$ is Legendrian with respect to $\Pi^*\widehat{\mathscr A}$
if and only if $\Gamma_f$ is Legendrian with respect to $\mathscr A$.
This finishes the proof.
\end{proof}

\begin{rem} In \cite{rybicki2}, the notation $\Gamma_f$ is associated to  $f - \id$, not to $f$,
unlike the present paper. In the present paper, we reserve the notation $\Gamma_f$ the 
Legendrian graph in $M_\Q$ in general to be compatible with the notation from \cite{bhupal,oh:shelukhin-conjecture}.
The notation $\Delta_f$ for the Legendrian map is reserved for a Legendrian submanifold associated to $f-\id$
applied \emph{only for the 1-jet bundle $J^1\CW_k^{2n+1}$ with respect to the contact form $\Pi^*{\widehat{\mathscr A}}$}
which is closer to $\Gamma_f$ adopted in \cite{rybicki2}.
\end{rem}

Then we would like to associate to each such contactomorphism $f$ a real-valued function on $\CW_k^{2n+1}$.
Knowing that \eqref{eq:Deltaid} holds, we can apply the standard Darboux-Weinstein theorem (contact version).
On the other hand, we will also want some 
fiber preserving property for the chart.  (See \cite[p.3300]{rybicki2}.) To naturally construct such a chart,
we will apply some \emph{equivariance} for the Darboux-Weinstein
chart exploiting the linear structure of $\R^m$ (resp. $\CW_k^m$) whose description is now in order.

We consider the abelian groups 
\be\label{eq:group-G}
G: = \R^{2n+1} \quad \text{\rm or} \quad T^k \times \R^{n+1-k}
\ee 
and consider its actions on $M_{\CW_k^{2n+1}}$ and on $J^1\R^{2n+1}$ respectively.
We consider
the contact $G$-actions on $(M_{\R^{2n+1}}, \widehat{\mathscr A})$
$$
\mathcal G_1: (g,(t, x,X)) \mapsto (t, x + (g,0), X + (g,0))
$$
and on $(J^1\R^{2n+1},\alpha_0)$ the one given by
$$
\mathcal G_2: (g,(t, x,X) \mapsto (t, x + (g,0), X).
$$ 
Here $(g,0) \in  \left(S^1 \times (T^{k-1} \times R^{n-k+1}) \right) \times \R^n \cong \CW_k^{2n+1}$ 
and
we identify the momentum coordinates $p_x$ with $X$. which is possible by the
linearity of $Q = \R^{2n+1}$.

Then the map $\delta\Pi$ is a $(\CG_1,\CG_2)$-equivariant 
contact diffeomorphism which makes
the following commuting diagram
$$
\xymatrix{(J^1\R^{2n+1},\Pi^*\widehat{\mathscr A}) \ar[dr]^\pi \ar[rr]^{\delta\Pi}  && (M_{\R^{2n+1}},\mathscr A)
 \ar[dl]_{{\text{\rm pr}}_2}\\
& \R^{2n+1}.
}
$$
The following is a special form of an equivariant Darboux-Weinstein
theorem implicitly utilized by Rybicki \cite{rybicki2}. Recall $\Delta_{\id} = \CZ_{\CW_k^m}$.

\begin{prop} There exists a $\mathcal G_2$-equivariant open neighborhoods
$U$ and $V$ of $\CZ_{\CW_k^m} = R$ in $J^1\R^{2n+1}$ and
$\mathcal G_2$-equivariant diffeomorphism $\Phi_U: U \to V$ such that 
\begin{enumerate}
\item  $\Phi_U^*\alpha_0 =\widehat{\mathscr A}$ and $\Phi_U \circ \Delta_{\id} = \CZ_{\CW_k^m}$.
\item 
It satisfies
$$
T_{(0,x,0)} \Phi_U (\xi_x, \xi_x + \xi_X,\xi_t) = (\xi_t,\xi_x, \xi_X) \in T_{(0,x,0)}J^1\R^{2n+1}
$$
for each element $(\xi_x, \xi_x + \xi_X,\xi_t) \in  T_{(0,x,0)} M_{\R^{2n+1}}$.
\end{enumerate}
\end{prop}
\begin{rem} The geometric meaning of Statement (2) is that the derivative $d\Phi_U$ maps the linear polarization 
$$
\{(a,\xi_x, v+\xi_x) \mid \, v  \in \R^{2n+1}, a \in \R\}
$$
to 
$$
\{(a,\xi_x,v) \mid  v \in \R^{2n+1}, a \in \R\}
$$
for each given vector $v_x \in \R^{2n+1} = T_x \CW_k^{n+1}$.
For $\xi_x =0$, the two subspaces coincide while for $\xi_x \neq 0$, the first subspace is moved by
a translation of the second by $\xi_x$ in the $X$ direction (i.e., the fiber direction).
\end{rem}

An immediate corollary of the $G$-equivariance is the following.
For readers' convenience, we give its proof in Appendix \ref{sec:equivariance}.

\begin{cor}\label{cor:equivariance} We have the expression
$$
\Pi\Phi_U^{-1} (t,x,X)  = (x + h_x(X,t), x + h_X(X,t),(t + h_t(X,t)) \in  M_{\R^{2n+1}}
$$
such that $h_t(0,x,0) = 0$, $h_x(0,x,0) = 0$ and $h_X(0,x,0) = 0$.
If we define the map $\mathsf H$ to be
\be\label{eq:mathsfH}
\mathsf H(x,X,t) = (h_t(X,t), h_x(X,t), h_X(X,t))
\ee
as a map $M_{\R^{2n+1}} \to J^1\R^{2n+1} $, it does not depend on $x$.
\end{cor}

Obviously from this corollary, the map 
$$
(\Pi\Phi_U^{-1})^{-1} \delta^{-1}: (M_{\R^{2n+1}},\mathscr A) \to  (J^1\R^{2n+1},\alpha_0)
$$
is a contact diffeomorphism which
has the form
\be\label{eq:PhiU-1delta}
(\Pi \Phi_U^{-1})^{-1} \delta^{-1} = \Pi^{-1} + \mathsf H 
\ee
on $U'$.
(Recall that $\Pi$ is just the $\R$ coordinate swapping map and the identity map
if we identify $J^1\R^{2n+1}$ and $M_{\R^{2n+1}}$ with $\R^{2(2n+1)+1}$.) We summarize the above discussion into the diagram
$$
\xymatrix{& (M_{\R^{2n+1}},\widehat{\mathscr A})\ar[r]^{\delta}  & (M_{\R^{2n+1}},\mathscr A) \ar[d]^{{\text{\rm pr}}_2}
\ar[dl]^{(\Pi \Phi_U^{-1})^{-1} \delta^{-1}}\\
(J^1\R^{2n+1},\Pi^*\widehat{\mathscr A}) \ar[r]_{\Phi_U} \ar[ur]^{\Pi} & (J^1\R^{2n+1},\alpha_0)\ar[u]_{\Pi\circ \Phi_U^{-1}} \ar[r]_{\pi} & \R^{2n+1}
}
$$
where all maps in the left parallelogram and $\delta$ are contact diffeomorphisms by definition. 
In particular, the map
\be\label{eq:PhiUDeltaf}
(\Pi \Phi_U^{-1})^{-1} \delta^{-1} \circ \Gamma_f = \Phi_U (\delta \Pi)^{-1} \circ \Gamma_f = \Phi_U \circ \Delta_f
\ee
is a Legendrian embedding of $\R^{2n+1}$ into $(J^1\R^{2n+1},\alpha_0)$ 
that is also graph-like for any contactomorphism $f$ sufficiently $C^1$-close to the identity.
(See \eqref{eq:Deltaf}.)

We remark that the map $\mathsf H$ depends only on the Darboux-Weinstein chart $\Phi_U$.

\section{Parametrization of $C^1$-small contactomorphisms by 1-jet potentials}
\label{sec:legendrianization}
 
 We recall that any Legendrian submanifold $R$
 $C^1$-close to the zero section $\CZ_N \subset J^1N$ can be written as the 1-jet graph
 $$
 R_u := \image j^1u = \{(z,q,p) \mid p = du(q), \, z = u(q)\}
 $$
 for some smooth function $u: R \to \R$. We call $u$ the \emph{(strict) generating function}
of $R$. This correspondence is one-to-one in a $C^1$-small perturbations of 
the zero section.
 
\begin{rem} We remark that if $f$ is a contactomorphism $C^1$-close to the
identity (say, $M_1^*(f) = \max\{\|f\|_1, \|\ell_f\|_1\}  < \frac14$), its lift to $\R^{2n+1}$ can be written as
$\widetilde f  = \id + \widetilde v$  for a map that is $C^1$ close to the 
zero map.   This also implies that the map $v$ is $C^1$-close to the zero section map of the 
\emph{front projection} $J^1\R^n \to \R \times \R^n$ and $v \equiv 0$ when $f = \id_{\CW_k^m}$.
However unlike the diffeomorphism case of \cite{mather}, \emph{there is no
obvious such \emph{linear perturbation result} of contactomorphisms}.
We suspect that this phenomenon leads to the contact counterpart of the discussion 
about the failure of Mather's scheme in \cite{mather,mather2} 
on the nose for the critical case $r  = n+1$ as explained in \cite{mather2}. 
We will investigate this phenomenon elsewhere.
\end{rem}
With slight abuse of terminology,
 we call such a contactomorphism $C^1$ close to the identity a \emph{$C^1$-small}
 contactomorphism.

By a suitable contact conformal rescaling of $\Phi_U$, 
we may assume
\be\label{eq:reference-chart}
\pi(\Phi_{U}(U)) \supset [-1,1]^{2n+1}
\ee
for the chart $\Phi_U$ where $\pi: J^1 \R^{2n+1} \to \R^{2n+1}$ is the canonical projection. 
We set
 \be\label{eq:KUr}
 K_{\Phi_U,r}: = \sup_{0 \leq s \leq r+1} \max 
 \left\{\left\|D^s\Phi_{U \cap \widetilde E_1^{(0)}}\right\|, 
\left\|D^s(\Phi_{U \cap \widetilde E_1^{(0)}})^{-1}\right\|\right\},
\ee
This constant is a universal constant depending only on the Darboux-Weinstein 
chart $\Phi_U$ and $r$. We remark that the set $U \cap \widetilde E_1^{(0)}$
is relatively compact and so $K_{\Phi_U,r} < \infty$ for all $r \geq 1$.

\emph{We will fix the chart $\Phi_U$
 as the reference  in the rest of the paper}
and other charts will be obtained by further conformal rescalings. The latters
will be denoted by
$$
\Phi_{U;A}, \quad A \geq 1
$$
depending on the constant $A$, whose precise definition is now in order.

\subsection{Rescaled Darboux-Weinstein chart $\Phi_{U;A}$}

We recall the maps $\chi_A$ and $\eta_A$ given in \eqref{eq:chit} and \eqref{eq:etat}
respectively with $t = \log A$. (See \eqref{eq:logA}.) Both satisfy that
\be\label{eq:chiA-conformal-factor}
\chi_A^*\alpha_0 = A^2 \alpha_0 = \eta_A^*\alpha_0
\ee
where $\alpha_0$ is the standard contact form on $\R^{2n+1} \cong J^1\R^n$.
Furthermore, we have
\bea
\chi_A\left([-1,1]^{2n+1}\right) =  [-2A^2,2A^2] \times [-A,A]^{2n}  \label{eq:chi-square}\\
\eta_A\left([-1,1]^{2n+1}\right) = [-A,A]^{n+1} \times [-1,1]^n. \label{eq:eta-square}
\eea
By a suitable conjugating process 
by the maps $\chi_A$  and $\nu_A$, we lift the map $\chi_A$ on $\CW_k^m$ to the maps
\beastar
\mu_A & =  & \chi_A \times \chi_A \times \id_\R \quad \text{on }\, 
 \CW_k^m \times \CW_k^m \times \R = M_{\CW_k^m}\\
\nu_A & =  & (A^2\id_\R) \times \chi_A \times \eta_A  \quad \text{on }\, \R \times T^*\CW_k^m
= J^1\CW_k^m
\eeastar
as contact automorphisms, respectively. In fact they can be lifted to 
$\R^{2(2n+1) +1} $ explicitly expressed as
\bea
\mu_A(z,q,p, Z,Q,P,\eta) & = & (A^2z, Aq, Ap, A^2Z,AQ,AP,\eta), \label{eq:muA}\\
\nu_A(t, z,q,p, Z,Q,P) & = & (A^2 t, A^2z, Aq, Ap, AZ, AQ,P) \label{eq:nuA}
\eea
in terms of the standard coordinates of $M_{\R^{2n+1}}$ and $J^1\R^{2n+1}$ respectively.
It is easy to check from this that they indeed satisfy  
$$
\mu_A^*{\widehat{\mathscr A}} = A^2 \widehat{\mathscr A}, \quad \nu_A^*\alpha_0 = A^2 \alpha_0
$$ 
respectively, i.e., $\mu_A$ 
and $\nu_A$ define contactomorphisms of $(J^1\R^{2n+1}, \widehat{\mathscr A})$ and $(J^1\R^{2n+1},\alpha_0)$, respectively.  This being said, we will also denote by the same
notation $\mu_A$ for the obvious action on $J^1\R^{2n+1}$ conjugate by the map $\Pi$,
recalling that $\Pi$ is essentially the identity map as a map defined on $\R^{2(2n+1)+1}$.

\begin{defn}[$\Phi_{U;A}$]
Let $A \geq1$ be given and consider the expression 
\be\label{eq:PhiUA}
\Phi_{U;A}: = \nu_A \circ \Phi_U \circ \mu_A^{-1}: U_A \to V_A 
\ee
and
write
$$
U_A : = \mu_A(U), \, V_A : = \nu_A(V) 
$$
where we regard the subsets as ones either on $J^1\R^{2n+1}$ or on $J^1\CW_k^m$.
\end{defn}

Then the rescaled chart map 
$$
\Phi_{U;A}: (U_A , \Pi^*\widehat{\mathscr A})\to (V_A,\alpha_0)
$$
is a  well-defined strict contactomorphism for all $A > 1$.  

\begin{prop}[Compare Proposition 4.2 (2) \cite{rybicki2}] \label{prop:PhiUK} Let $r \geq 2$ be given.
For any given $A_0 > 1$, let $1 \leq A \leq A_0$ and consider a subinterval $E \subset E_A^{(k)}$.
\begin{enumerate}
\item  Then the  map
$$
\Phi_{U;A} : U_A \to V_A
$$
is defined and satisfies $\Phi_{U;A}|_R = id_R$, 
$\Phi_{U;A}^*\widehat{\mathscr A} = \alpha_0$.
\item 
Define the constants
\be\label{eq:KUrA}
 K_{\Phi_U,r,A}: = \sup_{0 \leq s \leq r+1} \max 
 \left\{\left\|D^s\Phi_{U;A}|_{U_A \cap \widetilde E_A^{(0)}}\right\|, 
\left\|D^s(\Phi_{U;A} |_{U_A \cap \widetilde E_A^{(0)}})^{-1}\right\|\right\}.
\ee
Then $ K_{\Phi_U,r,A} \leq A^2 K_{\Phi_{U,r}}$ for all $0 \leq s\leq r$.
\end{enumerate}
\end{prop}
An upshot of this proposition is that when $A_0 > 0$ and $r$ are given, the constants
cam be uniformly controlled depending only on the original chart $\Phi_U$ and 
on $A_0$ which however does not depend on individual $A$ from $1 \leq A \leq A_0$.

Following \cite{rybicki2}, we  consider the cylinders
 \be\label{eq:CWmk}
 \CW_k^{2n+1}: = (S^1)^k \times \R^{2n+1 - k} 
\ee
for $k = 0, \ldots, n$
where we write the $z$ coordinates \emph{first} and set $q_0: = z$. 
For $k = 0$, we have $\CW_0^{2n+1} = \R^{2n+1}$ and for $k = 1, \ldots, n$, and for
$k \geq 1$ we can write
$$
\CW_k^{2n+1} \cong S^1 \times T^*(T^{k-1}  \times \R^{n-k+1}) 
$$
which is manifestly a contact manifold as the (circular) contactization of 
the symplectic manifold. $T^*(T^{k-1}  \times \R^{n-k+1})$.

For $k =1, \ldots, n$,  we  consider
\be\label{eq:EA(k)}
E_A^{(k)}: = (S^1)^k \times [-A,A]^{2n+1 - k} \subset \CW_k^{2n+1}
\ee
and for $k = 0$
 $$
E_A^{(0)}: =[-A,A]^{2n+1} \subset \R^{2n+1}.
$$
We write the associated coordinates by $(\xi_0, \cdots, \xi_n, p_1, \cdots, p_n)$
with 
\beastar
\xi_j &\equiv& q_j \mod 1\,\, \quad \text{for $j = 0, \cdots, k$}, \\
\xi_j & = & q_j\quad \quad \, \quad \qquad \text{for $j = k+1, \cdots, n$}.
\eeastar
We then consider the family of subsets
\beastar
\widetilde E_A^{(0)}: & = & [-2A^2,2A^2] \times [-A,A]^{m-1} \times \R^{m+1} 
\left( = \chi_A([-1,1]^m \times \R^{m+1}\right),\\
\widetilde E_A^{(k)}: & = & T^k \times [-A,A]^{m - k} \times \R^{m+1}
\left(= \pi \left(\R^k \times [-A,A]^{m-k} \times \R^{m+1}\right)\right)
\eeastar
and equip the  contact forms induced from 
$\widehat{\mathscr A}$ on the contact product $M_{\R^{2n+1}}$. 

 \subsection{Representation of contactomorphisms by their 1-jet potentials}
\label{subsec:representation}  
  
Let $C^\infty_E(\CW_k^{2n+1},\R)$ be the set of $\R$-valued functions 
compactly supported in a closed subset $E \subset \CW_k^{2n+1}$.

Composing the 1-jet map $j^1u$ with the Darboux-Weinstein chart
 $\Phi_{U;A}$, we obtain the following parametrization of 
 $C^1$-small neighborhood of the identity map
  in $\Cont(\CW_k^m,\xi)$.
    
 \begin{prop}\label{prop:GA}  Let $A_0 > 1$ be  given. Then there exist 
  $\CU_A$ is a $C^1$-small neighborhood
 of the identity and $\CV_A$  a $C^2$-small neighborhood of zero in $C^\infty_E(\CW_k^{2n+1},\R)$
 and a one-to-one correspondence 
 $$
 {\mathscr G}_A: \Cont_E(\CW_k^{2n+1},\alpha_0) \supset \CU_A \to  \CV_A \subset 
 C^\infty_E(\CW_k^{2n+1},\R)
 $$
 that satisfies ${\mathscr G}_A(id) = 0$ and is continuous
over $1 \leq A \leq A_0$.
 \end{prop}
 \begin{proof} Let $\pi: J^1\CW_k^m \to \CW_k^m$ be the canonical projection.
Consider submanifold 
$$
R_{g,A} : = \operatorname{Image} \Phi_{U;A} \circ \Delta_g
$$
which is Legendrian by construction. Moreover the projection $\pi$ restricted to $R_{g,A}$
 becomes one-to-one provided the $C^1$-norm of $g$ is sufficiently small. Therefore 
 there exists a unique real-valued function $u$ such that we can express
$R_{g,A}  = \operatorname{Image} j^1 u$ for the 1-jet map of $u$. 

 We define ${\mathscr G}_A(g)$ to be the unique function $u: \CW^{2n+1}_k \to \R$
 satisfying
 \be\label{eq:defining-PhiA}
\image (\Phi_{U;A} \circ \Delta_g) =  \image (j^1u)
\ee
where $u = u_{g,A}$ depends not only on $g$, $A$ but also on the chart $\Phi_{U;A}$.
 (See Diagram \ref{eq:diagram}.)  We alert readers that \emph{while their images coincide
 the two maps $\Phi_{U;A} \circ \Delta_g$ and $j^1u$ are different as a map.}

 Then we put
 \be\label{eq:GA}
 {\mathscr G}_A(g):= u_{g;A}.
 \ee 
For the statement on the properties, we further examine the definition.
By the definition of the Legendrianization parametrization map $\mathscr G_A$ and the equality
\eqref{eq:defining-PhiA},  there exists  $y = y(x)$ such that
$$
\Phi_{U;A}\circ \Delta_g(x) = (u_{g,A}(y),y, Du_{g,A}(y))
$$
for each $x \in \CW_k^m$, and that such  $y$ is unique, provided $g$ is sufficiently $C^1$-small and
the neighborhood $U$ is sufficiently small. We can express
\beastar
y &= & \pi_2 \Phi_{U;A}\circ \Delta_g(x)\\
u_{g,A}(y) &= & \pi_1\Phi_{U;A}\circ \Delta_g(x) \\
Du_{g,A}(y) & = & \pi_3 \Phi_{U;A}\circ \Delta_g(x).
\eeastar
Furthermore the map 
$$
\pi_2 \Phi_{U;A}\circ \Delta_g=: \Upsilon_{g,A}
$$
is a self diffeomorphism  of $\R^{2n+1}$ map 
if $g$ is sufficiently $C^1$-small. We can write the first equation as
\be\label{eq:Upsilon}
x = \Upsilon_{g,A}^{-1}(y).
\ee
Then we can express
\beastar
u_{g,A} &= & \pi_1\Phi_{U;A}\circ \Delta_g\circ \Upsilon_{g;A}^{-1} \\
Du_{g,A} & = & \pi_3 \Phi_{U;A}\circ \Delta_g\circ \Upsilon_{g;A}^{-1}.
\eeastar
This expression already clearly shows the continuity of the map
$$
\mathscr G_A: (g, A) \mapsto u_{g,A}
$$
in the  $C^r$ topology of $g$ and in the $C^{r+1}$ topology of $u = u_{g,A}$,
respectively.

The last statement immediately follows from this presentation.
 \end{proof}

\begin{rem} The fact that the two maps are not the same 
complicates the relationship between the contactomorphism $g$ and the function $u_{g,A}$
as shown below. This will give rise to  some difficulty later in Section \ref{sec:estimates-legendrianization} 
when we try to compare the $C^r$ estimates of $g$ and 
that of the function $u_{g,A}$. 
\end{rem}

 \begin{defn}[1-jet potential] Let
 \be\label{eq:CU1}
 \CU_1 \subset  \Cont_c(\CW_k^{2n+1}, \alpha_0)
 \ee
be a $C^1$-small neighborhood of the identity of $\Cont_c(\CW_k^{2n+1}, \alpha_0)$.
 For any $C^1$-small contactomorphism $g \in \CU_1$, 
 we call the function $u = u_g$ satisfying \eqref{eq:GA} (for $A = 1$) the \emph{1-jet potential} of 
 contactomorphism$g \in \CU_1$ with respect to $\alpha$ and $\Phi_U$.
 \end{defn}

We will fix a cut-off function  $\psi: \CW_k^{2n+1} \to [0,1]$ whose precise
defining properties will be given later.

\begin{prop}[Compare with Proposition 5.4 \cite{rybicki2}]
\label{prop:Graphf} Let $E \subset E_A^{(k)}$ be a sub-interval of $E_A^{(k)}$.
There exists a $C^1$-neighborhood $\CU_{\chi,A} \subset \CU_1$ of the identity 
in $ \Cont_{E_A^{(k)}}(\CW_k^{2n+1},\alpha_0)$
such that for any $g \in \CU_{\chi,A}$ with support $E$ the contactomorphism
$$
g^\psi: = {\mathscr G}_A^{-1}(\psi {\mathscr G}_A(g)) = {\mathscr G}_A^{-1}(\psi u_{g,A})
$$
is well-defined and $\supp(g^\psi) \subset E$. More precisely, we have $\supp(g^\psi) \subset \supp(\chi)$
and $g^\psi = g$ on any open subset $U \subset \CW_k^{2n+1}$ with $g = 1$ on $U$.
\end{prop}

We will just write $u_g = u_{g;A}$ as in \cite{rybicki2} for the simplicity of notation,
whenever there is no danger of confusion.

\begin{rem}\label{rem:M0g-j1ug} Observe that the identity \eqref{eq:defining-PhiA} 
relates the two maps $\Gamma_g$ and $j^1u_g$ for $u_g = \mathscr G_A(f)$
explicitly via the chart $\Phi_{U;A}$ which depends only on the fixed chart $\Phi_U$ and
the rescaling constant $A >1$. In particular, the identity shows the equivalence of the two
norms 
$$
M_0^*(g) = \max\{\|g - \id\|_{C^0}, \|\ell_g\|_{C^0}\}
$$
 and 
 $$
 \|j^1u_g\|_{C^0} =  \max\{\|Du_g\|_{C^0},\|u_g\|_{C^0}\}
 $$
 when $A$ varies $1 < A \leq A_0$ for any fixed constant $A_0 > 1$.
  \end{rem}

\section{Correcting contactomorphisms via the Legendrianization}

The construction in the present section, which was introduced and utilized by Rybicki \cite{rybicki2},
presents its feature applicable only to the case of contactomorphisms 
in which the parametrization of a $C^1$ 
neighborhood of the identity of $\Cont_c(\CW_k^{2n+1},\alpha_0)_0$  is achieved by
a $C^2$-neighborhood of $0$ in $C^\infty(\CW_k^{2n+1},\R)$ through taking the 1-jet potentials.
Such a construction was not needed for the case of general diffeomorphisms or even for
the case of symplectic diffeomorphisms \cite{banyaga}.

One important aspect of the Euclidean space $\R^n$
in the study of its diffeomorphism groups,
  although not manifest enough at the time of the advent of
\cite{mather}, is the \emph{linear structure} $\R^n$ so that any 
$C^1$-small perturbation of the identify map can be written in the form
$$
\id + v: \R^n \to \R^n
$$
where $v: \R^n \to \R^n$ is $C^1$-close to the zero map as well as 
its $C^0$-norm.

For the case of contact space $(\R^{2n+1},\xi)$, there is no such a simple form of
perturbation in the context of contactomorphisms, which prevents 
one from directly applying Mather's':
\begin{itemize}
\item construction of rolling-up operators, or
\item utilizing the homothetic transformations. 
\end{itemize}
The upshot of Rybicki's proof in \cite{rybicki2}
of perfectness of $\Cont_c(\R^{2n+1},\alpha_0) = \Cont_c^\infty(\R^{2n+1},\alpha_0)$ is to correctly
\emph{contactify} the two operations. 

We start with the observation that we have a natural covering projection
$$
\text{\rm pr}_{k+1}: \R^{2n+1} \to \CW_{k+1}^{2n+1}
$$
and that any sufficiently $C^1$-small (and so $C^0$-small) 
contactomorphism $g\in \CW_{k+1}^{2n+1}$ can be uniquely lifted to a 
contactomorphism $\widetilde g$ of  $\CW_{k+1}^{2n+1}$ that satisfies
$$
g = \widetilde g \circ \text{\rm pr}_k
$$
and that $ \Gamma_{\widetilde g}$ is periodic in the variable $\xi_k$.

Based on the interpretation given in Appendix \ref{sec:equivariance}, we can express 
the group $\Cont_c^{T^\ell}(\CW_k^{2n+1},\alpha_0)$ consisting of $T^\ell$-equivariant contactomorphisms
as follows:
For each $\ell = 1, \cdots n+1$, we have
\beastar
&{}& \Cont_c^{T^\ell}(\CW_k^{2n+1},\alpha_0)\\
& = & \{f \in \Cont_c(\CW_k^{2n+1},\alpha_0) \mid
\text{\rm $f- \id$ does not depend on $\xi_i$ with $0 \leq i \leq \ell$}\}.
\eeastar
Since $f- \id =: v$ does not depend on $\xi_i$ with $0 \leq i \leq \ell$, we may abuse the 
notation $v$ by omitting the projection $\pi: \R^{2n+1} \to \R^{2n+1 - \ell}$
from $v\circ \pi$ and just write
$$
v\circ \pi(\xi_0,\ldots, \xi_\ell, \ldots \xi_n, p_1, \ldots, p_n) 
: = v(\xi_{\ell+1}, \ldots, \xi_n, p_1, \ldots,p_n).
$$
When $\ell = k$, this becomes
\be\label{eq:Tk-invariance}
v(\xi_0,\ldots, \xi_k, \ldots \xi_n, p_1, \ldots, p_n) = v(\xi_{k+1}, \ldots, \xi_n, p_1, \ldots, p_n), 
\ee
i.e., $v$ can be identified with a map $v: \R^{n-k} \times \R^n \to  \CW_k^{2n+1}$ 
by abusing notation. 

\begin{cor} Suppose that $f  \in \Cont_c^{T^\ell}(\CW_k^{2n+1},\alpha_0) \cap \CU_1$, i.e.,
$T^\ell$-equivariant. Then the function $u_f = u_{f;A}$ defined by
$$
\mathscr G_A(f) =: u_f \in C^\infty(\CW_k^{2n+1}, \R)
$$ 
 a $T^k$-invariant function on $\CW_k^{2n+1}$ where the $T^k$ 
acts on the $k$ circle factors of $\CW_k^{2n+1} = S^1 \times T^*(T^{k-1} \times \R^{n-k}) \cong (S^1)^k \times \R^{2n - k+1}$
by the standard rotations of circles. In particular,  $u_f$ does not depend on  the variables
$\xi_0, \ldots, \xi_{k-1}$.
\end{cor}
\begin{proof} This is an immediate consequence of the discussion given in Section \ref{sec:legendrianization}.
\end{proof}

\section{Contact-scaling and shifting of the supports of contactomorphisms}
\label{sec:shifting-supports}

Suppose a positive integer $A > 0$ which is chosen sufficiently large whose
size is to be determined later. For each given such integer, 
we start with the $(2n+1)$-cube $[-1,1]^{2n+1}$ and consider its shifting by 
an integer $-(2A-1) \leq k_i \leq 2A - 1$ in the $p_i$-direction.  
By the choice $ 0 \leq k_i \leq 2A-1$, we have $[k_i-1, k_i +1] \subset [-2A,2A]$.

Then we consider the conjugation of $f$
\be\label{eq:rhofrho-1}
\rho_{A,t} \circ f \circ \rho_{A,t}^{-1}
\ee
by the (affine) contactomorphism
\be\label{eq:rhoAt}
\rho_{A,t} := \chi_{A^2}  \circ \sigma_i^t, \quad i = 1, \cdots, n.
\ee
\begin{rem}\label{rem:difference} We would like to highlight one difference between
our definition of $\rho_{A,t}$ and that of \cite{rybicki2}: We do not involve the front scaling 
transformation  $\eta_A$ but do only $\chi_A$ by taking 
 the square of $\chi_A$ instead.  This turns out to be a
crucial change to be made for the purpose of obtaining
 the optimal power of $A^{4-2r + 2n}$. In this regard, it is crucial
for the length of  $q$-rectangularpid to become $A^2$ which is responsible for the coefficient
$2$ in front of $r$ and $n$, and the number 4 in the constant term in the exponent 
comes from the power of the $z$ direction for the contact rescaling operation
$$
\chi_A^2(z,q,p) = (A^4 z, A^2 q, A^2 p).
$$
The map $\chi_A \circ \eta_A$ provides the power  1 of $A$ in the $q$-direction 
while the power 3 in the $z$ direction, which would give rise to the power
$A^{4 -r + n}$ which will give rise to only $r > n+ 4$. This is the reason why
we use the map $\chi_A^2 = \chi_{A^2}$ instead of $\chi_A\eta_A$ used in \cite{rybicki2}.
\end{rem}

We also define the vector version of the conjugation \eqref{eq:rhofrho-1}
\be\label{eq:rhoAvect}
\rho_{A,{\bf t}} = \chi_{A^2}  \circ \sigma_i^{{\bf t}} \quad i = 1, \cdots, n
\ee
where we write $ {\bf t}: = \sum_{i=1}^n t_i \vec e_i$ and
$$
\sigma_i^{{\bf t}}: = \sigma_1^{t_1} \circ \sigma_2^{t_2} \circ \cdots \circ \sigma_n^{t_n}:
$$
It has the following explicit formula (before applying cutting-off: we recall
 the discussion given in Subsection \ref{subsec:contact-cutoff} here)
\be\label{eq:rhoAtt}
\rho_{A,{\bf t}}(z,q,p) = \left(A^4 z + \sum_{i=1}^n, A^2  q , A^2 (p + t \vec f_i)\right), \quad 
{\bf t} = (t_1,t_2, \ldots,t_n)_p = \sum_{i=1}^n t_i \vec f_i.
\ee
Observe that the image $\rho_{A;{\bf t}}\left([-2,2]^{2n+1}\right)$ is contained in
the following rectangularpid
\be\label{eq:rho-square}
  A^4 
\left(\prod_{i=1}^n [ -2 - t \sum_{i=1}^n |q_i|, 2 + t\sum_{i=1}^n |q_i|]\right) \times
 [-2A^2,2A^2]^n  \times  A^2 \left(\prod_{i=1}^n [-2 - |t_i|,2+|t_i| ] \right).
\ee

\begin{lem} Let  ${\bf k} =  \sum_{i=1}^n k_i \vec f_i + \left(\sum_{i=1}^n k_i\right)\cdot \vec e_z$
with $[k_i| \leq 2A-1$ for all $i = 1, \cdots, n$. Consider $f \in \Cont_c(J^1 \R^n, \alpha_0)$ 
satisfying
$$
\supp f \subset  [-2,2]^{2n+1} + {\bf k}.
$$
Then for any ${\bf t} = (t_1, \ldots, t_n)$ with $|t_i| \leq 2A-1$,
\be \label{eq:supp-f-conjugation}
\supp \left( \rho_{A,{\bf t}} f \rho_{A,{\bf t}}^{-1}\right) 
\subset  [-3 A^5, 3 A^5] \times [-2A^2,2A^2]^n \times [-2A^3,2A^3]^n
\ee
provided $A \geq 2n$.
\end{lem}
\begin{proof} Consider the map $\rho_{A,t}$ above associated to $i$.
We evaluate
\be\label{eq:rho-1}
 \rho_{A,{\bf t}}^{-1}(z,q,p) = \left(\frac{z - \sum_{i=1}^n t_iq_i }{A^4}, \frac{q}{A^2}, 
 \frac{p - {\bf t} }{A^2}\right).
\ee
Therefore if this point is not contained in $\supp f$, then one of the following inequalities holds:
\begin{enumerate}
\item  $\frac{z  - \sum_{i=1}^n t_i q_i}{A^4} \not \in [-2,2] + \sum_{i=1}^n k_i$ for some $i$,
\item  $\left| \frac{q_i}{A^2} \right|> 2$ for some $i$,
\item 
$\frac{p_i - t_i }{A^2}  \not \in [-2+k_1, 2+ k_i] $
for some $i$, 
\end{enumerate}
If  (3)  holds, then 
$$
p_i < (-2+k_i) A^2+ t_i  \quad \text{\rm or} \quad p_i  >  (2+k_i) A^2 + t_i.
$$
Using $|k_i| \leq 2A-1$ and $|t_i| \leq 2A$, we derive
$$
p_i < - 2A^3  \quad \text{\rm or}\quad p_i >  2 A^3
$$
and hence $p_i \in [-2A^2, 2A^2]$.

If (1) holds but (2) fails to hold,  we have inequalities
$$
\frac{q_i  }{A^2}  \in [-2,2] 
$$
for all $i = 1, \cdots, n$ and
$$
z <  A^4\left(-2 + \sum_{i=1}^n k_i \right) + \sum_{i=1} t_i q_i \quad 
\text{\rm or} \quad z >  A^4 \left(2 + \sum_{i=1}^n k_i\right) + \sum_{i=1} t_i q_i
$$
for some $i$. The first inequality implies
$$
-2A^2  \leq q_i \leq 2 A^2
$$
Combining this with the second inequality, we have derived that any point $(z,q,p)$ satisfying
$$
z <  -2 A^4 +  nA^4(-2A +1 - 2A) \quad \text{\rm or} \quad z > 2A^4 + n A^4(2A-1 + 2A)
$$
is not in $\supp (\rho_{A,{\bf t}} f \rho_{A,{\bf t}}^{-1})$. In particular, this holds if
$z < -3 A^5$ or $z > 3 A^5$,  provided $A \geq 5n$.

Combining the above altogether, we have finished the proof.
\end{proof}

This leads us to the consideration of the following subsets of $\CW_k^{2n+1}$.

\begin{defn}[$J_A$] Let $0 \leq t_i \leq 2A$ and $A > 0$. We define 
\bea
J_A&: = & [-3A^5,3A^5] \times [-2A^2,2A^2]^n \times  [-2A^3,2A^3]^n \label{eq:JA}\\  
&\subset&  \R \times \R^n \times \R^n \cong  J^1\R^n. \nonumber
\eea
\end{defn}
We will always make this choice of $A > 0$ and $t_i$'s from now on.
\begin{rem}
We warn the readers that our definition of $J_A$ and many others with the same
notations from \cite{rybicki2} are different therefrom in their specific 
choices of the orders of powers of $A$. Our choices are made to 
 make the relevant constructions and estimates in \cite{rybicki2} as optimal as possible.
 We encourage readers to compare the differences of the exponents of the power of $A$
 appearing in the definitions of various contact cylinders and rectangularpids below.
 It is important for the power of $A$ to be 2 for the  $q$-cube factors appearing in
 \eqref{eq:KA} below for our purpose of obtaining the optimal thresholds.
 \end{rem}
We will also need to consider the following families of contact cylinders
\beastar
J_A^{(k)} & = & S^1 \times  (S^1)^{k-1} \times[-2A^2,2A^2]^{n-k+1} \times [-2A^3,2A^3]^n \\
K_A^{(k)} & = & S^1 \times  (S^1)^{k-1}  \times [-2,2] \times[-2A^2,2A^2]^{n-k} 
\times [-2A^3,2A^3]^n
\eeastar
for $k = 1, \ldots, n$ and
\bea
J_A^{(0)} & = & J_A = [-3A^5,3A^5]\times [-2A^2,2A^2]^n \times [-2A^3,2A^3]^n\nonumber\\
K_A^{(0)} & = & K_A = [-2,2] \times[-2A^2,2A^2]^n \times [-2A^3,2A^3]^n. \label{eq:KA}
\eea
of the construction of a family of Rybicki's rolling-up operators 
$$
\Psi_A^{(k)}: \Cont_{J_A^{(k)}}(\CW_k^{2n+1}, \alpha_0)_0 \cap \CU_1 \to 
\Cont_{K_A^{(k)}}(\CW_k^{2n+1},\alpha_0)_0
$$
for which the rolling occurs in the $q_k$-coordinate direction. Note that 
we have natural covering projections $J_A \to J_A^{(k)}$ and $K_A \to K_A^{(k)}$,
respectively. The fiber of $J_A \to J_A^{(k)}$ is isomorphic to
$$
\Z_{4A^2} \times (\Z_{4A^3})^{k-1}, 
$$
and the fiber of $K_A \to K_A^{(k)}$ is isomorphic to
$$
\Z_4 \times (\Z_{4A^2})^{k-1}.
$$
Similarly $J_A^{(k)} \to K_A^{(k)}$ has fiber isomorphic to $\Z_{A^2}$.

\section{Rybicki's fragmentation of the second kind: Definition}
\label{sec:2nd-fragmentation}

In \cite{rybicki2}, Rybicki introduced some fragmentation for the case of contactomorphisms 
which involves 
a \emph{fragmentation of the 1-jet potentials} defined in the previous section. 
Such a fragmentation is uniquely applicable to the case of contactomorphisms because its natural analog does not exist
either in the case of diffeomorphisms \cite{mather,epstein:commutators} nor in that of
symplectomorphisms \cite{banyaga}. 
Rybicki \cite{rybicki2} introduced  the following-type of  fragmentation 
which he calls \emph{the fragmentations of the second kind}.

We start with the Rybicki's fragmentation in the direction of $z = \xi_0$.

\begin{lem}[Compare with Proposition 5.6 \cite{rybicki2}]\label{lem:2nd-fragmentation-z} 
Let $2A > 1$ be an even integer,  $\psi:[0,1] \to [0,1]$ be a boundary-flattening
function such that $\psi \equiv 1$ on $[0,\frac14]$ and $\psi \equiv 0$ on $[\frac34,1]$, and let
$$
E_{2A} : = E_{2A}^{(0)} = [-2A,2A]^{2n+1}.
$$
Then there exists a $C^1$-neighborhood $\CU_{\psi,A}$ of the identity 
in $\Cont_{E_{2A}}(\R^{2n+1},\alpha_0)$
such that for any $g \in \CU_{\psi,A}$ there exists a factorization
\be\label{eq:psi-decompose2}
g = g_1 \cdots g_{4k+1}, 
\ee
that satisfies the  following properties:  The factorization is uniquely determined by 
$\Phi_A$, $\psi$ and $A$ so that 
$\supp(f_K)$ is contained in an interval of the form
$$
\left(\left[k - \frac34,k+\frac34\right] \times \R^{2n} \right) \cap E_{2A},
$$
with $ k \in \Z, |k_i| \leq 2A$.
\end{lem}
\begin{proof} We extend $\psi$ to $[-1,1]$ as an even function and then to $\R$ as a 2-periodic function. We lift it to $E_{2A}$.
Let $g$ be a sufficiently $C^1$-small contactomorphism with $\supp g \subset E_{2A}$ and consider
the function
$$
g^\psi: = \mathscr G_A^{-1}\left(\psi \mathscr G_A(g)\right) = \mathscr G_A^{-1}(\psi u_g).
$$
We now take a factorization $g = g_2^\psi g_1^\psi$ by first defining
\be\label{eq:gpsi1}
g^\psi_1  : =   g_{-2A} \circ g_{-2(A-1)} \circ \cdots \circ g_{2(A-1)} \circ g_{2A},
\ee
and then setting
\be\label{eq:gpsi2}
g^\psi_2  : =   g(g^\psi_1)^{-1}
\ee
where $g_i$'s satisfy $\supp g_{2k} \subset [2k-\frac34,2k + \frac34 ] \times \R^{2n}$ and 
$\supp g_{2k+1} \subset [2k+\frac14,2k + \frac74] \times \R^{2n}$. 
\end{proof}
We mention that the collections
$\{\supp g_{2k}\}_{k=1}^n$ and $\{\supp g_{2k+1}\}_{k=1}^n$ are disjoint from one another
respectively.
By applying the above construction consecutively to all variables $(z, q, p) = (\xi_0,\xi,p)$, we 
obtain the following.

\begin{prop}[Compare with Proposition 5.7 \cite{rybicki2}]\label{prop:2nd-fragmentation} 
Let $A$, $\psi$ and $E_{2A}$ be as in Lemma \ref{lem:2nd-fragmentation-z}.
Then there exists a $C^1$-neighborhood $\CU_{\psi,A}$ of the identity 
in $\Cont_{E_{2A}}(\R^{2n+1},\alpha_0)$
such that for any $g \in \CU_{\psi,A}$ there exists a factorization
$$
g = g_1 \cdots g_{a_m}, \quad a_m = (4A + 1)^m
$$
that satisfies the  following properties:  The factorization is uniquely determined by 
$\Phi_A$, $\chi$ and $A$ so that 
$\supp(f_K)$ is contained in an interval of the form
$$
\left(\left[k_1 - \frac34,k_1+\frac34\right] \times \cdots \times 
\left[k_m - \frac34,k_m+\frac34\right]\right) \cap E_{2A},
$$
with $ k_i \in \Z, |k_i| \leq 2A$.
\end{prop}

\section{The `hat' operation: deforming to an $S^1$-equivariant map}
\label{sec:hat-operation}

 Now we take the `hat' operation of turning the given contactomorphism into 
 one that becomes $S^1$-symmetric in an additional direction of $q_i$s.
 This is the analog to the construction given in \cite[p. 524]{mather}. However
 the direct application of Mather's construction cannot work for the contactomorphisms
 \emph{because the addition operation $+$ on the vector space $\R^{2n+1}$ 
 does not respect the contact property.} Here enters one of Rybicki's key ideas
 of exploiting the representation of contactomorphisms $g$ sufficiently 
 $C^1$-close to the identity by the Legendrian graph of $g -id$
 in the contact product $(M_{\CW_k^{2n+1}}, \widehat{\mathscr A})$
 followed by  their contact potentials $u_g \in C^\infty(\CW_k^{2n+1},\R)$ via the following
 sequence of one-to-one correspondences:
 \be\label{eq:equivalence}
 g \longleftrightarrow g - \id \longleftrightarrow \Gamma_g \longleftrightarrow\Delta_g
  \longleftrightarrow \Phi_{U;A}(\Delta_g) 
 \longleftrightarrow \text{\rm Image}\, j^1 u_g \longleftrightarrow u_g.
 \ee
 This construction is reversible and respects the $T^{k+1}$-action on $\CW_k^{2n+1}$
 which are given by the linear rotations of the underlying $(k+1)$ torus 
 $\CW_k^{2n+1} \to (S^1)^{k+1} = T^{k+1}$.

 The upshot of this step is that both $\CW_k^{2n+1}$
 and  the set $C^\infty(\CW_k^{2n+1},\R)$ are  \emph{ linear},
 and hence we can apply the Mather-type constructions thereto and then read back the above
 diagram to obtain a \emph{contact diffeomorphism} $g$ associated to any given
 function $v \in C^\infty(\CW_k^{2n+1},\R)$ sufficiently $C^2$-close to the zero function.
The detail of the construction is now in order. (See \cite[p. 3309]{rybicki2} for the 
relevant counterpart.)

Consider the cylinders
$$
E_A^{(k)}: = (S^1)^k \times [-A,A]^{2n+1-k} \subset \CW_k^{2n+1}, \quad k = 1, \ldots, n+1.
$$
We start with a $T^k$ equivariant element
$$
g \in \Cont_{E_A^{(k+1)}}^{T^k}(\CW_k^{2n+1},\alpha_0)_0
\subset \Cont_c^{T^k}(\CW_k^{2n+1},\alpha_0)_0
$$
sufficiently $C^1$-close to the identity.  By definition,  $g$ is equivariant under the $T^k$ action,
and so $u_g$ is a $T^k$-invariant real-valued function by Proposition \ref{prop:Graphf}.
Therefore we can express
$$
u_g = u'_g \circ \text{\rm pr}_k^c; \quad \text{\rm pr}_k^c: \CW_{k+1}^{2n+1} \to \R^{2n-k-1}
$$
for some function $u'_g: \R^{2n-k-1} \to \R$. By the identification of 
$$
\R^{2n-k-1} 
 \cong \{[(0,\ldots,0)]\} \times  \R^{2n-k-1} \subset \R^{2n+1},
 $$ 
 $u'_g$ can be canonically defined from $u_g$ by defining
 $$
 u'_g(\xi_k, \xi_{k+1}, \ldots, \xi_n, p): = u([(0,\ldots, 0)],\xi_k, \xi_{k+1}, \ldots, \xi_n, p).
 $$
 Here $[(0,\ldots, 0)] \in \R^{k+1}/\Z^{k+1}$ is the identity element.
 
Now we define a $T^{k+1}$-invariant function $\widetilde u_g$ by setting
\be\label{eq:vg}
\widetilde u_g(\xi_0, \xi, p): = u_g'(0, \xi_{k+1},\cdots, \xi_n, p)
\ee
which is now clearly invariant under the translation in the additional direction of $\xi_k$.
It follows that $u_g$ is again contained in the given neighborhood $\CV_2$. 
Then we define
\be\label{eq:ghat}
\widehat g: = \mathscr G_A^{-1}(\widetilde u_g)
\ee
which is now $T^{k+1}$-equivariant and hence contained in 
$$
 \widehat g \in \Cont_{E_A^{(k+1)}}^{T^{k+1}}(\CW_k^{2n+1},\alpha_0)_0
 \subset \Cont_c^{T^{k+1}} (\CW_k^{2n+1},\alpha_0)_0.
$$
Here we remark that the domains of the contactomorphisms can be
naturally identified with
$$
 S^1 \times T^*(T^k \times \R^{n-k}),
$$
respectively for $0 \leq k \leq n$. In summary, each hat operation adds the $S^1$-equivariance in 
the one more direction of $\xi$'s.

\section{Rolling-up operator and unfolding-fragmentation operators } 
\label{sec:unfolding-fragmentation}
 
Denote by $\pi_k: \CW_k^{2n+1} \to \CW_{k+1}^{2n+1}$ the natural projection in the $q_k$
direction induced by the covering projection $\R \to S^1$.
We consider a contactomorphism $g$ of $\CW_{k+1}^{2n+1}$
contained  in a sufficiently $C^1$-small neighborhood $\CU_1$ 
of the identity with $\supp g \subset J_A^{(k)}$, i.e.,
 in $\Cont_{J_A^{(k)}}(\CW_k^{2n+1}, \alpha_0)$. More specifically, we will 
assume  $M_1^*(g) < \frac14$. 

\subsection{Mather's rolling-up operator}

We first recall Mather's rolling-up operators $\Theta_A^{(k)}$ in the current context
\be\label{eq:ThetakA-defn}
\Theta_A^{(k)}(g)(\theta_0, \ldots, \theta_{k-1}, q_{k+1}, \ldots, q_n, p) 
= \pi_k\left((T_k g)^N(z, q_1, \ldots, q_n, p)\right)
\ee
where we make a choice of $N$ as follows.
For any 
$x = (\theta_0,\cdots, \theta_{k-1},\ldots, q_n,p) \in \CW_{k+1}^{2n+1}$, 
we choose $\widetilde x \in \R^{n+1} \times \R^n$ with $\pi_{k+1}(\widetilde x) = x$
with $q_k < -2 A^2$ for the covering map $\pi_{k+1}: \R^{2n+1} \to \CW_{k+1}^{2n+1}$. 
Choose a sufficiently large $N \in \N$ so that 
$$
q_k\left((T_k f)^N(\widetilde x)\right) > 2 A^2.
$$
(It is easy to check that it is enough to choose any $N > 4A^2 +4A $
by starting with $\widetilde x$ with $q_k(\widetilde x) = - 2A^2 - 2A$.)

The following summarizes basic properties of Mather's rolling-up operators
\cite[Defintiion p. 520]{mather} applied to the contactomorphisms.

\begin{prop}[Compare with Proposition 8.1 \cite{rybicki2}]
Let $k = 0, \ldots, k$. After shrinking $\CU_1$ if necessary, 
$$
\Theta_A^{(k)}: \Cont(\CW_k^{2n+1},\alpha_0)_0 \cap \CU_1 \to 
\Cont_{K_A^{(k+1)}}(\CW_{k+1}^{2n+1},\alpha_0)
$$
satisfies the following properties:
\begin{enumerate}
\item $\Theta_A^{(k)}$ is continuous and preserves the identity.
\item $\Theta_A^{(k)}(\Cont^{T^k}(\CW_k^{2n+1},\alpha_0)_0 \cap \CU_1) \subset 
\Cont^{T^k}\left(\CW_{k+1}^{2n+1},\alpha_0\right)_0$, i.e., $\Theta_A^{(k)}(g)$ is 
also $S^1_{k}$-equivariant.
\end{enumerate}
\end{prop}

\subsection{Unfolding-fragmentation operators $\Xi_{A;N}^{(k)}$}

Rybicki \cite{rybicki2} also considers the following map (for $N =2$)
$$
\Xi_{A;N}^{(k)}: \Cont_{J_A^{(k+1)}}(\CW_{k+1}^{2n+1},\alpha_0) \cap \CU_2 \to 
\Cont_{K_A^{(k)}}(\CW_k^{2n+1},\alpha_0)_0, \quad k = 0, \ldots, n,
$$
where $\CU_2$ is a $C^1$-small neighborhood of $\id$ in $
\Cont_{J_A^{(k+1)}}(\CW_{k+1}^{2n+1},\alpha_0)$. This is a contact counterpart of
Mather's operator of `fragmentation followed by shifting supports' 
\cite[Construction p. 524]{mather}. We call the map
an \emph{unfolding-fragmentation operator in the $q_k$-direction}.
Its construction is now in order. This is where the construction of \cite{rybicki2} makes
a stark difference from that of \cite{mather} and the Legendrianization followed by taking
the contact potential plays a fundamental role in Rybicki's proof.

We need to give the general definition of $\Xi_{A;N}^{(k)}$ applied to the $N$-fragmentation,
while \cite[p.3313]{rybicki2} gave the construction only for
the case of the 2-fragmentation and stop short of giving the general definition 
associated to the $N$-fragmentation, although he implicitly employed the
defintion for the general case in the proof of \cite[Lemma 8.6]{rybicki2}.
(See the end of the proof in \cite[p. 3318]{rybicki2} and Remark \ref{rem:rybicki-error}
 of the present paper.)
To make clear the dependence on $N$ of
its definition, we denote by
$$
\Xi_{A;N}^{(k)}
$$
the one associated to the $N$-fragmentation with $N = 2,\, 3\, ...$, leaving the 
corresponding construction for $N = 2$ to \cite[Section 8]{rybicki2}.

We denote by $\rho:[0,1] \to [0,1]$ the standard boundary flattening function
once and for all that satisfies
\bea\label{eq:bdy-flattening-rho}
\rho(t) & = & \begin{cases} 1 \quad & \text{for }  0 \leq t \leq \frac14 \\
1 \quad & \text{for }  \frac34 \leq t \leq 0
\end{cases} \nonumber\\
\rho'(t) &\geq&  0.
\eea
We then extend the function to the interval $[-1,1]$ as an even function which we 
still denote by $\rho:[-1,1] \to [0,1]$.
We also define $\widetilde \rho$ to be the time-reversal $\widetilde \rho(t) = \rho(1-t)$.
We then  extend the function periodically to whole $\R$ with period 2.
 
Having the identification $S^1 = \R/\Z$ and the covering projection $\pi: \R \to S^1$ in our mind, 
we consider the $N$ pieces of subintervals of length 2 given by
\be\label{eq:Ij}
I_j = \left[-N + 2(j-1), -N + 2j\right], \quad j = 1, \ldots, N.
\ee
We denote the left and the right boundary points of $I_j$ by
$b_j^\pm$ respectively, i.e., 
\be\label{eq:bj+-}
b_j^- = -N + 2(j-1), \quad  b_j^+ = -N + 2j.
\ee
We write its concentric subinterval of the length 1
\be\label{eq:Ij'}
I_j' =  \left[-N + 2\left(j - \frac34\right), -N + 2\left(j - \frac14\right) \right] \subset  I_j
\ee
centered at 
\be\label{eq:center}
-N + 2j -1 =: c_j, \quad  j =1, \ldots, N.
\ee
In this regard, we will sometimes write them as
$$
I_j = I_j(c_j), \quad I_j': =\frac12 I_j(c_j)
$$
for the clarity of presentation. We also define 
\be\label{eq:Ak}
A_\varepsilon(c_k) =  [c_k -\varepsilon, c_k + \varepsilon].
\ee
Now we scale the intervals $[-N,N]$ down to $[-1,1]$ and define
\bea
I_j^N & = & \frac1N I_j , \quad (I_j')^N = \frac12 I_j^N, \label{eq:Ij-N}\\
 A_\varepsilon^N(c_k) &: =&  \frac1N A_\varepsilon(c_k). \label{eq:AkN}
 \eea

Since $\rho(t) \equiv 1$ for $t = 0 \mod 2$ or and
$\rho(t) \equiv 0$ for $t \equiv 1 \mod 2$, we can extend it to the whole $\R$ 
2-periodically.   Then we have
\be\label{eq:supp-psi}
\supp \rho \subset  \bigcup_{j = 1}^{N} (-N + 2j - 1) +  \left[-\frac34 ,\frac34 \right]  
= \bigcup_{j=1}^N  A^N_{\frac34}(c_{2j}).
\ee
Later it will be more convenient to enumerate the intervals symmetrically with respect to 
the origin so that the interval centered at the origin always appears at $j = 0$. In this ordering
we can write
\be\label{eq:ordering-union}
 \bigcup_{j=1}^N  A^N_{\frac34}(c_{2j}) =  \bigcup_{j = - [\frac{N+1}{2}]}^{[\frac{N+1}{2}]}
 A^N_{\frac34}\left(c_{2(j + [\frac{n+1}{2}] + 1)}\right) 
 \ee
 where $[b]$ is the largest integer smaller than or equal to a real number $b$.

We define a multi-bump function by setting
\be\label{eq:psi-k}
\psi_k^N (\xi_0,\xi, p): = \rho(N \xi_k), \quad k = 1, \ldots, n
\ee
respectively,  and lift it to a function defined on
$\CW_{k+1}^m$.  For each given 
$$
g \in \Cont_{J_A^{(k+1)}}(\CW_{k+1}^m, \alpha_0)_0 \cap \CU_2,
$$ 
we define
$$
g^{\psi_k}: = {\mathscr G}_A^{-1}(\psi_k^N {\mathscr G}_A(g)) 
= {\mathscr G}_A^{-1}(\psi_k^N u_g).
$$
Let $g_1^{\psi_k}$ (resp. $g_2^{\psi_k}$) be the unique lift of $(g^{\psi_k})^{-1} g$
(resp. $g^{\psi_k}$) to $\CW_k^m$ which are periodic contactomorphisms
on 
$$
T^k \times \R \times [-2A^2, 2A^2]^{n-k} \times [-2A^3, 2A^3]^n. 
$$
For the notational convenience, we set 
$$
\CE_{A,n,k} : =  [-2A^2, 2A^2]^{n-k} \times [-2A^3, 2A^3]^n.
$$
By definition, we have
\be\label{eq:fragmentation}
g = g_2^{\psi_k} g_1^{\psi_k}.
\ee
Furthermore, by construction, for a small enough $C^1$ neighborhood 
$\CU_2$, there is a sufficiently small $\varepsilon > 0$ such that 
\beastar
g_1^{\psi_k} & = & g  \quad \text{on $T^k \times I'_{2j-1}
\times \CE_{A,n,k}$ }\\
g_2^{\psi_k} & = & g  \quad \text{on $T^k \times I'_{2j}
\times \CE_{A,n,k}$ }
\eeastar
for all $j = 1, \ldots N$. We put
\beastar
E_k^-&  =  &\{(\xi_0,\xi,y) \in \CW_k^m \mid -1 \leq \xi_k \leq 0\} \\
E_k^+ & = &  \left\{(\xi_0,\xi,y) \in \CW_k^m \, \Big|\,  \frac12 \leq \xi_k \leq 3/2\right\}.
\eeastar

The following definition is the same as those of \cite[Construction of p.524]{mather},
\cite[p.119]{epstein:commutators} and \cite[Equation (8.2)]{rybicki2} (for $N=2$).

\begin{defn}[The map $\Xi_{A;N}^{(k)}$] \label{defn:XiAk}
We define $\Xi_{A;N}^{(k)}(g) : = f$ by requiring $f$ to satisfy the requirement:
$$
\pi_k f = \begin{cases} g_1^{\psi_k} \quad & \text{\rm on $E_k^-$}, \quad f(E_k^-) = E_k^-\\
g_2^{\psi_k} \quad & \text{\rm on $E_k^+$},  \quad f(E_k^+) = E_k^+
\end{cases}
$$
and
$$
f = \id \quad  \text{\rm on $\CW_k^m \setminus (E_k^- \cup E_k^+)$}.
$$
\end{defn}
This map is well-defined since $E_k^+ \cap E_k^- = \emptyset$ and by the
1-periodicity of the multi-bump function $\psi_k^N$. 
The following conservation of supports from $g$ to $u_g$ and vice versa are important in 
the comparison study of $g$ and the associated $u_g$.
\begin{lem}\label{lem:support-preserving}
\be\label{eq:support-comparison}
\supp (g-\id) \cup \supp \ell_g = \supp j^1u_g.
\ee
\end{lem}

The following properties of $\Xi_{A;N}^{(k)}$ are
immediate to check which will be used later. (See \cite[p.3313]{rybicki2} 
for a simpler relevant discussion thereon for the case of $2$-fragmentation.)

\begin{lem}\label{lem:pikXiAk} For any given sufficiently small $\varepsilon  > 0$,
there exists some $\delta > 0$ depending only on $\varepsilon$ such that
\be\label{eq:pikXiA}
\pi_k \Xi_{A;N} ^{(k)}(g) = g \quad \text{\rm on } \, 
A^N_\varepsilon \times \R^{2n}
\ee
for all $g \in \Cont_{J_A^{(k+1)}}(\CW_{k+1}^m, \alpha_0)_0 \cap \CU_2$ with 
$M_1^*(g) \leq \delta$. Here $A^N_\varepsilon$ is the union
\be\label{eq:AeN}
A^N_\varepsilon = \bigcup_{j=1}^{2N} A^N_\varepsilon(c_j)
\ee
where the interval $A^N_\varepsilon(c_j)$ is as
introduced in \eqref{eq:AkN}.
\end{lem}
\begin{proof} By the definitions of $\mathscr G_A(g)$ and $g^{\psi_k}= g^{\psi_k^N}$, we have
$$
g^{\psi_k^N} = \mathscr G_A^{-1}(\psi_k^N \mathscr G_A(g)).
$$
If $\xi_k \in A^N_{2j}(\frac{1}{8N})$, then we have $\psi_k^N(\xi_k) = 1$ by definition of $\psi$ and so
$$
g^{\psi_k}(x) = \mathscr G_A^{-1} (u_g(x)).
$$
By definition, we have $\Phi_{U;A} \circ \Delta_{\id}(x) = (x,0,0)$. It follows from 
the definition of the map $h_X$ given in Corollary \ref{cor:equivariance} that we
 have the estimates
$$
\|Du_g\| \leq C_1 \|D(g - id)\|_{C^0} + C _2 \|h_X\|_{C^0}
$$
where $C_i$ depend only on the Darboux-Weinstein chart $\Phi_{U}$, $A$ and $N$ but 
are independent of $g$, provided $\|M_1^*(g)\|$ is sufficiently small. We also have
$$
\max\{\| g -\id\|, \|Df\|, \|\ell_g\|\} \leq C_3 ( \|Du_g\|_{C^0} + \|u_g\|_{C^0})
$$
and hence follows that $M_0^*(g) = \max\{\|g - \id\|_{C^0}, \|\ell_g\|_{C^0}\}$ 
and that $\|j^1u_g\|_{C^0}$ are comparable to each other. 

On the other hand, on $A^N_{2j-1}(\frac{1}{8N})$, we have
$\psi_k(x) = 0$ on $\cup I'_{2j-1}$ and $\psi_k^N(x) = 1$ on $\cup I'_{2j}$
and so $g^{\psi_k^N}(x) =  g$. Then by 
Lemma \ref{lem:support-preserving} and combining the above discussion, 
we have proved the lemma
by continuity of the map $\mathscr G_A$, the definition of $\psi$ in \eqref{eq:psi-k}
 and the support identity \eqref{eq:support-comparison}.
\end{proof}

\section{Rolling-up contactomorphisms}
\label{sec:rolling-up} 

By now, we have constructed  the map
$$
{\Theta_A^{(k)}} : \Cont_{J_A^{(k)}} (\CW_k^{2n+1},\alpha_0)_0 \cap \CU_1  
\longrightarrow  \Cont_{J_A^{(k+1)}} (\CW_{k+1}^{2n+1},\alpha_0)_0
$$
and 
$$
\Xi_{A;N}^{(k)}:  \Cont_{J_A^{(k+1)}} (\CW_{k+1}^{2n+1},\alpha_0)_0    \cap \CU_2
\longrightarrow  \Cont_{K_A^{(k)}}(\CW_k^{2n+1},\alpha_0)_0  
$$
for each integer $N \geq 2$ by suitably choosing $\CU_1$ and $\CU_2$.  

\subsection{Properties of $\Theta_A^{(k)}$ and $\Xi_{A;N}^{(k)}$}

We also have  the inclusion map 
 $$
 \Cont_{K_A^{(k)}}(\CW_k^{2n+1},\alpha_0)_0 \hookrightarrow \Cont_{J_A^{(k)}} (\CW_k^{2n+1},\alpha_0)_0
 $$
 since $K_A^{(k)} \subset J_A^{(k)}$,
 and hence we may canonically regard the map $\Xi_{A;N}^{(k)}$ also as a map
 $$
 \Xi_{A;N}^{(k)}:  \Cont_{J_A^{(k+1)}} (\CW_{k+1}^{2n+1},\alpha_0)_0    \cap \CU_2
\longrightarrow \Cont_{J_A^{(k)}}(\CW_k^{2n+1},\alpha_0)_0,
$$
especially $\Cont_{J_A^{(k)}}(\CW_k^{2n+1},\alpha_0)_0$ as its codomain.
 We may choose $\CU_2$ so small and then choose $\CU_1$ so
 that the composition of the two maps are defined. We will choose $\CU_2$ so small 
 that the map $ \Xi_{A;N}^{(k)}$ is defined and that we have a commutative diagram
\be\label{eq:nonlinear-sequence}
\xymatrix{
\Cont_{J_A^{k}} (\CW_k^{2n+1},\alpha_0)_0    \cap \CU_1 
\ar[r]^{\Theta_A^{(k)}}  &\Cont_{J_A^{(k+1)}} (\CW_{k+1}^{2n+1},\alpha_0)_0 \\
\Cont_{K_A^{k}} (\CW_k^{2n+1},\alpha_0)_0    \cap \CU_1 \ar@{^{(}->}[u] 
& \Cont_{J_A^{(k+1)}} (\CW_{k+1}^{2n+1},\alpha_0)_0    \cap \CU_2 
\ar[l]^{\Xi_{A;N}^{(k)}} \ar[u]_{\Theta_A^{(k)}\circ \Xi_{A;N}^{(k)}}
}
\ee
where the left vertical arrow map is just the inclusion map induced by the 
inclusion map  $K_A^k \subset J_A^k$.

The following, especially the \emph{strict equality} of the composition in Statement (3), 
plays an important
role in Rybicki's construction $\vartheta$ of the contact counterpart of Mather's 
rolling-up operator $\theta_f$ that appears in the proof of 
 \cite[Theorem 2, p.518]{mather}.

 \begin{prop}[Compare with Proposition 8.2 \cite{rybicki2}]\label{prop:rybicki82}
  Taking $\CU_2$ and then $\CU_1$
 sufficiently small, we have the following:
 \begin{enumerate}
 \item $\Xi_{A;N} ^{(k)}$ is continuous and preserves the identity.
 \item $\Xi_{A;N}^{(k)}\left( \Cont_{J_A^{(k+1)}}^{T^k}(\CW_{k+1}^{2n+1},\alpha_0)_0\right) \subset 
 \Cont_{K_A^{(k)}}^{T^k}(\CW_k^{2n+1},\alpha_0)_0$.
 \item  We have $\Theta_A^{(k)}\Xi_{A;N}^{(k)}(g) = g$  for any 
 $ g \in \text{\rm Dom}(\Xi_{A;N}^{(k)})$. 
 \end{enumerate}
 \end{prop}
\begin{proof}  Statements (1), (2) are straightforward to check. We focus on the proof of
(3).
Write 
$$
f: = \Xi_{A;N}^{(k)}(g) \in \Cont_c^k(\CW_k^m,\alpha_0)_0 \cap \CU_1
$$
Let $x = (\xi_0,\xi_1, \ldots, \xi_{k-1}, \xi_k, \xi_{k+1}, \ldots, \xi_n, p) \in \CW_{k+1}^m$ and 
$$
\widetilde x = (\xi_0,\xi_1, \ldots, \xi_{k-1}, q_k,  \xi_{k+1}, \ldots, \xi_n, p) \in \CW_k^m
$$
with $q_k \in \R$ satisfying $q_k = \xi_k \mod 1$. We take $q_k < - 2 A^2$. 
Then we have
$$
\Theta_A^{(k)}(f)(x): = \pi_k(((T_k f)^N(\widetilde x))) 
$$
by the definition of $\Theta_A^{(k)}$.   

We consider the three cases separately: 
$$
\widetilde x \in E_k^-, \quad \widetilde x \in E_k^+, \quad 
\widetilde x \in \CW_k^m \setminus E_k^- \cup E_k^+.
$$
 By definition,
this trichotomy is preserved by the application of $f: \CW_k^m \to \CW_k^m$.

For the last case, we have
\beastar
\pi_k(((T_k f)^N(\widetilde x)))) & = &  \pi_k(\vec e_k + f (T_k f)^{N-1}(\widetilde x)) \\
& = &  \pi_k(f (T_k f)^{N-1}(\widetilde x)) = h_{\star_1}\pi_k((T_k f)^{N-1}(\widetilde x))
\eeastar
where $ h_{\star_1}$ is the map given by
\be\label{eq:h*1}
h_{\star_1}(\widetilde y) = \begin{cases} h_0(\widetilde y) \quad &\text{for $\widetilde y \in E_k^-$} \\
h_1(\widetilde y) \quad & \text{for $\widetilde  y\in E_k^+$}\\
\widetilde y \quad & \text{for $\widetilde y \in \CW^m_k \setminus (E_k^- \cup E_k^+)$}
\end{cases}
\ee
depending on the location of  $\widetilde y: = (T_k f)^{N-1}(\widetilde x)$. (Recall 
where $h_0 = g_1^{\psi_k}$ and $h_1 = g_2^{\psi_k}$.)

Similarly we define $h_{\star2}$ so that
$$
\pi_k((T_k f)^{N-1}(\widetilde x)) = h_{\star_2} \pi_k((T_k f)^{N-2}(\widetilde x)).
$$
By repeating this argument inductively, we have obtained
$$
\pi_k(((T_k f)^N(\widetilde x)))) = h_{\star_N} \cdots h_{\star_1}(\widetilde x).
$$
By the property 
$$
\quad f(E_k^-) = E_k^-,  \quad f(E_k^+) = E_k^+, \quad f(\CW_k^m \setminus E_k^+ \cup E_k^-)
$$
and $(T_j f(\widetilde x))_k  > (T_{j-1}f(\widetilde x))_k$,
it follows that all $h_{\star_\ell} = id$ except possibly for at most two $\ell$'s,
and hence $h_{\star_N} \cdots h_{\star_1} = h_2 h_1 = g$. This finishes the proof.
\end{proof} 

\begin{rem}\label{rem:homotopy-to-identity}
\begin{enumerate}
\item The property $\Theta_A^{(k)}\circ \Xi_{A;N}^{(k)} = \id$ is a fundamental ingredient in
Mather's construction in general. (See \cite{mather}-\cite{mather4},
\cite{epstein:commutators}, especially \cite{mather3} for a detailed analysis on its implication.)
\item
While the composition $\Theta_A^{(k)}\circ \Xi_{A;N}^{(k)}$ is the identity map,
 we will show that the other composition $ \Xi_{A;N}^{(k)} \circ  \Theta_A^{(k)}$
is not the identity map on the nose, but will be `homotopic to the identity'.
\end{enumerate}
\end{rem}
 
 We also state the following two lemmata from \cite{rybicki2}.
 
  \begin{lem}[Lemma 8.3 \cite{rybicki2}]\label{lem:[f]=[g]}
   If $f, \, g \in \text{\rm Dom}(\Theta^{(k)})$ and 
 $\Theta^{(k)}(f) = \Theta^{(k)}(g)$,
 then $[f] = [g]$ in $H_1(\Cont_c(\CW_{k+1}^{2n+1}, \alpha_0)_0)$.
 \end{lem}
 \begin{proof}
 See the proof of \cite[Lemma 8.3]{rybicki2}.
 \end{proof}
 
 We remark that  when $g_i = \Theta^{(k)}(f_i)$ for $i=1, \ldots, \ell$, its product
 $g_1\cdots g_\ell$ may not be contained in the image of $\Theta^{(k)}$ and the equality
 $$
 \Theta^{(k)}(f_1 \cdots f_\ell) = g_1\cdots g_\ell
 $$
 fails to hold in general.
 In other words, the map $\Theta^{(k)}$ is \emph{not multiplicative} on the nose.
 But the following lemma shows that it is multiplicative in homology.
 
 \begin{lem}[Lemma 8.4 \cite{rybicki2}]\label{lem:rybicki2}Let $k =1, \ldots, n$.
 \begin{enumerate}
 \item Suppose $g_i = \Theta^{(k)}(f_i)$ for $i=1, \ldots, \ell$. Then there exists a
 collection of  $\overline f_i$ such that $\{\supp \overline f_i\}$ are disjoint, 
 $[\overline f_i] = [f_i]$ in $H_1(\Cont_c(\CW_k^m, \alpha_0)_0)$, and satisfies
 $$
 g_1 \cdots g_\ell = \Theta^{(k)}\left(\overline f_1\cdots \overline f_\ell\right)
 $$
 \item If $g_1, \, g_2, \, g_1g_2 \in \text{\rm Dom}(\Xi^{(k)})$, then we have
 $$
 [\Xi^{(k)}(g_1g_2)] = [\Xi^{(k)}(g_1) \Xi^{(k)}(g_2) ]
 $$
  in $H_1(\Cont_c(\CW_k^m, \alpha_0)_0)$.
 \item If $[g] = e$  in $H_1(\Cont_c(\CW_k^m, \alpha_0)_0)$, there is 
 $f \in \Cont_c(\CW_k^m, \alpha_0)_0$ such that $\Theta^{(k)}(f) = g$ and $[f] = e$
 in $H_1(\Cont_c(\CW_k^m, \alpha_0)_0)$.
 \end{enumerate}
 \end{lem} 
 \begin{proof} See the proof of \cite[Lemma 8.4]{rybicki2}.
 \end{proof}

 \subsection{The hat operation applied}
 
Towards the verification of the claim made in Remark \ref{rem:homotopy-to-identity},
 we will first use the \emph{hat operation} to deform the composition map
 $$
 \Xi_{A;N}^{(k)} \circ  \Theta_A^{(k)} : \Cont_{J_A^{(k)}} \left(\CW_k^{2n+1},\alpha_0\right)_0 \cap \CU_1
 \to \Cont_{K_A^{(k)}}\left(\CW_k^{2n+1},\alpha_0\right)_0
 $$
 so that we can define our wanted `auxillary rolling-up operator'
 $$
\Psi_{A;N}^{(k)} : \Cont_{J_A^{(k)}}^{T^k}\left(\CW_k^{2n+1},\alpha_0\right)_0 \cap \CU_1
 \to \Cont_{K_A^{(k)}}^{T^k}\left(\CW_k^{2n+1},\alpha_0\right)_0 
 $$
\emph{that is equivariant with respect to the $T^k$ actions}
on the domain and on the codomain thereof.
 Recall $\text{\rm Dom}\left(\Xi_{A;N}^{(k)}\right)
 = \Cont_{J_A^{(k+1)}} (\CW_{k+1}^{2n+1},\alpha_0)_0    \cap \CU_2$.
  
 Now, we apply the hat operation explained in Section \ref{sec:hat-operation} to
 the function $\Theta_A^{(k)}(g) = u_g$,  and define a map
 $$
\widehat {\Theta_A^{(k)}} : \Cont_{J_A^{(k)}} ^{T^k}\left(\CW_k^{2n+1},\alpha_0\right)_0 \cap \CU_1  
 \longrightarrow  \Cont_{J_A^{(k+1)}}^{T^{k+1}}\left(\CW_{k+1}^{2n+1},\alpha_0\right)_0 
 $$
by putting
 \be\label{eq:hatThetaAk}
 \widehat\Theta_A^{(k)}(g): = \widehat{\Theta_A^{(k)}(g)}.
 \ee
 
Then we introduce Rybicki's axillary rolling-up operators $\Psi_{A;N}^{(k)}$.
 
 \begin{prop}[Proposition 8.5 \cite{rybicki2}]\label{prop:PsiAk}
  Let $r \geq 2$ and $k = 0, \ldots, n$. There 
 exists a neighborhood $\CU_3 \subset \CU_1 \subset \Cont_c(\CW_k^m, \alpha_0)_0$
 and a map $\Psi_{A;N}^{(k)}$ such that the map
 $$
 \Psi_{A;N} ^{(k)}:\Cont_{J_A^{(k)}}^{T^k}\left(\CW_k^m, \alpha_0\right)_0  \cap \CU_3 \to 
 \Cont_{K_A^{(k)}}^{T^k}\left(\CW_k^m, \alpha_0\right)_0
 $$
 satisfies the following:
 \begin{enumerate}
 \item $\Psi_{A;N}^{(k)}(\id) = \id$.
 \item For any $g \in \text{\rm Dom}(\Psi_{A;N}^{(k)})$, we have
 $$
 [\Psi_{A;N}^{(k)}(g) \cdot  \Xi_{A;N} ^{(k)} \widehat{\Theta}_A^{(k)}(g)] = [g]
 $$
 in $H_1(\Cont_c(\CW_k^m,\alpha_0)$.
 \end{enumerate}
 \end{prop}
 \begin{proof}  For readers' convenience, 
 we just recall the definition of $\Psi_{A;N}^{(k)}$ leaving the verification of its properties 
 stated in this proposition to  the proof of \cite[Proposition 8.5]{rybicki2}. 
 
 Let $g \in \Cont_{J_A^{(k)}}(\CW_k^m,\alpha_0) \cap \CU_1$. Then we define
 \be\label{eq:PsiAk-defn}
 \Psi_{A;N}^{(k)}(g): = \Xi_{A;N}^{(k)}\left(\Theta_A^{(k)}(g) \cdot ( \widehat \Theta_A^{(k)}(g))^{-1}\right).
 \ee
  \end{proof}

Finally, we are ready to define Rybicki's
 contact rolling-up operator $\Psi_A = \Psi_{A;N}$.
Keeping the definition depends on the integer $N$ in mind, we will sometimes omit $N$
from the notations of $\Xi_{A;N}$ or $\Psi_{A;N}$ unless the dependence on $N$ needs to be emphasized.

 For the simplicity of notation, we introduce the following notations: 
 \bea
 \Theta_A^{(k >)} & := & \Theta_A^{(k)} \circ \cdots \circ \Theta_A^{(0)}, \quad
 \widehat \Theta_A^{(k >)} := \widehat \Theta_A^{(k)} \circ \cdots \circ \widehat \Theta_A^{(0)}
 \label{eq:Theta-k>} \\
 \Xi_{A;N}^{(< k)} & := & \Xi_{A;N}^{(0)} \circ \cdots \circ \Xi_{A;N}^{(k)} \label{eq:Xi-<k}
 \eea
 We state the half of \cite[Lemma 8.6]{rybicki2} separately, the proof of which we
  refer readers thereto.
 \begin{lem}[Lemma 8.6 (1) \cite{rybicki2}]\label{lem:rybicki86(1)} Assume $A$ and $N$ are
 given.
 Let $\CU_3$ be a sufficiently $C^1$-small neighborhood of
 the identity in $\Cont_c(\R^{2n+1}, \alpha_0)_0$. Then for all $g \in \CU_3$
 $$
 \left[\Xi_{A}^{(<n)} \Theta_A^{(n>)}(g)\right] = [g]
 $$
 in  $H_1(\Cont_c(\R^{2n+1}, \alpha_0)_0$.
\end{lem}
(Statement of this lemma looks different from that of \cite[Lemma 8.6 (1)]{rybicki2} (which is stated only for $a = 2$, though)
but equivalent  if $\Xi^{(<n)}( = \Xi^{(<n)}_A)$ therein is replaced by the current $\Xi_{A;N}^{(<n)}$
which should depend on $N$.)

\begin{prop}[Compare with Proposition 8.7 \cite{rybicki2}]\label{prop:main} 
Let $r \geq 2$. There exists a 
sufficiently small $C^1$ neighborhood $\CU_4 = \CU_{\chi,r,A}$ in $\Cont_c(\R^{2n+1},\alpha_0)_0$
and a mapping, called the contact rolling-up operator,
$$
\Psi_A: \Cont_{J_A}(\R^{2n+1},\alpha_0)_0 \cap \CU_4 \to \Cont_{K_A}(\R^{2n+1},\alpha_0)_0
$$
that satisfies the following:
\begin{enumerate}
\item $\Psi_A$ is continuous and $g(\id) = \id$.
\item For any $g \in \text{\rm Dom}(\Psi_A)$, $[\Psi_A(g)] = [g]$ in $H_1(\Cont_c(\R^{2n+1},
\alpha_0)_0$.
\end{enumerate}
\end{prop}
\begin{proof} The proof is essentially a duplication of  the proof given in course of the proof of
\cite[Proposition 8.7]{rybicki2} but we give details of the geometric construction
for the self-containedness of the proof here.

Let $g \in \Cont_{J_A}(\R^m, \alpha_0)_0 \cap \CU_1$. Define 
\be\label{eq:Psi-A}
\Psi_A(g) = g_0 g_1 \cdots g_n
\ee
where $g_0 = \Psi_A^{(0)}(g)$ and 
\be\label{eq:gk}
g_k = \Xi^{(<k-1)}\Psi_A^{(k)}\widehat{\Theta_A}^{(k-1)}(g), \quad k=1, \ldots, n.
\ee
Statement (1) is apparent by definition.

Now we inductively evaluate
\beastar
[\Psi_A(g)] & = & [g_0 g_1 \cdots g_n] \\
& = & [g_0 g_1 \cdots g_n \cdot \Xi_A^{(< n)}\widehat{\Theta}_A^{(n>)}(g)] \\
& = & \left[g_0 g_1 \cdots g_{n-1}\left(\Xi_A^{(<n-1)}\Psi_A^{(n)}\widehat{\Theta}_A^{(n-1>)}(g)\right)\cdot \left(\Xi_A^{(<n-1)}\Xi_A^{(n)}
\widehat{\Theta}_A^{(n)}\widehat{\Theta}_A^{(n-1>)}(g)\right)\right]\\
& = & \left[g_0 g_1 \cdots g_{n-1}\cdot \Xi_A^{(<n-1)}\left(\Psi_A^{(n)}\left(\widehat{\Theta}^{(n-1>)}_A(g)\right) \cdot \Xi_A^{(n)}
\widehat{\Theta}_A^{(n)}\left( \widehat{\Theta}_A^{(n-1>)}(g)\right)\right)\right]\\
& = &  \left[g_0 g_1 \cdots g_{n-1}\cdot  \cdot\Xi_A^{(<n-1)} \widehat{\Theta}_A^{(n-1>)}(g)\right]
\eeastar
where we apply Lemma \ref{lem:rybicki86(1)} for the second equality, the definition \eqref{eq:gk} for the third,
and Proposition \ref{prop:PsiAk} (2) for the fifth equality. By repeating this process inductively downward from $k=n$ 
to $k = 0$, we have derived
$$
[\Psi_A(g)] = \left[g_0 \cdot \Xi^{(0)}_A \widehat{\Theta}^{(0)}_A(g)\right] = 
\left[\Psi_A^{(0)}(g)\cdot \Xi_A^{(0)} \widehat{\Theta}^{(0)}_A\right] = [g]
$$
where we apply the definition of $g_0$ and then again Proposition \ref{prop:PsiAk} (2) for the
third equality. Combining the two, we have finished the proof of Statement (2).
\end{proof}

\part{Optimal $C^r$ estimates on contactomorphisms}

In this part, we prove all the necessary $C^r$ estimates of the various maps
appearing in the proof of the main theorem. 

Before starting the estimates, we first recall the remark made by Mather himself in
the beginning of \cite[Section 6]{mather} almost verbatim, except the change of 
the covering projection $\R^m \to \CC_i$ therein by  the covering projection 
$\R^m \to \CW_k^m$:
``\emph{The projection mapping $\R^{2n+1} \to \CW_k^{2n+1}$ with 
$\CW_k^{2n+1}= S^1 \times T^*(T^{k-1} \times \R^{n-k+1})$ 
gives us a preferred system of
coordinates in a neighborhood of any point of $\CW_k^{2n+1}$. The transition mappings between different
coordinate systems which we obtain in this way are all translations. It follows that the $r$th derivative of any
$C^r$ mapping of $\CW_k^{2n+1}$ into itself is defined independently of the choice of 
preferred coordinate system. The $r$th derivative of such a mapping $v$ is a mapping $D^r v: \CW_k^{2n+1} \to SL^r(\R^{2n+1},\R^{2n+1})$
of $\CW_k^{2n+1}$ into the space of symmetric $r$-linear mappings of $\R^{2n+1}$ into itself.}''

The same kind of practice enables us to work mainly with the Euclidean space
$\R^{2n+1}$ in all the following estimates.

\section{Basic $C^r$ estimates on the spaces $\Diff_c^r(\R^m)$; summary}
\label{sec:estimates-summary}

In this section, we collect some function spaces, norms and basic 
estimates of them on the Euclidean spaces or on the cylinders over tori.
We start those used by Mather \cite{mather} and Epstein \cite{epstein:commutators}
associated with the general diffeomorphism groups, and then go to the case of
contactomorphisms as used by Rybicki \cite{rybicki2} afterwards in the next section.

Let $f: U \to \R^m$ be a $C^r$-function, where $U$ is an open subset of $\R^n$. We define
$$
\|f\|_r: = \sup_{x \in U} \|D^rf(x)\|.
$$
We also consider maps between open subsets of spaces like $S^i \times \R^{n-i}$ with $0 \leq i \leq n$.

If $r \geq 1$ and $f$ is a diffeomorphism, we write
$$
M_r(f) = \sup\{\|f -\id\|, \|f\|_1, \ldots, \|f\|_r\}.
$$
If ${\bf f} = (f_1, \ldots, f_k)$ is a $k$-tuple of $C^r$-diffeomorphisms, we write
$M_r({\bf f}) = \sup_{1 \leq i \leq k} M_r(f_i)$.

We recall the formula
$$
D(f \circ g) = (Df \circ g)\cdot (Dg)
$$
where the right hand side is a composition of two linear maps, or a
matrix multiplication of $n \times n$ matrices. For the higher derivatives, we have
\bea\label{eq:higher-derivative-product}
D^r (f\circ g) & = & (D^r f)(Dg \times \cdots \times Dg) + (Df \circ g)(D^rg) \nonumber\\
&{}&  + \sum C(i;j_1, \ldots, j_i) (D^if \circ g)(D^{j_1}g \times \cdots \times D^{j_i} g)
\eea
where $C(i;j_1, \ldots, j_i)$ is an integer which is independent of $f$, $g$ and 
even of dimensions of their domains and codomains for
$$
1 < i < r, \quad j_1 + \ldots + j_i = r, \quad j_s \geq 1.
$$
We recall that $(D^if \circ g)$ is a \emph{multilinear map} of $i$ arguments.
\emph{This implies that we have $j_s \geq 2$ at least one $s$}. For the simplicity of notation, we write
$$
D^Jg = (D^{j_1} \times \cdots \times D^{j_i}) g, \quad J = (j_1, \ldots, j_i).
$$
Then we can write
\be\label{eq:Drfg}
D^r(f \circ g) = (D^r f)(Dg \times\cdots \times Dg) + (Df \circ g)(D^rg)
+ \sum C(i;J) (D^if \circ g)(D^J g).
\ee
We see that
\be\label{eq:M1fg}
M_1(f\circ g) \leq M_1(f)(1 + M_1(g)) + M_1(g)
\ee
by writing $f \circ g -\id = (f -\id) \circ g + (g - id)$.

\begin{defn}[Admissible polynomial] A polynomial is called \emph{admissible} if 
its coefficients are non-negative integers, and \emph{has no constant or linear terms}.
\end{defn}
We will denote an admissible polynomial by $F_{(*)}(x_1, \ldots x_\ell)$ in general where
$(*)$ denotes the set of parameters into the coefficients of the polynomial, e.g.,
$$
(*) = \{r,k,\cdots\}
$$
where $r$ is the order of differentiation $D^r$ and $k$ is the order of composition
as $f_1 \circ \cdots \circ f_k$ and etc. The polynomial will vary depending on the circumstances
that will appear later in the various estimates we carry out. We will (locally) 
enumerate them when we need to locally introduce several of them at the same time.

We derive the inequality
\be\label{eq:r-norm-product}
\|f\circ g\|_r \leq \|f\|_r(1 + M_1^*(g))^r + \|g\|_r(1 + M_1(f))^r 
+ F_{1;r}(M_{r-1}(f), M_{r-1}(g))
\ee
where $F_{1,r}$ may be taken to be zero.

\begin{prop}[Proposition 1.6 \cite{epstein:commutators}]\label{prop:r-norm-product} For each $ r\geq 2$ and $k \geq 2$,
there is an admissible polynomial
 $F_{2,k,r}$ of one variable with the following property. 
 For ${\bf f} = (f_1, \ldots, f_k)$ which is composable, we have
\be\label{eq:r-norm-products}
\|f_1, \circ \cdots \circ f_k\|_r \leq k \|{\bf f}\|_r (1 + M_1({\bf f}))^{r(k-1)}
+ F_{2;k,r}(M_{r-1}({\bf f})).
\ee
Moreover
\be\label{eq:1-norm-products}
\|f_1 \circ \cdots \circ f_k\|_1 \leq k M_1({\bf f}) (1 + M_1({\bf f}))^{k-1}.
\ee
\end{prop}

The following is a slight variation of \cite[Lemma 1.7]{epstein:commutators} with the replacement 
of $M_1(f) < \frac12$ by $M_1(f) < \frac14$.

\begin{prop}[Compare with Lemma 1.7 \cite{epstein:commutators}] For each $ r\geq 2$ and $k \geq 2$,
there is an admissible polynomial
$F_{3,r}$ of one variable with the following property. Let $f$ be a diffeomorphism of $\R^n$ 
satisfying $M_1(f) < \frac14$ and let $r \geq 2$. Then
\be\label{eq:r-norm-inverse}
\|f^{-1}\|_r \leq (1 + M_1(f))^{2(r+1)}\|f\|_r + F_{3;r}(M_{r-1}(f)).
\ee
Also
\be\label{eq:1-norm-inverse}
\|f^{-1}\|_1 \leq M_1(f)(1 + M_1(f))^2 \leq 2 M_1(f).
\ee
\end{prop}

\begin{rem}\label{rem:choice-varepsilon} In general, if we choose $M_1^*(f) < \frac1N$, then 
there exists $\delta = \delta(N) \to 1$ as $N \to \infty$ such that
$$
\|f^{-1}\|_r \leq (1 + M_1(f))^{\delta(N) (r+1)}\|f\|_r + F_{3;r}(M_{r-1}(f)).
$$
Also
$$
\|f^{-1}\|_1 \leq M_1(f)(1 + M_1(f))^2 \leq \delta(N)  M_1(f).
$$
In particular, by letting $N \to \infty$, we can make $\delta(N)$ as close to 1 as we want.
\end{rem}

 \section{$C^r$ estimates of contactomorphisms of the products}
 
 For notational convenience, we will
 use the following notations systematically following those of \cite{rybicki2}.
 
Let $E \subset \R^m$ be a closed subset. We define
\be\label{eq:RE}
R_E: = \sup_{x \in E}\text{\rm dist}\left(x, \overline{\R^m \setminus E}\right) \leq \infty.
\ee
For any $f \in \Cont_c(\R^{2n+1},\alpha_0)$ and $r \geq 0$, we put
\be\label{eq:mur*}
\mu_r^*(f): = \max\{\|D^r(f-id)\|, \|D^r \ell_f\|\}
\ee
and
\be\label{eq:Mr*}
M_r^*(f) = \max\{ \mu_0^*(f),\mu_1^*(f),\ldots, \mu_r^*(f)\}.
\ee
We consider the contact cylinders $\CW_k^m = T^k \times \R^{m-k}$ with $m = 2n+1$, and 
$$
E_A^{(k)}= T^k \times [-A,A]^{m-k}, \quad k = 1,\ldots, n+1.
$$
By the remark of Mather \cite{mather} recalled in the beginning of this part, we can do 
the estimates on $\R^m$ which we will focus on.

\subsection{$C^r$ estimates of conformal exponents of the products}
 
\begin{nota} For a given multi-index $\underline{k}: = (1, \ldots, k)$, we write
 $$
 g_{\underline{k}}^\circ: = g_k \circ g_{k-1} \circ \cdots \circ g_1.
 $$
 We  denote by ${\bf g}$ the tuple $(g_1, \cdots, g_k)$ and 
 by ${\bf g}_{;\ell}$ the sub-tuple of ${\bf g}$ given by
 $$
 {\bf g}_{;\ell}: = (g_1, \cdots, g_\ell), \quad 1 \leq \ell \leq k.
 $$
 Then we put
 $$
 \|{\bf g}\|_r : = \max_{i=1}^k  \|g_i\|_r
 $$
 for each given  such a tuple, and 
 \beastar
 \ell_{\bf g} & : = & (\ell_{g_1}, \ldots, \ell_{g_k}),\\
 \|\ell_{\bf g}\|_r & : = & \max_{i=1}^k \|\ell_{g_i}\|_r.
 \eeastar
 \end{nota}
 With these notations set up, we state the following basic $C^r$ estimates on 
 the conformal exponents of the products.
 
\begin{lem} Let $g_1, \ldots, g_m$ be a set of contactomorphisms
associated to  the multi-index $I = (1,2,\cdots, m)$.
Then for $r, \, |I| \geq 2$, we have
\bea\label{eq:Dr-gpsiI}
M_r^*({\bf g}) &\leq&  r \|\ell_{\bf g}\|_r \left(1 + M_1^*({\bf g})\right) + r M_r^*({\bf g})
\left(1 + \|\ell_{\bf g}\|_1\right)^r \nonumber \\
&{}& + F_{1,r}(\|\ell_{\bf g}\|_{r-1}, \{M^*_{r-1}({\bf g}_{;j} \}_j)
\eea
for a multi-variable admissible polynomial $F_{1,r}$, and
\bea\label{eq:Dr-gpsi-1}
\|\ell_{g^{-1}}\|_r &\leq&  \|\ell_g\|_r \left(1 + M_1^*(g^{-1})\right)^r + \|g^{-1}\|_r(1 + \|\ell_g\|_1)^r
\nonumber \\
&{}& +  F_{2,r}\left(\|\ell_g\|_{r-1}, M_{r-1}(g^{-1})\right)
\eea
for a 2 variable admissible polynomial $F_{2,r}$.
\end{lem}
\begin{proof} For the proof of \eqref{eq:Dr-gpsiI}, we first use \eqref{eq:gphin} to write
$$
\ell_{g_{\underline{k}}^\circ} = \sum_{j=1}^k \ell_{g_j}\circ g^\circ_{\underline{j}}.
$$
Then we derive
$$
D^r(\ell_{g_{\underline{m}}^\circ}) = \sum_{j=1}^m D^r(\ell_{g_j}\circ g^\circ_{\underline{j}}).
$$
We then apply \eqref{eq:Drfg} and obtain
\beastar
D^r(\ell_{g_j}\circ g^\circ_{\underline{j}}) & \leq &
\|\ell_{g_j}\|_r \left(1 + \|g^\circ_{\underline{j}}\|_1\right)^r + \|g^\circ_{\underline{j}}\|_r (1 + \|\ell_{g_j}\|_1)^r \\
&{}& 
+ F_{1,r,j} \left(\|\ell_{g_j}\|_{r-1}, M_{r-1}(g^\circ_{\underline{j}})\right).
\eeastar
By summing this over $1 \leq j \leq k$ and using the inequalities
$\|\ell_{{\bf g}_{;j}}\| \leq \|\ell_{\bf g}\|$ and 
$$
\|g^\circ_{\underline{j}}\|_\ell \leq M_r^*({\bf g}_{;j}),
$$
we get
\beastar
M_r^*({\bf g}) & \leq & r \|\ell_{\bf g}\|_r \left(1 + M_1^*({\bf g})\right) + r M_r^*({\bf g})
\left(1 + \|\ell_{\bf g}\|_1\right)^r \nonumber \\
&{} & + \sum_{j=1}^k F_{1,r,j}(\|\ell_{\bf g}\|_{r-1}, M^*_{r-1}({\bf g}_{;j})).
\eeastar
Setting a multi-variable admissible polynomial
$$
F_{1,r}(a, b_1, \cdots, b_k): =  \sum_{j=1}^k F_{1,r,j}(a, b_j),
$$
we have proved \eqref{eq:Dr-gpsiI} i.e., 
\beastar
M_r^*({\bf g}) & \leq& r \|\ell_{\bf g}\|_r \left(1 + M_1^*({\bf g})\right) + r M_r^*({\bf g})
\left(1 + \|\ell_{\bf g}\|_1\right)^r  \nonumber \\
&{}& + F_{1,r}(\|\ell_{\bf g}\|_{r-1}, \{M^*_{r-1}({\bf g}_{;j} \}_j).
\eeastar

 To prove \eqref{eq:Dr-gpsi-1},
we first recall $\ell_{g^{-1}} = - \ell_g \circ g^{-1}$. Then we apply \eqref{eq:Drfg} to obtain
\beastar
\|\ell_{g^{-1}}\|_r & \leq & \|\ell_g\|_r \left(1 + M_1^*(g^{-1})\right)^r + \|g^{-1}\|_r(1 + \|\ell_g\|_1)^r 
\nonumber \\
&{} & +  F_{2,r}\left(\ell_g\|_{r-1}, M_{r-1}(g^{-1})\right)
\eeastar
for a two-variable admissible polynomial $F_{2,r,k}$ of degree less than $r$.
Then by substituting $\|g^{-1}\|_1 \leq 2 M_1^*(g)$ for $M_1(g) < \frac14$, we have finished the proof.
\end{proof}

\subsection{Basic $C^r$ estimates on $\Cont_c^r(\R^{2n+1},\alpha_0)$}

The following is the list of basic estimates with respect to the amended norms \eqref{eq:mur*} and
\eqref{eq:Mr*} adapted to the case of contactomorphisms of those listed in
Section \ref{sec:estimates-summary} for the case of general diffeomorphisms.

\begin{lem} Let $R_E < \infty$ and let $r \geq 0$ be given. Then there exists a constant
$C$ independent of $f$ depending only on $R_E$ such that
$$
\mu_r^*(f) \leq C \mu_{r+1}^*(f)
$$
for all $f \in \Cont_E(\R^m, \alpha_0)$.
\end{lem}
\begin{proof} Let $x \in E$. Since $R_E < \infty$, there exists $x_0 \in \R^m \setminus E$ with
$|x - x_0| \leq R_E$. By the fundamental theorem of calculus, we have 
$$
D^r(f)(x) - D^rf(x_0) = \int_0^{|x - x_0|} \frac{d}{ds}(D^rf((1-s)x_0 + sx)\, ds
$$
where $x_0 \in \R^m \setminus E$ and $x \in E$. The lemma immediately follows from
the chain rule
$$
\frac{d}{ds}(D^rf((1-s)x_0 + sx) = \sum_{i=1}^m (\xi_i(x) - \xi_i(x_0))\frac{\del}{\del \xi_i}(D^rf)((1-s)x_0 + sx)
$$
with the choice of constant $C = m R_E$, since $\frac{\del}{\del \xi_i}(D^rf)(x_0) = 0$ for any $r \geq 0$
at $x_0 \in \R^m \supp f$. The same estimate also applies to $\ell_f$ since $\ell_f(x_0) = 0$ outside the support of $f$.
\end{proof}

The following is \cite[Lemma 3.6]{rybicki2} with a slight variation of some numerics in the 
statements.

\begin{prop}\label{prop:basicestimates-contact} 
Let $f_1, \ldots, f_k \in \Cont_c(\R^m,\alpha_0)$ and 
${\bf f} = (f_1, \ldots,f_k)$. 
\begin{enumerate}
\item For $r = 1$, we have
$$
\mu_1^*(f_1\circ\cdots \circ f_k) \leq k \mu_1^*({\bf f})\left((1+ M_0^*({\bf f})(1+ M_1^*({\bf f}))\right)^{k-1}.
$$
\item For any $r, \, k \geq 2$, there exists a two variable admissible polynomial $F_{r,k}$ such that
\be\label{eq:k-product}
\mu_r^*(f_1\circ\cdots \circ  f_k) \leq k \mu_r^*({\bf f})\left((1+ \mu_0^*({\bf f})^{k-1}(1+ \mu_1^*({\bf f}))\right)^{r(k-1)}
+ F_{r,k}(M_{r-1}^*({\bf f})
\ee
\item Suppose $\mu_0^*(f), \, \mu_1^*(f) < \delta < 1$. Then 
$$
\mu_1^*(f^{-1}) \leq \frac{ \mu_1^*(f)}{1-\delta}
$$
\item Suppose $\mu_0^*(f), \, \mu_1^*(f) < \delta < 1$. Then for any $r \geq 2$, there exists an admissible polynomial
$F_r$ such that for any $f \in \Cont_c(\R^m,\alpha_0)$ such that
$$
\mu_r^*(f^{-1}) \leq \left(\frac{1}{1-\delta}\right)^{r+2} \left(1 + \frac{\mu_1^*(f)}{1-\delta}\right)^{r+1}\mu_r^*(f)
 + F_r\left(M_{r-1}^*(f)\right).
$$
\end{enumerate}
\end{prop}
\begin{proof} The proof will be the same as that of \cite[Lemma 3.6]{rybicki2}, except the fact that we 
are using the conformal exponent $\ell_f = \log \lambda_f$ which slightly simplifies the estimates
thanks to the \emph{more linear} formulae of the exponent of the product, i.e.,
$$
\ell_{f_1\circ \cdots \circ f_k} = \sum_{i=1}^k \ell_{f_i} \circ (f_{i+1} \circ \cdots \circ f_k).
$$
(Recall that the corresponding formula for the conformal factor $\lambda_f$ are \emph{multiplicative}.)
Since we use $\ell_g$ instead of $\lambda_g$ in our estimates, we just focus on
this difference of the calculation with the replacement of $\lambda_f$ by the conformal exponents $\ell_f$.

We start with the basic inequality
$$
\|Df^{-1}\| \leq \frac{1}{1 - M_1(f)} 
$$
provided $M_1(f) < 1$, which follows from the inequality $\|M^{-1}\| \leq (1 - \|\id -M\|)^{-1}$
for an invertible matrix $M$.

Therefore we have $\|\ell_{f^{-1}}\| = \|\ell_f\|$, and 
$$
D\ell_{f^{-1}} = -\left(D\ell_f  \circ f^{-1}\right) Df^{-1}
$$
which implies
$$
\|D\ell_{f^{-1}}\| \leq \|D\ell_f \circ f^{-1}\| \|Df^{-1}\| \leq  \frac{\|D\ell_f\|}{1-M_1^*(f)}.
$$
Therefore we obtain
$$
\mu_1^*(f^{-1}) = \max \{\|Df\|, \|D\ell_f\|\} \leq \max \left\{\|Df\|, \frac{\|D\ell_f\| }{1-M_1^*(f)}\right\}
\leq \frac{\mu_1^*(f)}{1-M_1^*(f)}.
$$
This proves (1) for the case $r =1$. For higher $r \geq 2$, we inductively perform the estimates
similarly as the proof of \cite[Lemma 3.6]{rybicki2} with the replacement of $\lambda_f$ and $\ell_f$.

In particular, if $M_1(f) < \delta$, the inequality is reduced to
$$
\mu_1^*(f^{-1}) \leq \frac{\mu_1(f)}{1-\delta}
$$
and 
$$
\mu_r^*(f^{-1}) \leq \left(\frac{1}{1-\delta}\right)^{r+2} \left(1 + \frac{\mu_1^*(f)}{1-\delta}\right)^{r+1}\mu_r^*(f)
 + F_r\left(M_{r-1}^*(f)\right).
$$
\end{proof}

\section{Estimates on derivatives under the Legendrianization}
\label{sec:estimates-legendrianization}

This is the central section of the present paper as well as in \cite{rybicki2}. The materials of the
present section are not involved in  Mather's  case of general diffeomorphisms 
and other later literature
but that relies crucially on some contact geometric aspect related to the geometry of
Legendrianization of contactomorphisms laid out in Section \ref{sec:Darboux-Weinstein} and Section \ref{sec:legendrianization}.
In this regard, the following equation is the reason why the estimates given in 
Proposition \ref{prop:derivative-uf} hold:
\be\label{eq:fundamental-eq}
\Phi_{U;A}^{-1} \circ \Delta_g( \CW_k^m)=  j^1u(\CW_k^m), \quad m = 2n+1
\ee
on $\CW_k^m$ for some real-valued function $u = u_g$  on $\CW_k^m$.

We start with the estimates involving the map 
$\mathscr G_A$ defined in \eqref{eq:GA}. Since the chart $\Phi_U$ is fixed and will not
be changed, we will suppress the dependence on $\Phi_U$ of various constants 
appearing below.  Essentially all constants depend on this map $\Phi_U$.

The following is a comparison result between the $C^r$-norm of the contactomorphism $f$ 
and the $C^{r+1}$-norm of its 1-jet potential $u_f$.  This is the optimal version of
\cite[Proposition 4.6]{rybicki2}.

\begin{prop}[Compare with Proposition 4.6 \cite{rybicki2}]\label{prop:derivative-uf}
Let $A_0> 2$ be a sufficiently large given constant. For $2 < A \leq A_0$, let
 $E \subset E_A^{(k)}$ be a  given 
subinterval. For any $r \geq 2$ there is a $C^1$-neighborhood $\CU_1$ of the identity in
$\Cont_E(\CW_k^{2n+1},\alpha_0)$ such that for any $g \in \CU_1$ with $M^*_1(g) < \frac14$,
 the function
$u= u_g: = \mathscr G_A(g)$ satisfies the following:
\begin{enumerate}
\item There exists constant $C_1 = C_1(r,\Phi_U)$ such that 
\bea\label{eq:Mr*f-ur+1}
M_r^*(g)  \leq  C_1 \|u_g\|_{r +1} 
 +  C_1 A(1 + \|u_g\|_2)^r  + A F_{2,r}(\|h_X\|_{r-1}, \|u_g\|_r).
\eea
holds uniformly over $1< A \leq A_0$.
\item There exists a constant $C_2 = C_2(r,\Phi_U)$ depending only on $r$  and the 
chart $\Phi_U$ for which
\bea\label{eq:2st-inequality}
\|u\|_{r+1} \leq C_2 M_r^*(g) + A^2 P_r(M_{r-1}^*(g))
\eea
holds uniformly over $1< A \leq A_0$.
\end{enumerate}
\end{prop}

Before launching on the proof of this proposition, we need some digression into
the consequence of the definition  $\mathscr G_A(g)$.
Recall the definition $\Phi_{U;A} = \nu_A \circ \Phi_U \circ \mu_A^{-1}$. 
We observe
$$
\pi_2 \circ \nu_A = \chi_A \circ \pi_2
$$
where $\pi_2: \R \times \R^{2n+1} \times \R^{2n+1} \to \R^{2n+1}$ is the projection to the 
second factor, and that
\bea\label{eq:mjA-1delta}
\mu_A^{-1} \circ \delta (t,x,X) & = & \delta \mu_A^{-1}(x,X,t) = 
\delta \left(t, \chi_A^{-1}(x),\chi_A^{-1}(X), t\right) \nonumber\\
& = & \delta\left(\chi_A^{-1}(x), \chi_A^{-1}(X),t \right).
\eea

We regard a function $u: \CW_k^{2n+1} \to \R$
as a periodic function on $\R^{2n+1}$ which we denote it by the same letter $u = u(x)$
for $x = (z, q, p)$ in the canonical coordinates of $ J^1\R^n \cong \R^{2n+1}$. Then
$$
du = \frac{\del u}{\del z} dz + \sum_{i=1}^n \frac{\del u}{\del q_i} dq_i 
+ \sum_{i=1}^n \frac{\del u}{\del p_i} dp_i .
$$
We write the coefficient vector of $du$ as
$$
Du = (D_z u, D_q u, D_p u), \quad D_q u= (D_{q;1}, \ldots, D_{q;n}), \, 
D_p u= (D_{p;1}, \ldots, D_{p;n})
$$
Then we consider a rescaled Darboux-Weinstein chart $\Phi_{U;A}$.

\subsection{Proof of Statement (1)}

For the proof of Statement (1) of Proposition \ref{prop:derivative-uf}, 
we start with \eqref{eq:fundamental-eq}. We can express $\Gamma_g$ as
$$
\Gamma_g(x) =  \delta \Pi \Phi_{U;A}^{-1} (j^1u(y))
$$
for some function $u = u_g$ provided $g$ is sufficiently $C^1$ close to the identity map.
It follows from the discussion around \eqref{eq:PhiUDeltaf} that $(\delta \Pi \Phi_{U;A}^{-1})^{-1} \circ \Gamma_g$
is a Legendrian submanifold of $\alpha_0$ when $g$ is sufficiently $C^1$-close to the identity.

Then its $p$ and $\eta$ components become
\bea\label{eq:g-id-ellg}
g (x) & = & \text{\rm pr}_2\left(\delta\Pi  \Phi_{U;A}^{-1}((j^1u)(y))\right) \label{eq:g-id(x)1}\\
\ell_g(x) & = & \text{\rm pr}_3 \left(\delta\Pi \Phi_{U;A}^{-1}((j^1u(y)))\right). \label{eq:ellg(x)1}
\eea
We also have
$$
 \Phi_{U;A}^{-1} = \mu_A \circ  \Phi_U^{-1} \circ \nu_A^{-1}
$$
by the definition \eqref{eq:PhiUA}.

In the following calculations, for the simiplicity of notation, \emph{we will identify both
$J^1\R^{2n+1}$ and $M_{\R^{2n+1}}$ with $\R^{2(2n+1)+1}$ and so regard $\Pi$ as
the coordinate swapping self map, $H$ as another self map on $\R^{2(2n+1)+1}$ and the $G$-actions $\CG_1$ and $\CG_2$ act
all on the same space $\R^{2(2n+1)+1}$.} Precisely speaking the expression `$\Pi + \mathsf H$'
should have been written as
$$
\Pi+ \mathsf H \circ \Pi
$$

Then by substituting \eqref{eq:PhiU-1delta}
into \eqref{eq:g-id(x)1}, we derive the formula
\beastar
g (x) & = & \chi_A \text{\rm pr}_2 (\Pi + \mathsf H)(\nu_A^{-1} j^1u(y)) \\
\ell_g(x) & = & \text{\rm pr}_1 (\Pi + \mathsf H)(\nu_A^{-1} j^1u(y))
\eeastar
from \eqref{eq:PhiU-1delta}, where $\mathsf H = (h_x,h_X,h_t)$.
We evaluate
\bea\label{eq:g(x)2}
g(x) & = & \chi_A \text{\rm pr}_2 (\Pi + \mathsf H)(\nu_A^{-1} j^1u(y)) \nonumber\\
& = & \chi_A \text{\rm pr}_2 (\Pi + \mathsf H)(\nu_A^{-1} (u(y), y, Du(y)) \nonumber\\
& = &  \chi_A\eta_A^{-1}Du(y) + \chi_A h_X(\nu_A^{-1}(j^1u(y)).
\eea
Similarly, using \eqref{eq:ellg(x)1}, 
\be\label{eq:ellg(x)2}
\ell_g(x) = A^{-2}(u(y) + h_t(\nu_A^{-1}(j^1u(y) ).
\ee
We have
$$
\chi_A \eta_A^{-1}(z,q,p) = (A z,q, Ap), \quad
\eta_A \chi_A^{-1}(t,x,X) = (A^{-1} z,q, A^{-1} p).
$$
Combining them with the above formulae, we derive
$$
\|D(\chi_A(h_X \mu_A^{-1} j^1u)\| \leq A \|D(h_X \eta_A^{-1} j^1u)\|,
$$
and taking further derivatives using the formula \eqref{eq:Drfg},  we have obtained
\beastar
\|D \Gamma_g \|_r & \leq& A \|D(h_X j^1u) \|_r \\
& \leq &  A  \|D \chi_A^{-1} j^1u\|_r (1 + M_1 (h_X))^r \\
&{}&  + A \|h_X\|_r (1 + M_1^*(\chi_A^{-1}j^1u))^r 
 + A F_{1,r} (M_{r-1}^*(h_X), M_{r-1}^*(\chi_A^{-1} j^1u))\\
 & \leq & C \|Dj^1u\|_r (1 + M_1 (h_X))^r \\
&{}&  + C A \|h_X\|_r (1 + M_1^*(j^1u))^r 
 + A F_{1,r} (M_{r-1}^*(h_X), M_{r-1}^*(j^1u))
 \eeastar
 for any $r\geq 1$.
 Here the third inequality holds, since $\|Dj^1u\| \leq CM_1^*(g)$,
$\|u\|_r \leq C \|u\|_{r+1}$, $\|\chi_A^{-1}\| \leq \frac1A \leq 1$ and $\chi_A$ is
a linear invertible map.

\subsection{Proof of Statement (2)}

Again we start with \eqref{eq:g(x)2} and \eqref{eq:ellg(x)2}.
We rewrite them into
\bea
Du(y)  & =  & \eta_A \chi_A^{-1} g(x) - \eta_A h_X(\nu_A^{-1}(j^1u(x)) \label{eq:Du(x)3}\\
u(y) & = & A^{-2}\left(\ell_g(x)  -     h_t(\nu_A^{-1}(j^1u(x))) \right). \label{eq:u(x)3}
\eea
We will derive the estimate of $\|u\|_{r+1}$ 
in terms of $\|g\|_r$ inductively over $r$ from these two, remembering that $|x -y|  \leq \|h_x\|$. 

For $\|u\|_0$ and $\|Du\|_0$, we derive
$$
|u(y)| \leq A^{-2} |\ell_g(x)| + | h_t(\nu_A^{-1}(j^1u(y))) | \leq A^{-2}(\|\ell_g\| + \|h_t\|)
$$
and
$$
|Du(y)| \leq \| \eta_A \chi_A^{-1} g\|_0 + \|\eta_A h_X\|_0 \leq \|g\|_0 + \|h_X\|_0.
$$
Combining the two and using $A \geq 1$, we have derived
\bea\label{eq:||u||1}
\|u\|_1 & \leq & A^{-2}(\|\ell_g\| + \|h_t\|) + \|g\|_0 + \|h_X\|_0 \nonumber \\
& \leq&  M_0^*(g) + \|\mathsf H\|_0.
\eea
For the higher derivatives $\|D^ru\|$, $r \geq 2$, we 
start from
$$
\Gamma_g(x) =  \delta\Pi \Phi_{U;A}^{-1} (j^1u(y))).
$$
which is equivalent to
\be\label{eq:j1u(y)}
j^1u(y) = \Phi_{U;A} (\delta\Pi)^{-1} \Gamma_g(x).
\ee
In particular, we have
$$
x = \text{\rm pr}_2\delta\Pi \Phi_{U;A}^{-1}(j^1u(y)).
$$
We mention that the map
$$
y \mapsto \text{\rm pr}_2\delta\Pi \Phi_{U;A}^{-1}\circ (j^1u(y)))
$$
is invertible as a map to $R_{g,A} = \image \Gamma_g$ from $\CW_k^m$ provided $g$ is sufficiently
$C^1$-small. By writing the inverse thereof by $\underline{\Upsilon}_{A,g}$, we can write
$x = \underline{\Upsilon}_{A,g} (\Gamma_g(y))$
for the map 
\be\label{eq:line-Upsilon}
x =  \underline{\Upsilon}_{A,g}(\Gamma_g(y))
: = (\text{\rm pr}_2\delta\Pi\Phi_{U;A}^{-1}j_1u )^{-1}(y) = (\text{\rm pr}_2 \circ \Gamma_g)^{-1}(y)
= g^{-1}(y).
\ee

Therefore substituting $x = g^{-1}(y)$ into \eqref{eq:j1u(y)}, we can express
\beastar
j^1 u(y) & = &  \Phi_{U;A} (\delta\Pi)^{-1} \Gamma_g \circ g^{-1}(y)
= \Phi_{U;A} (\delta\Pi)^{-1}(g^{-1}(y), y, \ell_g \circ g^{-1}(y))\\
& = &  \Phi_{U;A} (\delta\Pi)^{-1}(g^{-1}(y), y, - \ell_{g^{-1}}(y)).
\eeastar
Recalling $\ell_{g^{-1}}= - \ell_g \circ g^{-1}$, we write
\be\label{eq:alephA}
\aleph_A : = \Phi_{U;A}^{-1}( \delta\Pi )^{-1}, \quad K_g(y): = (g^{-1}(y), y, - \ell_{g^{-1}}(y))
\ee
\begin{lem} For $r \geq 1$, we have
$$
\|j^1u\|_r \leq  C M_r^*(g)
+ F_{r,3}\left(\max\{\|\aleph_A\|_{r-1}, M_{r-1}^*(g)\}\right).
$$
\end{lem}
\begin{proof}  We decompose
$$
j^1 u(y) =  \aleph_A \circ  K_g (y).
$$
By applying \eqref{eq:r-norm-product} here with $f_1 = \aleph_A, \, f_2 = K_g,$, 
we obtain the formula
\beastar
\|j^1u\|_r 
& \leq &(1 +  \|\aleph_A\|_1) M_r^*(g) + (1 + M_1^*(g))^r \|\aleph_A\|_r 
+ F_{r}\left(\max\{\|\aleph_A\|_{r-1}, M_{r-1}^*(g)\}\right)
\eeastar
where the terms $\|\aleph_A\|_r$ measured after the evaluation of $K_g$.
In particular for $r \geq 1$, we have
$$
 \|\aleph_A\|_r  \leq C \frac1{A^2} \|K_g\| \leq \frac{C}{A^2} M_1^*(g) \leq \frac{CC'}{A^2} M_r^*(g)
$$
where the first inequality follows Proposition \ref{prop:PhiUK} (2). Therefore we can bound
$$
(1 + M_1^*(g))^r \|\aleph_A\|_r  \, \, \leq C'' 
$$
uniformly over $1 \leq A \leq A_0$ and hence
$$
\|j^1u\|_r  \leq  (1 + 2C''')  \|\aleph_A\|_r 
+ F_{r}\left(\max\{\|\aleph_A\|_{r-1}, M_{r-1}^*(g)\}\right).
$$
\end{proof}

Therefore  we obtain
$$
\|u\|_{r+1} \leq C M_r^*(g) + A^2 P_r(M_{r-1}^*(g))
$$
inductively over $r \geq 1$, where $P_r$ is a polynomial that has no constant term.
This   finishes the proof of the second inequality of Proposition \ref{prop:derivative-uf}
if we have made 
$$
\|h_X\|_1, \, M_1^*(g)< \delta 
$$
for a sufficiently small $\delta$ by choosing the Darboux-Weinstein chart $\Phi_U$
with the neighborhood of $U$ of $\Delta_{\CW_k^m}$ sufficiently small, and considering 
$g$ in a sufficiently small neighborhood of the identity.

\section{Fragmentation of the second kind: Estimates}

Our goal of this section is to establish the following derivative estimates.

\begin{prop}[Compare with Proposition 5.7 \cite{rybicki2}]\label{prop:2nd-fragmentation-estimates} 
Let $2A \geq 2$ be an even integer,  $\rho:[0,1] \to [0,1]$ be a boundary-flattening
function such that $\rho\equiv 1$ on $[0,\frac14]$ and $\rho \equiv 0$ on $[\frac34,1]$, and let
$$
E_{2A} : = E_{2A}^{(0)} = [-2A,2A]^{2n+1}
$$
Then there exists a $C^1$-neighborhood $\CU_{\chi,A}$ of the identity 
in $\Cont_{E_{2A}}(\R^{2n+1},\alpha_0)$ such that for any $g \in \CU_{\chi,A}$ 
the factorization 
$$
g = g_1 \cdots g_{a_m}, \quad a_m = (4A + 1)^m
$$
given in Proposition \ref{prop:2nd-fragmentation}  satisfies the following estimates:
 Whenever $\supp g \subset E \subset E_{2A}$ with $R_E \leq 2$, the inequalities 
\bea
 M_r^*(g_K) & \leq & C_\chi  M_r^*(g) +  A P_{\chi,r}(M_{r-1}^*(g)),\\
 M_r^*(g_K) & \leq &  C_{\chi,r} M_r^*(g)
\eea
hold for all $K = 1, \ldots, a_m$ and $r \geq 2$. 
\end{prop}

We will prove the estimate inductively over the number $1 \leq \ell \leq n+1$ of directions of 
the fragmentation. 

We start with the case $\ell =1$. In this case, we restate the above proposition as 
the following lemma.

\begin{lem}[Compare with Proposition 5.6 (2) \cite{rybicki2}]\label{lem:2nd-fragmentation-estimates}
Let $2A > 1$ be an even integer,  $\chi:[0,1] \to [0,1]$ be a boundary-flattening
function such that $\chi \equiv 1$ on $[0,\frac14]$ and $\rho \equiv 0$ on $[\frac34,1]$, and let
$$
E_{2A} : = E_{2A}^{(0)} = [-2A,2A]^{2n+1}
$$
Then there exists a $C^1$-neighborhood $\CU_{\chi,A}$ of the identity 
in $\Cont_{E_{2A}}(\R^{2n+1},\alpha_0)$ such that for any $g \in \CU_{\chi,A}$ 
$$
\supp g_K \subset \left(\left[k- \frac34,k+\frac34 \right] \times \R^{2n}\right)
\cap E_{2A}
$$
with $k \in \Z$, $|k| \leq 2A$, there exists a factorization 
\be\label{eq:psi-decompose1}
g = g_1 \cdots g_{4A +1}, 
\ee
that satisfies the following estimates: Whenever 
$\supp g \subset E \subset E_{2A}$ with $R_E \leq 2$, the inequalities 
\bea\label{eq:murgkappa}
 M_r^*(g_K) & \leq & C_\chi  M_r^*(g) + A P_{\chi,r}(M_{r-1}^*(g)),
 \nonumber\\
 M_r^*(g_K) & \leq &  C_{\chi,r} M_r^*(g)
\eea
hold for all $K = 1, \ldots, 4A+1$ and $r \geq 2$. 
\end{lem}
\begin{proof} As in the proof of Proposition \ref{prop:2nd-fragmentation},
we extend $\chi$ to $[-1,1]$ as an even function and then to $\R$ as a 2-periodic function.
For any  sufficiently $C^1$-small contactomorphism $g$ with $\supp g \subset E_{2A}$,
we consider the function $ g^\psi:  = \mathscr G_A^{-1}(\psi u_g)$ and its factorization
\beastar
g^\psi_1 & : = &  g_{-2A} \circ g_{-2(A-1)} \circ \cdots \circ g_{2(A-1)} \circ g_{2A},\\
g^\psi_2 & : = &  g(g^\psi_1)^{-1}
\eeastar
with $\supp g_{2k} \subset [2k-\frac34,2k + \frac34 ] \times \R^{2n}$ and 
$\supp g_{2k+1} \subset [2k+\frac14,2k + \frac74] \times \R^{2n}$
constructed in the proof of Proposition \ref{prop:2nd-fragmentation}.

Note that $\supp g_{2k}$ is contained in the disjoint union
$$
 \left[2k+\frac14,2k + \frac74\right] \times \R^{2n}
 $$
 Therefore we have
\be\label{eq:murgchi}
 M_r^*(g^\psi) \leq \sum_{k = -A}^A M_r^*(g_{2k}).
\ee
 On the other hand, we have
$$
 M_r^*(g_{2k})  \leq  C M_r^*(g ) + A P_{\chi,r}(\sup_{0 \leq s \leq r}\|j_1u\|_s, \|H\|_r)
 $$
and hence have derived
 $$
 M_r^*(g_{2k}) \leq C_{\chi,r}  M_r^*(g) +  A P_{\chi, r}\left(M_{r-1}^*(h_X), M_{r-1}^*(g -\id)\right)
 $$
 after rechoosing the polynomial $P_{\chi,r}$. 

 \end{proof}

\begin{proof}[Proof of Proposition \ref{prop:2nd-fragmentation-estimates}]

By applying the above construction consecutively to all variables $(q,p,z)$, we have
finished the proof.
 Whenever $\supp g \subset E \subset E_{2A}$ with $R_E \leq 2$, the inequalities 
\bea\label{eq:Mr*gkappa}
 M_r^*(g_K) & \leq & C_{\chi,r}  M_r^*(g) +  A P_{\chi,r}(M_{r-1}^*(g)), \nonumber\\
 M_r^*(g_K) & \leq & C_{\chi,r} M_r^*(g)
\eea
hold for all $K = 1, \ldots, 4A + 1$ and $r \geq 2$, \emph{provided we
consider $g$ from a sufficiently small $C^1$ neighborhood of the identity}. Analogous decompositions can be obtained 
with respect to other variables $q_i$ and $y_i$. This finishes the proof.
\end{proof}

\section{The threshold determining optimal scaling estimates}
\label{sec:threshold}

We consider the situation of Section \ref{sec:shifting-supports}. For each $A > 1$, we consider
the square 
\be\label{eq:IA}
I_A =  [-2,2] \times  [-2,2]^n \times [-2A,2A]^n \subset M_{\R^{2n+1}} \cong \R^{2(2n+1)+1}.
\ee
The following is the optimal scaling estimates that essentially determines the 
threshold $r = n+ 3$ for the dichotomy appearing later in the main theorem of the 
present paper. This optimal inequality is the contact counterpart of the inequality \cite[p.518]{mather}, 
\cite[Equation (5.2)]{epstein:commutators}.

\begin{prop}[Compare with Proposition 6.1 \cite{rybicki2}] \label{prop:rybicki2}
If $|t_i|\leq 2A$ for $i = 1, \ldots, n$
and $g \in \Cont_{I_A}(\R^{2n+1},\alpha_0)_0$, we have
$$
M_r^*(\rho_{A, {\bf t}} \circ g \circ \rho_{A,{\bf t}}^{-1}) 
\leq A^{4-2r} (2n)^{r+1} M_r^*(g).
$$
\end{prop}
\begin{proof} We know
$$
\supp (\rho_{A, {\bf t}} \circ g \circ \rho_{A,{\bf t}}^{-1}) \subset J_A
$$
for all $g \in \Cont_c(\R^{2n+1},\alpha_0)$ with support in the shifted $I_1$
$$
I_1 + (2k-1) \vec e_i = [-2,2]^{n+i} \times [k-1,k+1] \times [-2,2]^{n-i} 
$$
for all $|k| \leq 2A-1$ so that $-2A < k-1 <  2(A-1)$ and $ -2(A-1) < k+1 < 2A$.

We have
$$
 \rho_{A,{\bf t}}^{-1}(z,q,p) = \left(\frac{z - \sum_{i=1}^n t_i q_i }{A^4}, \frac{q}{A^2}, 
 \frac{p- {\bf t}}{A^3}\right)
 $$
 from \eqref{eq:rho-1}.
Then we compute
$$
D\rho_{A,t}^{-1}(z,q,p) = A^{-4} \vec e_z \, dz  
+ \sum_{i=1}^n \left ( - \frac{\vec e_zt_i}{A^4} \,  
 +  \frac{ \vec e_i}{A^2} \right) dq_i +   \frac1{A^3} 
\sum_{j=1}^n \vec f_j \, dp_j \
$$
as a vector valued one-form  on $\R^{2n+1}$. 
Since $|t| \leq 2A$ and $A \geq 1$,  we obtain $\|D \rho_{A,t}^{-1}\| \leq 2 A^{-2}$ and
\bea
\|D_{zz} \rho_{A,t}\| \leq  A^4, \quad \|D_{qz} \rho_{A,t}\| \leq A^2 |{\bf t}|_\infty \leq 2A^2 \\
\|D_{q_iq_i}\rho_{A,{\bf t}}\| \leq A^3, \quad \|D_{p_ip_i}\rho_{A,{\bf t}}\| \leq A^2.
\eea
This implies $\|D  \rho_{A,t} \| \leq A^4$.
Then we estimate
\beastar
M_r^*(\rho_{A, {\bf t}} \circ g \circ \rho_{A,{\bf t}}^{-1}) & \leq & \|D \rho_{A,{\bf t}}\|^r
(2n A^4) \|D^rg\| (2A^{-2})^r) \\
& \leq & A^{4-2r}(2n)^{r+1} M_r^*(g).
\eeastar
\end{proof}

\section{Estimates on the rolling-up and the fragmentation operators}

In this section, we establish crucial estimates concerning the rolling-up and the fragmentation operators
denoted by $\Theta_A^{(k)}$ and $\Xi_{A;N}^{(k)}$ in \cite{rybicki2}.
\emph{We emphasize that to obtain the threshold $r=n+2$, the order of power of $A$ being 2 is crucial.}

Recall the counterpart \eqref{eq:ThetakA-defn} of Mother's rolling-up operators for which 
we made the
choice 
\be\label{eq:qk<-2A2}
q_k\left(\widetilde x\right) < -2 A^2
\ee
and choose $N> 0$ large enough so that
\be\label{eq:qk>2A2}
q_k\left((T_kg)^N(\widetilde x)\right) > 2A^2
\ee
therein.

\begin{prop}[Compare with  Inequality (2.2)\cite{epstein:commutators}]
\label{prop:MThetakA}
Let $k = 0, \ldots, k$. After shrinking $\CU_1$ if necessary, 
$\Theta_A^{(k)}$ satisfies the following estimates:
There exists a constant $K_1 > 0$ such that
\be\label{eq:MrThetakA}
M_r^*(\Theta_A^{(k)}(g)) \leq K_1 A^2 (1+ M_1^*(g))^{rK_1A^2} M_r^*(g) 
+ F_{r,A}(M_{r-1}^*(g)),
\ee
and 
\be\label{eq:M1ThetakA}
M_1^*(\Theta_A^{(k)}(g)) \leq K_1A^2 M_1^*(g) (1 + M_1^*(g))^{K_1A^2}.
\ee
for any $g \in \text{\rm Dom}(\Theta_A^{(k)})$. Moreover $F_{1,A} = 0$.
\end{prop}
\begin{proof} The proof is similar to that of \cite[Inequality (2.2)]{epstein:commutators}.
Similarly as therein, we choose
\be\label{eq:choice-N}
N = 8A^2 + 4.
\ee
Then by the same argument as in the proof of  \cite[Inequality (2.2)]{epstein:commutators},
we obtain
$$
M_r^*(\Theta_A^{(k)}(g)) \leq K_1 A^2 (1+ M_1^*(g))^{rK_1A^2} M_r^*(g) 
+ F_{r,A}(M_{r-1}^*(g)),
$$
where $F_{r,A}$ is an admissible polynomial of one variable. We also have
$$
M_1^*(\Theta_A^{(k)}(g)) \leq K_1A^2 M_1^*(g) (1 + M_1^*(g))^{K_1A^2}.
$$
This finishes the proof.
\end{proof}

The following is a key corollary of the above proposition.

\begin{cor}\label{cor:MrThetakA} Assume the same hypotheses as in Proposition \ref{prop:MThetakA}.
There exists a constant $K_2 > 0$ such that
\be\label{eq:murThetakA-2}
M_r^*(\Theta_A^{(k)}(g)) \leq K_2 r A^2 M r_r^*(g) 
+ F_{r,A}(M_{r-1}^*(g)),
\ee
and 
\be\label{eq:M1ThetakA-2}
M_1^*(\Theta_A^{(k)}(g)) \leq K_2A^2 M_1^*(g).
\ee
for any $g \in \text{\rm Dom}(\Theta_A^{(k)})$. 
\end{cor}
\begin{proof} The current proof goes along the same line as that of
\cite[p.117]{epstein:commutators}. Recall the definition of exponential function
$$
\lim_{n \to \infty}\left(1+ \frac1n\right)^{nx} = e^x
$$
and the function $n \to (1+ \frac1n)^{nx}$ is an increasing function for 
any fixed $x > 0$. Therefore if we consider $g$'s whose $C^1$-norm
$M_1^*(g) < \frac{1}{A^2}$, then 
$$
K_1 A^2 (1+ M_1^*(g))^{rK_1A^2} \leq  A^2 K_1\left (1+ \frac{1}{A^2}\right)^{rK_1A^2} \leq A^2 e^{r K_1}.
$$
By setting $K_2 = K_1 e^{r K_1}$, we have finished \eqref{eq:M1ThetakA-2}.

Substituting this into \eqref{eq:MrThetakA}, we obtain
$$
M_r^*(\Theta_A^{(k)}(g)) \leq K_2 r A^2 M_r^*(g) + F_{r,A}(M_{r-1}^*(g)).
$$
This finishes the proof.
\end{proof}

We next prove the following estimates.
 
 \begin{prop}[Compare with Proposition 8.2 \cite{rybicki2}]\label{prop:rybicki82-estimates}
 By shrinking  $\CU_2$ and then $\CU_1$ sufficiently, 
 there are constants $C_{\chi, A}, \, \beta$ and $K_1$ such that 
 \be\label{eq:mur*XiAk}
M_r^*\left(\Xi_{A;N}^{(k)}(g)\right) \leq C_{\chi,r} \, M_r^*(g)
 +  A P_{\chi,r}(M_{r-1}^*(g)).
 \ee
  \end{prop}
\begin{proof}
This is an immediate consequence of \eqref{eq:murgchi} and \eqref{eq:murgkappa}.
\end{proof}

\part{Proof of the main theorems}

  In this section, we give the proofs of the main theorems using
  the derivative estimates which are established in Part II.

 \section{Rybicki's fundamental homological lemma}
 \label{sec:key-lemma}
 
 In this subsection, we  explain a key lemma \cite[Lemma 8.6]{rybicki2} (2) that plays a 
 fundamental role in Rybicki's proof of  perfectness of the $C^\infty$ contactomorphism group,
 which is shared by many ingredient used in Mather's proof in \cite{mather}--\cite{mather4}
 and \cite{epstein:commutators}.
 
 We have already stated the first half of the statement of \cite[Lemma 8.6]{rybicki2}
 in Lemma \ref{lem:rybicki86(1)}.
Now we separate the second half of the statement of \cite[Lemma 8.6]{rybicki2} here, 
which we need to improve them by the reasons mentioned in Remark \ref{rem:rybicki-error}
in the introduction of the present paper. Because of this, we needed to provide its details with some corrections
and amplifications of the construction of the unfolding-fragmentation operator
 associated the $N$-fragmentation with $N > 2$ as given in Section \ref{sec:unfolding-fragmentation}.
 Furthermore we also need to make a  finer choice of  various numerical constants 
 appearing in the construction.  

Wee recall 
$$
\CW_{n+1}^{2n+1}  \cong  T^{n+1} \times \R^n  \cong S^1 \times T^*(T^n \times \R),
$$
and more generally
\be\label{eq:CWk}
\CW_k^{2n+1} \cong S^1 \times T^*(T^{k-1} \times \R^{n-k+1})
\ee
for $k = 0, \ldots, n$. Let $f \in \CU_3$ and consider the $T^{n+1}$-equivariant map
 \be\label{eq:defn-f*}
 f^*: = \widehat \Theta_A^{(n>)}(f) : \CW_{n+1}^{2n+1} \to \CW_{n+1}^{2n+1}.
\ee
\begin{lem} The map $f^*$ satisfies the following:
\begin{enumerate}
\item It has the form
\be\label{eq:f*}
f^*(\xi_0, \xi, y) = f^*(z,q,p) =  (z + f_0^*(p), q+ f_1^*(p), p), \quad (\xi_0,\xi,y) = (z,q,p)
\ee
for some function $(f_0^*,f_1^*): \R^n_p \to \R^{n+1}_{(z,q)} = \R^z \times \R^n_q$.
\item 
We have
\be\label{eq:Xia=Xia'}
[\Xi^{(k)}_{A;a}(f^*)] = [\Xi^{(k)}_{A;a'}(f^*)]
\ee
\end{enumerate}
\end{lem}
\begin{proof}
Since $f^*$ is $T^{n+1}$-equivariant, we can write its unique lifting to $\R^{2n+1}$,
still denoted by $f^*$ such that $f^*-id$ is $\R^{n+1}$-equivariant. This implies 
there exists a map $v: \R^n_p \to \R^{n+1}_{(z,q)}$ of the form given by
\be\label{eq:v}
v(p): = (f_0^*(p),f_1^*(p))
\ee
which is a section of the projection $\R^{2n+1} \to \R^n_p$. 
(See Appendix \ref{sec:equivariance} for the proof in a similar context.)
This proves Statement (1).

The statement (2) follows from
 Lemma \ref{lem:[f]=[g]} and Proposition \ref{prop:rybicki82}(3). 
\end{proof}
We have the following suggestive expression of $f*$
\be\label{eq:f*2}
f^* = \id + v \circ \pi_3.
\ee 

\begin{lem}\label{lem:[ga]=[ga']}
We put the map
\be\label{eq:g}
g_a := \Xi_{A;a}^{(<n)}(f^*): \R^{2n+1} \to \R^{2n+1}
\ee
Then
$[g_a] = [g_a']$ for all pairs $(a,a')$ with $a, \, a' \geq 2$. 
\end{lem}
\begin{proof} This follows from 
the definition \eqref{eq:CWk} and \eqref{eq:Xia=Xia'}.
\end{proof}
This lemma enables us to define the following cohomology class
independent of $a$ but depending only on $f$ (and on $A$).

\begin{defn}\label{defn:omegaf} Let $f \in \Cont_c(\CW_n^{2n+1},\alpha_0)_0$ be given and $f^*$ be as in \eqref{eq:defn-f*}.
We denote by $\omega(f)$ this common class of $g_a$ above
in $H^1(\Cont_c(\R^{2n+1},\alpha_0)_0)$. 
\end{defn}
This class is misleadingly denoted by $[g]$ in \cite[Section 8]{rybicki2} without encoding the
dependence of the definition of $g$ therein on $a$ even though the definition of $g$ 
depends on $a$ and so the independence of the class on $a$ should have 
been proved in advance, but not even mentioned.
 
 In this regard, the following is the correct statement of \cite[Lemma 8.6]{rybicki2}.
 
\begin{prop}\label{prop:[g]=e}
For any integer $a \geq 2$, there exists an element $g_a' \in 
\Cont_c(\R^{2n+1},\alpha_0)$ such that
 $$
[g_a'] = \omega(f) \quad  \& \quad [g_a'] = [g_a^{a^{n+2}}]
 $$
 in  $H_1(\Cont_c(\R^{2n+1}, \alpha_0)_0)$ 
 for all $f \in \Cont_c(\CW_k^{2n+1},\alpha_0)$ sufficiently $C^1$-close to the identity.
 \end{prop}
 
Once we have made the above correct statement to prove,  its proof will be
 a consequence of the arguments employed in the proof of \cite[Lemma 8.6]{rybicki2} by
combining the strict identity $\Theta_A^{(n>)} \Xi_{A;a}^{(<n)}(g_a) = g_a$
and an inductive application of Proposition \ref{prop:rybicki82} generalized to
the case $a > 2$, 

The entirety of the next two sections will be occupied by the proof of Proposition
 \ref{prop:[g]=e}.
We divide our discussion into the two cases $a =2$ and $a > 2$ purely for the 
simplicity and convenience of presentation since the details of the latter case 
are not very different from the former case. However
as mentioned before, one needs to specify the dependence on $a$ in 
the identity ``$[g] = [g^{a^{n+2}}]$'' from \cite{rybicki2} which we will do here.

In the next section, the case for $a =2$ will be explained in detail, and then in the section after
we consider the general case $a > 3$  in  Proposition \ref{prop:g=gan+2} and indicate how the proof of the case $a = 2$ can be adapted 
to the general case of $a \geq 2$.
 
\section{Reformulation of Rybicki's identity ``$[g] = \left[g^{2^{n+2}}\right]$''}
\label{sec:g=g2n+2}
 
In this section, we will provide
a reformulation of the aforementioned Rybicki's identity ``$[g] = \left[g^{2^{n+2}}\right]$'' 
and then give its proof closely following his proof from \cite{rybicki2}.

We first introduce the following collection of subsets of $\R^{2n+1}$:
where 
\be\label{eq:InA}
\CI_{n;A} = \left(\left[-\frac12,0\right] \cup \left[\frac14,\frac34\right]\right)^n \times 
[-2A,2A]^n.
\ee
For  each $1 \leq \ell \leq n+1$ and $\delta > 0$, we define
\be\label{eq:Jelldelta}
\CJ_{\ell,\delta}: = \left(\left[-\frac14 - \delta, -\frac14 + \delta\right] \cup 
\left[\frac12 - \delta, \frac12 + \delta\right]\right)^\ell \times \R^n.
\ee
and
\be\label{eq:JelldeltaA}
\CJ_{\ell,\delta;A}: = \left(\left[-\frac14 - \delta, -\frac14 + \delta\right] \cup 
\left[\frac12 - \delta, \frac12 + \delta\right]\right)^\ell \times [-2A,2A]^n.
\ee
We mention that for $\ell = n+1$, we have
$$
\CI_{n+1,\delta;A} \subset I_A
$$
where we recall $I_A = [-2,2]^{n+1} \times [-2A,2A]^n$ is the reference rectangularpid.

Then, by the definitions of the map $\Xi^{(k)}$ (Definition \ref{defn:XiAk}) and of $g$ above in \eqref{eq:g},
the equality
\be\label{eq:g=f*}
g = \id + v\circ \pi_3 
\ee
holds on the union
\beastar
&{}& \left(\left[-\frac12 - \varepsilon, -\frac12 + \varepsilon\right]
 \cup [1-\varepsilon,
1+ \varepsilon]\right)^{n+1} \times [-2A,2A]^n\\
&\bigcup& \left(\left[-\frac18+\varepsilon,\frac18 + \varepsilon\right] \cup \left[\frac38-\varepsilon,\frac58 + \varepsilon\right] \right)
\times  \CI_{n;A} 
\eeastar
for some $\varepsilon > 0$ (Lemma \ref{lem:pikXiAk}).
Furthermore we have
\be\label{eq:supp-g}
\supp g \subset \left([-1,0] \cup \left[\frac12,\frac32\right]\right)^{n+1}
\times [-2A,2A]^n
\ee
and $\Theta^{(n>)}(g) = f^*$ by Proposition \ref{prop:rybicki82} (3).
The following is a key lemma toward the proof of Proposition \ref{prop:[g]=e},
which we call \emph{Rybicki's identity}  is one of the crucial element in the proof.

\begin{prop}[Equation (8.7) \cite{rybicki2}] \label{prop:key-lemma-2} Consider the case $a = 2$ and
let $g_2 = \Xi_{A;2}^{(<n)}(f^*)$. Then we have
$$
[g_2] = \left[g_2^{2^{n+2}}\right]
$$
in $H_1(\Cont_c(\R^{2n+1},\alpha_0)_0)$.
\end{prop}

The proof of this proposition will occupy the rest of this subsection whose 
details are rather tedious. Our proof closely follows but also
fixes some ambiguities of the argument used in the proof
of the identity given in \cite[p. 3317- 3318]{rybicki2}
by first making precise the meaning thereof and its notations, and then
much amplifying and optimizing the details of the proof thereof.

As an intermediate step, we will define a contactomorphism denoted by $g_2'$ that 
we will show simultaneously satisfies the following two equalities
\be\label{eq:[g2]=[g2']}
[g_2] = [g_2'] \quad \&  \quad [g_2'] = [g_2^{2^{n+2}}].
\ee
The definition of $g_2'$ will take $n+1$ steps starting from the zero-th step.
(Construction of the final element $g_2'$ is somewhat reminiscent of Mather's
construction performed in \cite[Section 3]{mather2}.)

As the zero-th step, we start with considering the conjugation
\be\label{eq:conjugationbyeta2}
h = \eta_2^{-1} g_2 \eta_2
\ee
by the front scaling map $\eta_2$. Then it follows from \eqref{eq:supp-g} that
$$
\supp h \subset \left(\left[-\frac12,0\right] \cup \left[\frac14,\frac34\right]\right) \times 
\CI_{n;A}
$$
where $\CI_{n;A}$ is as in \eqref{eq:InA}. We define the map 
\be\label{eq:f*1/2}
f^*_{\frac12} : = \id + \frac12 v\circ \pi_3
\ee
where $v$ is the map given in \eqref{eq:v}.

\begin{lem}\label{lem:h=f*2} We have
$h= f^*_{\frac12}$ on 
\beastar
\CJ_{n+1,\varepsilon/2,A} \bigcup 
\left(\left(\left[-\frac1{16}.\frac1{16}\right]\cup \left[\frac3{16}, \frac5{16} \right] \right)
\times \CJ_{n,\varepsilon/2, A}\right).
\eeastar
\end{lem}

\begin{rem}
Lemma \ref{lem:h0-append} and Lemma \ref{lem:h0th0-append} below
are stated in the course of the proof of \cite[Lemma 8.6]{rybicki2} without details of 
their proofs. This lack of the details makes Rybicki's proof thereof 
rather hard to digest. Largely for the purpose of convincing the current author himself and
for the convenience of the readers, we provide complete details of the proofs of these
sublemmata here partially because we need to generalize the arguments
for the case of $a > 2$ and also because some of the numerics appearing 
in the course of the proof of  \cite[Lemma 8.6]{rybicki2} are not explicitly given 
but need to be clearly understood for the extension to $a > 2$. The lack of details 
has prevented the present author from  penetrating the details of the proof 
and delayed the necessary generalization to the case of $a > 2$, 
until the present author himself rewrites all the details
given here. Other than these, the present subsection is largely a duplication of 
\cite[Lemma 8.6]{rybicki2} with some semantic improvement of its presentation 
in the case of $a = 2$.
\end{rem}

\subsection{Step 0 of the construction of $g_2'$: 2-fragmentation}

We take a fragmentation $h = \overline h_0 \widehat h_0$ similarly as in the definition of 
$\Xi_{A;a}^{(k)}$ so that
\be\label{eq:hbar0hhat1}
\overline h_0 = \begin{cases} h \quad \text{on } [-\frac12,0] \times \R^{2n} \\
\id \quad \text{there off}
\end{cases},
\quad\widehat  h_0 = \begin{cases}  h \quad \text{on } [\frac14,\frac34] \times \R^{2n}\\
\id \quad \text{there off}
\end{cases}
\ee
Then we define 
\bea\label{eq:h0}
h_0  =   \widehat h_0 \tau_{0,\frac12}\overline h_0 \tau_{0,\frac12}^{-1} 
\label{eq:h0-a}
\eea
and state a list of some technical properties of the map $h_0$  in the following list 
of lemmata that will enter into the proof of Proposition \ref{prop:key-lemma-2}.
\begin{lem}\label{eq:[h]=[h0]}
$[h_0] = [h]$ in $H^1(\Cont_c(\R^{2n+1}))$.
\end{lem}
\begin{proof} We compute
\beastar
h^{-1}h_0 & = & (\overline h_0\widehat h_0)^{-1}(\widehat h_0 \tau_{0,\frac12} \overline h_0 \tau_{0,\frac12}^{-1})\\
& = & \widehat h_0^{-1} \overline h_0^{-1} \widehat h_0 \tau_{0,\frac12} \overline h_0 \tau_{0,\frac12}^{-1}
=\left[\widehat h_0^{-1},\overline h_0^{-1}\right]  \left[\overline h_0^{-1},\tau_{0,\frac12}\right].
\eeastar
This finishes the proof.
\end{proof}

\begin{lem}\label{lem:h0-append}
We have
$$
h_0 = f^*_{\frac12} \quad \text{\rm on }\left[\frac14-\varepsilon,  \frac14 + \varepsilon\right]
 \times \CJ_{n,\varepsilon,A},
$$
 and 
$$
\supp h_0 \subset\left [0,\frac34\right] \times \CI_{n;A}
$$
provided $M_1^*(f) < \delta$, for a sufficiently small $\delta> 0$.
\end{lem}
\begin{proof}
We have $\supp \widehat h_0 \subset \left[\frac14,\frac34 \right]$ by \eqref{eq:hbar0hhat1}
and hence
$$
\supp(\tau_{0,\frac12} \overline h_0 \tau_{0,\frac12}^{-1}) 
\subset \left(\left [-\frac14, \frac14\right] \times \R^{2n}\right).
$$
Let 
$$
x = (z,q,p) \in  \left[\frac14-\varepsilon, \frac12 + \varepsilon \right] \times
\left( \left[-\frac14-\varepsilon, -\frac14 + \varepsilon\right] \cup \left[\frac12 - \varepsilon,\frac12+\varepsilon\right]\right)^n
\times [-2A,2A]^n.
$$
(This rectangularpid is nothing but $\CJ_{n,\varepsilon,A}$ from \eqref{eq:JelldeltaA} for
$(\ell,\delta) = (n,\varepsilon)$.)
Obviously
$$
\tau_{0,\frac12}^{-1}(x) = \left(z- \frac12, q, p\right), \quad z- \frac12 
\in \left[-\frac14 - \varepsilon, \varepsilon \right]
$$
on which $\widehat h_1 = \id$.
Therefore we obtain
$$
h_0(z,q,p) = \overline h_0 \tau_{0,\frac12}\left(z - \frac12, q,p\right) 
= \overline h_0\left(z - \frac12, q,p\right).
$$
Furthermore $\overline h_0 = h$  on $\left[-\frac14 - \varepsilon, 0 \right]$
 by \eqref{eq:hbar0hhat1} (for $\overline h_0$) and so
$$
h_0(z,q,p) = h(z,q,p).
$$
Finally, Lemma \ref{lem:h=f*2} proves $h = f_{1/2}^*$ for 
$z \in \left[\frac14-\varepsilon, \frac12 + \varepsilon\right] \subset \left[\frac38, \frac58\right]$,
provided $\delta > 0$ is sufficiently small.

For the support property, suppose
$$
(z,q,p) \not \in \left[0,\frac34\right] \times \CI_{n;A} = \left[0,\frac34\right] \times
\left(\left[-\frac12,0\right] \cup \left[\frac14,\frac34\right]\right)^n \times [-2A,2A]^n.
$$
Again a direct evaluation, which we omit the straightforward details, shows
$h_0(z,q,p) = (z,q,p)$.  This confirms the required support property.

\end{proof}

\begin{lem}\label{lem:h0th0-append}
We have:
\begin{enumerate}
\item $h_0 \tau_{0,\frac12} h_0 = f^*_{\frac12}$ on  $[0,\frac14] \times \R^{2n}$,
\item $h_0 \tau_{0,\frac12} h_0 = \eta_2^{-1} (\Xi^{<n-1)}(f^*)) \eta_2$ on $[0,\frac14] \times \R^{2n}$.
\end{enumerate}
Here $\Xi^{(<n-1)}(f^*) \in \Cont(\CW_1^{2n+1},\alpha_0)$ is viewed as an element of 
$\Cont(\R^{2n+1},\alpha_0)$ with period 1 with respect to the $z$ variable.
\end{lem}
\begin{proof} By the fragmentation $h = \overline h_0 \widehat h_0$ 
and the definition of $h_0$, we derive
$$
 h_0 \tau_{0,\frac12}h_0 =  \overline h_0 \tau_{0,\frac12}
h \tau_{0,\frac12} \widehat h_0 \tau_{0,\frac12}^{-1}
$$
by inserting the definition into the left hand side. On the other hand, on
$\left[0,\frac14\right] \times \CJ_{n,\varepsilon}$, we have 
$$
z - \frac12 \in \left[-\frac12,-\frac14\right].
$$
Then \eqref{eq:hbar0hhat1} (for $\widehat h_0$) implies
$$
\widehat h_0\tau_{0,\frac12}^{-1}(z,q,p) = \left(z - \frac12,q,p\right).
$$
This in turn implies 
$$
\tau_{0,\frac12} \widehat h_0\tau_{0,\frac12}^{-1}(z,q,p) = (z,q,p)
$$
and hence
\bea\label{eq:h0tauh0}
h_0 \tau_{0,\frac12} h_0(z,q,p) & = & \overline h_0 \tau_{0,\frac12}h(z,q,p) \nonumber\\
& = & \overline h_0 \tau_{0,\frac12}\eta_2^{-1}g\eta_2(z,q,p).
\eea
Since $g = \Xi^{(< n)}(f^*) = f^*$ for $(z,q,p) \in \left[\frac38-\varepsilon,\frac58+\varepsilon\right]$,
we compute
\be\label{eq:gconjugate=f*12}
\eta_2^{-1}g \eta_2(z,q,p) = f^*_{\frac12} (z,q,p)
\ee
thereon. From \eqref{eq:f*1/2}, we derive
$$
|z\left(f^*_{\frac12} (z,q,p)\right) - z| \leq \frac12 \|v\|_{C^0} = C\, M_1^*(f).
$$
Therefore if $M_1^*(f) < \delta$ for a sufficiently small $\delta > 0$, we have
$$
\left[\frac38-\varepsilon,\frac58+\varepsilon\right] \cap \supp \overline h_0 = \emptyset.
$$
This, combined with \eqref{eq:h0tauh0} and \eqref{eq:gconjugate=f*12}, implies 
\beastar
\overline h_0 \left(\tau_{0,\frac12}(h(z,q,p))\right) 
& = & \tau_{0,\frac12}\left(z + \frac12f_0^*(p), q + \frac12f_1^*(p), p\right) \\
& = & \left(z + \frac12 + \frac12f_0^*(p), q + \frac12f_1^*(p), p\right).
\eeastar
Combining the above discussions, we have finished the proof Statement (1).

For Statement (2), we start with \eqref{eq:fundamental-eq}
$$
h_0 \tau_{0,\frac12} h_0(z,q,p) = \overline h_0 \tau_{0,\frac12}\eta_2^{-1}g\eta_2(z,q,p).
$$
Writing $\Xi_A^{(n)} = \Xi_{A;2}^{(n)}$ and
substituting of $g = \Xi_{A}^{(<n)}(f^*) = \Xi_{A}^{(< n-1)} \circ \Xi_{A}^{(n)}(f^*)$ thereinto
makes the right hand side thereof become
\beastar
h_0 \tau_{0,\frac12} h_0(z,q,p)& = & (\overline h_0 \tau_{0,\frac12}\eta_2^{-1}) \circ \Xi_A^{(< n-1)} \circ( \Xi_A^{(n)} (\eta_2(z,q,p))) \\
 & = &  (\overline h_0 \tau_{0,\frac12}\eta_2^{-1})\circ \Xi_A^{(< n-1)} (\eta_2(z,q,p))) \\
 & = &\left((\overline h_0 \tau_{0,\frac12}\eta_2^{-1})\circ\Xi_A^{(< n-1)} \circ \eta_2\right)
 (z,q,p)\eeastar
If $z \in [0,\frac14]$,  $(z,2q,2p) \not \in \supp \Xi_A^{(n)}$ and hence
$$
\Xi_A^{(n)} (\eta_2(z,q,p))) = \eta_2(z,q,p).
$$
On the other hand, Statement (1) proves 
$h_0 \tau_{0,\frac12} h_0(z,q,p) = f_{\frac12}^*(z,q,p)$ when $z \in [0,\frac14]$.
Combining the two, we have derived
$$
\eta_2^{-1}\circ\Xi_A^{(< n-1)} \circ \eta_2 (z,q,p) 
= (\overline h_0 \tau_{0,\frac12})^{-1}(f_{\frac12}^*(z,q,p)).
$$
We compute
$$
 (\overline h_0 \tau_{0,\frac12})^{-1}(f_{\frac12}^*(z,q,p))
 = (\tau_{0,\frac12})^{-1} \overline h_0^{-1}
 \left(z + \frac12 + \frac12 f_0^*(p), q + \frac12 f_1^*(p),p\right).
 $$
 Since $z + \frac12 +  \frac12 f_0^*(p) \in [- \frac12 \|f_0^*\|, 1 + \frac12 \|f_0^*\|]$
 and $\overline h \equiv \id$ on $ [0,2] \setminus  [- \frac12 +\varepsilon, \varepsilon]$,
 $$
 \overline h_0^{-1}\left(z + \frac12 + \frac12 f_0^*(p), q + \frac12 f_1^*(p),p\right) 
 = \left(z + \frac12 + \frac12 f_0^*(p), q + \frac12 f_1^*(p),p\right).
 $$
 Then we derive
 $$
  (\tau_{0,\frac12})^{-1} \overline h_0^{-1}\left(z + \frac12 + \frac12 f_0^*(p), q + \frac12 f_1^*(p),p\right) 
 = \left(z  + \frac12 f_0^*(p), q + \frac12 f_1^*(p),p\right) = f_{\frac12}^*(z,q,p)
 $$
 on $[0,\frac14] \times \R^{2n}$.
 Combining the last 4 equalities, we have finished the proof of Statement (2).
 \end{proof}

\subsection{Downward induction for $k_\ell$ with $\ell$ from $n$ to $0$}

Now we define 
\be\label{defn:k0}
k_0 = h_0 \tau_{0,\frac12} h_0 \tau_{0,\frac12}^{-1} = h_0^2[h_0^{-1},\tau_{0,\frac12}]
\ee 
by considering the
variable $\xi_0$.  By the 
similar arguments used in the study of $h_0$ above, verification of 
the  following list of properties is straightforward,
\begin{enumerate}
\item $\supp(k_0) \subset [0,\frac54] \times \CJ_{n,\varepsilon,A}$,
\item $k_0 = h $ on $[\frac14 - \varepsilon, 1+\epsilon] \times \CJ_{n,\varepsilon,A}$,
\item $k_0 \tau_{0,1} k_0 = f^*_{\frac12} $ on $[0,\frac14] \times \CJ_{n,\varepsilon, A}$,
\item $k_0 \tau_{0,1} k_0 = \eta_2^{-1} \Xi^{(<n-1)}(f^*) \eta_2$ on $[0,\frac14] \times \R^{2n}$,
\item $\Theta^{(0)}(k_0) =f^*_{\frac12}$ on $S^1 \times \CJ_{n,\varepsilon,A}$,
\item  $\Theta^{(0)}(k_0) = \eta_2^{-1} \Xi^{(<n-1)}(f^*) \eta_2$ on $[0,\frac14] \times \R^{2n}$
on $\CW_1^{2n+1}$,
\item $[k_0] = [h_0^2] = [h^2] = [g_2^2]$.
\end{enumerate}
The last equality of Statement (7) follows from  Lemma \ref{lem:h0-append}.

Next, starting with $k_0$, by replacing $h$ by $k_0$, we define $\overline h_1, \, \widehat h_1, \, h_1$ 
and $k_1$ analogously as above, but now with respect to the variable $\xi_1= q_1 (\mod 1)$. Then we define
$$
k_1 = h_1 \tau_{1,\frac12} h_1 \tau_{1,\frac12}^{-1} = h_1^2[h_1^{-1},\tau_{1,\frac12}]. 
$$
It satisfies
\begin{enumerate}
\item $\Theta^{(1)}(k_1) =  f^*_{\frac12}$ on $T^{2}  \times \CJ_{n,\varepsilon,A}$,
\item $\Theta^{(1>)}(k_1) = f^*_{\frac12}$ on $T^{2}\times \CJ_{n-1,\varepsilon,A}$,
\item  $\Theta^{(1>)}(k_1) = \eta_2^{-1} \Xi^{<n-2)}(f^*) \eta_2 
= f^*_{\frac12}$ on $\CW_2^{2n+1}$.
\item $[k_1] = [k_0]^2 =  [h^4] = [g_2^4]$.
\end{enumerate}

Continuing this procedure inductively, we obtain a sequence 
$$
h_2, k_2, \ldots, h_n, k_n \in \Cont_c(\R^{2n+1},\alpha_0)_0
$$
such that it satisfies
\begin{enumerate}
\item $\Theta^{(n)}(k_n) =  f^*_{\frac12}$ on $T^{n+1}  \times \CJ_{1,\varepsilon,A}$,
\item $\Theta^{(n >)}(k_n) = f^*_{\frac12}$ on $T^{n+1}\times \CJ_{0,\varepsilon,A}$,
\item  $\Theta^{(n>)}(k_n) = \eta_2^{-1} \Xi^{(0)}(f^*) \eta_2 
= f^*_{\frac12}$ on $\CW_{n+1}^{2n+1}$.
\item $[k_n] = [k_{n-1}]^2  = [g_2^{2^{n+1}}]$.
\end{enumerate}
Finally we define
\be\label{eq:g2-defn}
g_2' : = \tau k_n \tau^{-1} k_n
\ee
for a suitable translation in the direction of  $(\xi_0,\xi)$ as 
in Section \ref{sec:unfolding-fragmentation} so that $g_2'$ has its support that is
a  disjoint union of connected intervals of the same size. (See the proof of 
\cite[Lemma 8.4]{rybicki2}.)
Then we summarize the above discussion into the following.
\begin{lem}
$g_2'$ satisfies the second equality of \eqref{eq:[g2]=[g2']}.
\end{lem}
This finishes the proof of Proposition \ref{prop:key-lemma-2} for the case $a = 2$.

\section{The identity $[g_a] = [g_a^{a^{n+2}}]$ for $a > 2$}

The above process of defining $g_a'$ for $a > 2$ starting from 
$g_a =\Xi_{A;a}^{(<n)}(f^*)$ can be applied verbatim 
for any integer $a \geq 3$ utilizing the $a$-fragmentation operator $\Xi_{A;a}^{(k)}$
defined in Section \ref{sec:rolling-up}
with the replacement of $N=2$ by $N=a$.

 Our goal is to prove the following, which is the counterpart of Proposition \ref{prop:key-lemma-2}.
 
\begin{prop} \label{prop:g=gan+2} 
Let $a \geq  2$ and consider 
 $g_a : = \Xi_{A;a}^{(<n))}(f^*)$. Then
$$
[g_a] = \left[g_a^{a^{n+2}}\right]
$$
in $H_1(\Cont_c(\R^{2n+1},\alpha_0)_0)$.
\end{prop}
The same kind of proof with the replacement of $2$ by an arbitrary integer $a \geq 2$
with some changes can be given to generalize Proposition \ref{prop:key-lemma-2} as follows.

Again we will define a contactomorphism denoted by $g_a'$ that 
 satisfies the two equalities 
\be\label{eq:[ga]=[ga']}
[g_a'] = [g_a] \quad \&  \quad [g_a'] = [g_a^{a^{n+2}}]
\ee
and the definition of $g_a'$ from $g_a$ will take $n+1$ steps. 
We will briefly indicate necessary changes to be made from \eqref{eq:[g2]=[g2']}.

Let $a > 2$ be any given integer.
As the zero-th step, we start with considering the conjugation
$$
h = \eta_a^{-1} g_a \eta_a
$$
by the front scaling map $\eta_a$ similarly as in \eqref{eq:conjugationbyeta2}.

We consider the interval $[-1,1]$ into $a$ pieces of subintervals of length 2
and then scale them back by the ratio $a$
$$
I_j^a = \frac1a [-a+2(j-1),-a+ 2j] = \left[-1 + \frac{2(j-1)}{a}, -1 + \frac{2j}{a}\right], \quad j=1, \ldots, 
$$
and take the  union of their `halves' 
$$
(I_j^a)' = \frac12 I_j^a, $$
 and the union
$$
(I')^a : = \bigcup_{j=1}^a (I_j^a)'.
$$
Then we consider the bump
function $\psi_k$ defined on $[-a,a]$ in \eqref{eq:psi-k} 
and extend periodically to whole $\R$.

 With these preparations, the process of defining an element 
$g_a' \in \Cont_c(\R^{2n+1},\alpha_0)$ satisfying $[g_a'] = [g_a^{a^{n+2}}]$ 
for general $a> 3$ is entirely similar to the case of $a = 2$. After its construction, 
we will also have established
$$
\omega(f) = [g_a^{a^{n+2}}].
$$
We start with the following 
\begin{lem} \label{lem:[g]=[g_a]} $[g_a'] = \omega(f)$
\end{lem}
\begin{proof}
In the course of the proof of Lemma \ref{lem:[ga]=[ga']}, we have established
$$
\Theta^{(n>)}(g_a) = f^*.
$$
By applying Lemma \ref{lem:rybicki86(1)} and the definition of $g$ in turn, we obtain
$$
[g_a'] = [\Xi_{A;a}^{(<n)}(f^*)] = [g_a] = \omega(f)
$$
where the last equality comes from Lemma \ref{lem:[ga]=[ga']} and the definition of $\omega(f) 
\in \Cont_c(\R^{2n+1},\alpha_0)$.
\end{proof}

Again the definition of $g_a'$ will take $n+1$ steps starting from the zero-th step.
As the zero-th step, we start with considering the conjugation
$$
h = \eta_a^{-1} g_a \eta_a.
$$ 
Then the proof of the following lemma is similar to that of Lemma \ref{lem:h=f*2}
\begin{lem}\label{lem:h=f*a}
We define the map $f^*_{\frac1a}$ by
$$
f^*_{\frac1a}(z,q,p) := \left(z + \frac1a f_0^*(p), q + \frac1a f_1^*(p), p\right).
$$
Then
$$
h(z,q,p) = f^*_{\frac1a}(z,q,p)
$$
on 
$$
\CJ_{n+1,\varepsilon/2,A} \bigcup 
\left(\left(A^a_{\frac1{2a}}(\text{\rm ev}) \cup\left(A^a_{\frac1{2a}}(\text{\rm ev}) + \frac1{2a}\right)  
\right)
\times \CJ_{n,\varepsilon/2, A}\right)
$$
where we define
$$
A^a_{\text{\rm ev}} := \bigcup_{i=1}^a A^a_{\frac{1}{2a}}\left(\frac{c_{2j}}{a}\right).
$$
Recall the definition \eqref{eq:AkN} for the interval $A^N_{\frac{1}{2N}}$ in general.
\end{lem}

\subsection{Step 0 of the construction of $g_a'$: $a$-fragmentation}

We take an $a$-fragmentation $h = \overline h_0 \widehat h_0$ similarly as in the definition of 
$\Xi_{A;N}^{(k)}$ so that
\be\label{eq:hbar0hhat1-a}
\overline h_0 = \begin{cases} h \quad \text{on }  \frac12 I^a  \times \R^{2n} \\
\id \quad \text{there off}
\end{cases},
\quad\widehat  h_0 = \begin{cases} h \quad \text{on } 
\left(\left(A^a_\text{\rm ev} + \frac1{2a}\right)+ \frac1{2a}\right)\times \R^{2n}\\
\id \quad \text{there off}
\end{cases}
\ee
where by definition we have
$$
\frac12 I^a = \bigcup_{i=1}^{N} \left[-1 +\frac{2i-1}{a}, -1  + \frac{2i+1}{a} \right], \quad 
A^a_\text{\rm ev} + \frac1{2a} = \bigcup_{j=1}^a A^a_{\frac{1}{2a}} \left(\frac{c_{2j}}{a} 
+  \frac1{2a}\right)
$$
In particular, we can further decompose $\overline h_0$ into
$$
\overline h_0 = \overline h_{0,1} \cdots \overline h_{0,a}
$$
so that their supports are pairwise disjoint.
Then we define the counterpart of \eqref{eq:h0} for the $a$-fragmentation to be
\be\label{eq:h0-a}
h_0
= \widehat h_0 \left(\tau_{0,\frac1a}\overline h_{0,1} \tau_{0,\frac1a}^{-1}\right)  
\left(\tau_{1,\frac1a}\overline h_{0,2} \tau_{1,\frac1a}^{-1}\right) 
 \cdots \left(\tau_{n,\frac1a}\overline h_{0,n} \tau_{n,\frac1a}^{-1}\right).
\ee

We state a list of the counterparts of the properties of the map $h_0$ 
for the case of $a > 2$ in the following list 
of lemmata that will enter into the proof of  Lemma \ref{lem:[ga]=[ga']}.
We omit their proofs since they are entirely similar to the case of $a = 2$
once the correct statements for $a > 2$ are made.

\begin{lem}\label{lem:[h0]-append-a} 
We have  $[h_0] = [h]$.
\end{lem}

\begin{lem}\label{lem:h0-append-a}
We have
$$
h_0 = f^*_{\frac1a} \quad \text{\rm on } \frac1{4} I^a \times \CJ_{n,\varepsilon,A},
$$
 and 
$$
\supp h_0 \subset
 \frac34 I^a \times \CJ_{n,\varepsilon,A}
 $$
provided $M_1^*(f) < \delta$ for a sufficiently small.
\end{lem}

\begin{lem}\label{lem:h0th0-append-a}
We have:
\begin{enumerate}
\item $h_0 \tau_{0,\frac1a} h_0 \tau_{0,\frac2a} \cdots \tau_{0,\frac{a-1}{a}} h_0
= f^*_{\frac1a}$ on  $ \frac14A^a_{\text{\rm ev}}  \times \R^{2n}$,
\item $h_0 \tau_{0,\frac1a} h_0 \tau_{0,\frac2a} \cdots \tau_{0,\frac{a-1}{a}} h_0 = \eta_a^{-1} (\Xi^{<n-1)}(f^*)) \eta_a$ on 
$\frac14 A^a_{\text{\rm ev}}  \times \R^{2n}$.
\end{enumerate}
Here $\Xi^{(<n-1)}(f^*) \in \Cont(\CW_1^{2n+1},\alpha_0)$ is viewed as an element of 
$\Cont(\R^{2n+1},\alpha_0)$ with period 1 with respect to $z$ variable.
\end{lem}

\subsection{Downward induction for $a > 2$}

Now we define 
$$
k_0 = h_0 \tau_{0,\frac1a} h_0 \tau_{0,\frac2a} \cdots \tau_{0,\frac{a-1}{a}} h_0
$$
using the variable $\xi_0$.  By the 
similar arguments used in the study of $h_0$ for the case $a=2$ above, verification of 
the following list of properties is straightforward,
\begin{enumerate}
\item $\supp(k_0) \subset \bigcup_{j=-[\frac{a+1}{2}]}^{\frac{[a+1}{2}]} \left(\frac{2j}{a} + \frac1a[0,\frac54]\right)
 \times \CJ_{n,\varepsilon,A}$,
\item $k_0 = h $ on 
$\bigcup_{j=-[\frac{a+1}{2}]}^{[\frac{a+1}{2}]} \left(\frac{2 j}{a} + \frac1a[\frac14 - \varepsilon, 1+\epsilon]\right) 
 \times \CJ_{n,\varepsilon,A}$,
\item $k_0 \tau_{0,1} k_0 = f^*_{\frac12} $ on 
$\bigcup_{j=-[\frac{a+1}{2}]}^{[\frac{a+1}{2}]} \left(\frac{2 j}{a} + \frac1a [0,\frac14] \right) \times \CJ_{n,\varepsilon, A}$,
\item $k_0 \tau_{0,1} k_0 = \eta_2^{-1} \Xi^{(<n-1)}(f^*) \eta_2$ on 
$\bigcup_{j=-[\frac{a+1}{2}]}^{[\frac{a+1}{2}]} \left(\frac{2 j}{a} + \frac1a [0,\frac14] \right) \times \R^{2n}$,
\item $\Theta^{(0)}(k_0) =f^*_{a}$ on $S^1 \times \CJ_{n,\varepsilon,A}$,
\item  $\Theta^{(0)}(k_0) = \eta_2^{-1} \Xi^{(<n-1)}(f^*) \eta_2$ on 
$\bigcup_{j=-[\frac{a+1}{2}]}^{[\frac{a+1}{2}]} \left(\frac{2j}{a} +\frac1a [0,\frac14]\right)  \times \R^{2n}
\subset \CW_1^{2n+1}$,
\item $[k_0] = [h_0^a] = [h^a] = [g_2^a]$.
\end{enumerate}

Now we define
\be\label{eq:g2'-defn}
g_2' : =  \left (\tau_{0,\frac1a} k_n \tau_{0,\frac1a}^{-1}\right) 
\left(\tau_{0,\frac2a} h_n \tau_{0,\frac2a}^{-1}\right)
 \cdots \left (\tau_{0,\frac{a-1}{a}} k_n \tau_{0,\frac{a-1}{a}}^{-1}\right) k_n.
\ee
Then we summarize the above discussion into the following.
\begin{lem}
$g_a'$ satisfies $[g_2']=[g_a^{a^{n+2}}]$.
\end{lem}
This finishes the proof of Lemma \ref{prop:g=gan+2} for the general cases $a > 2$.

\subsection{Wrap-up of the proof of {$\omega(f) = e$}}

After the homological identity from Lemma \ref{prop:g=gan+2} is established,
we wrap up the proof of Proposition \ref{prop:[g]=e} utilizing a simple number theoretic
argument as in \cite{rybicki2}. We have 
$$
[g_a] = [g_a^{a^{n+2}}] 
$$
from Proposition \ref{prop:g=gan+2} which implies 
$$
\omega(f) = \omega(f)^{a^{n+2}}.
$$
Therefore we have derived that either $\omega(f) = e$, which will finish the proof,
or otherwise $\text{\rm ord}(\omega(f))= \ell_0 > 1$ and $\omega(f)$ satisfies
$$
(\omega(f))^{a^{n+2} - 1} = e.
$$
From now on, suppose the latter holds for $\ell_0 > 1$.
Since this holds for every integer $a \geq 2$, we also have 
$$
\omega(f)^{a^{n+2} - b^{n+2}} = e
$$
and hence $\ell_0 \mid a^{n+2} - b^{n+2}$ for every pair of positive integer $(a,b)$ with 
$a, \, b \geq 1$. In particular it also holds for $(a,b) = (\ell_0,1)$, i.e.,  $\ell_0$ divides
$$
\ell_0^{n+2} - 1^{n+2} = \ell_0^{n+2} - 1.
$$
This contradicts to the simple fact $\ell_0$ is not a divisor of $\ell_0^{n+2} -1$,
 since we assume $\ell_0 > 1$.
Therefore we conclude that $\omega(f) = e$ which finishes the proof of
Proposition \ref{prop:[g]=e}.

\begin{rem}\label{rem:mystery} The way how this homological identity is used in 
the proof is rather peculiar which the author feels deserves more scrutiny on 
its meaning. It seems to the author that it is a replacement of the more common
practice of infinite repetition construction which is also used, for example by Tsuboi
\cite{tsuboi3}, in his perfectness proof of $\Cont^r_c(M,\alpha)$ for the opposite case of $r < n+ \frac32$
of the threshold.
\end{rem}

\section{Wrap-up of the proof of Theorem \ref{thm:perfect} for $r > n+2$}
\label{sec:wrap-up}

We are now ready to wrap up the proof of perfectness of 
$\Cont_c^*(\R^{2n+1}, \alpha_0)$ combining the arguments used by 
Mather \cite[Section 3]{mather,mather2}, \cite{epstein:commutators} and \cite[Section 9]{rybicki2} 
for all $(r,\delta)$ with $r \geq 1$,
$0 < \delta \leq 1$ for $(r,\delta) \neq (n+1, \frac12)$.
Let $f_0 \in \Cont_c(\R^{2n+1},\alpha_0)$.

We need to show that $f_0$ belongs to the commutator subgroup thereof. By the fragmentation lemma, Lemma \ref{lem:fragmentation-intro},
we may assume that 
\be\label{eq:suppf0}
\supp(f_0) \subset I_A = [-2,2] \times [-2,2]^n \times [-2A,2A]^n.
\ee
Furthermore by a contact conformal rescaling, which does not change its conjugacy class,
we may assume that $M_r^*(f_0)$ is as small as we want.
Let $A > 0$ be a sufficiently large positive integer which is to be fixed later, and
let $I_A, \, J_A$ and $K_A$ be the intervals in $\R^{2n+1}$ given in \eqref{eq:IA}, \eqref{eq:JA} and
\eqref{eq:KA} respectively. Let $r \geq 1$ be given and define
\be\label{eq:CL}
\CL_r(\varepsilon,A) := \{ u \in C^{r+1}_{I_A}(\R^{2n+1}) \mid \|D^{r+1} u\| \leq \varepsilon \}
\ee
where $\varepsilon > 0$ is a sufficiently small constant which will be fixed in the course of the proof.
We observe that $\CL_r(\varepsilon,A)$ is a convex and compact subset of a locally convex space
$C^{r+1}_{I_A}(\R^{2n+1})$.

\begin{lem}\label{lem:vartheta-existence} Suppose that $r > n+2$. Then there
exist constants, a sufficient small $\epsilon_0> 0$ and a sufficiently large $A> 0$,
for which there exists a continuous map $\vartheta: \CL_r(\varepsilon,A) \to \CL_r(\varepsilon,A)$.
\end{lem}
\begin{proof} The proof of this lemma duplicates 
the 10 steps laid out by Rybicki \cite[Section 9]{rybicki2}. (This is the contact
replacement of Mather's strategy \cite[Section 3]{mather} that was applied to the case of 
diffeomorphisms $\Diff_c(M)_0$). 
 
We  may assume
$$
\text{\rm pr}_{1,2}(U)  \supset I_A
$$
as mentioned before
again after contact conformal rescaling of $f_0$. Then here are the aforementioned 
Rybicki's ten steps with some changes of various numerics appearing in 
the construction and \emph{with the change of the form of the contact scaling from $\chi_A\eta_A$ by
$\chi_{A^2}$} in Step (5):
\begin{enumerate}
\item For any $u \in \CL_r(\varepsilon,A)$, consider $f \in \Cont_c(\R^{2n+1},\alpha_0)$ given as
$f = \mathscr G_{A}^{-1}(u)$.
\item Set $g = f f_0$. Then we have the inequality 
\be\label{eq:step2}
M_r^*(g) \leq C \|u\|_{r+1}
\ee
from \eqref{eq:Mr*f-ur+1}.
\item Use a fragmentation of the second kind for $g = ff_0$ 
(Proposition \ref{prop:2nd-fragmentation})
and obtain a fragmentation $g = g_1\circ \cdots \circ g_{a_n}$ with $a_n = (8A^2 + 8)$, and
each $g_K$ is supported in 
$$
\left([-2,2]^{n+1} \times [k_1 -1, k_1+1] \times \cdots \times [k_1-1,k_n+1] \right) \cap I_A
$$
with integers $k_i$ such that $|k_i| \leq 2A-1$, $i =1, \ldots, n$.
\item Use the operation of shifting supports of contactomorphisms described Section 
\ref{sec:shifting-supports}. For any $K = 1, \ldots, a_n$, we define
\beastar
\widetilde g_K &:=& \sigma_{A,\bf t} \circ g_K \circ \sigma_{A,\bf t}^{-1}\\
& = & \left (\sigma_{n,t_n} \left( \sigma_{n-1,t_{n-1}} \left(
\cdots \left(\sigma_{1,t_1} g_K \sigma_{1,t_1}^{-1}\right) \cdots \right)
\sigma_{n-1,t_{n-1}}^{-1}\right)\sigma_{n,t_n}^{-1} \right)
\eeastar
for a suitable ${\bf t} = (t_1,\cdots,t_n) \in \R^n$ depending on $K$ in such a way that
$$
\supp(\widetilde g_K) \subset [-A^5, A^5] \times [-2,2]^{2n}
$$
for all $K$. Here we take $|t_i| \leq 2 A-1$, $i = 1, \cdots, n$ and $A > 5n$. 
\item For each $K = 1, \ldots, a_n$, define the conjugation
 $$
 h_K = (\chi_{A^2})\widetilde g_K (\chi_{A^2})^{-1}, \quad 
 \supp (h_K) \subset J_A.
 $$
 Then we have 
 \be\label{eq:hk=tildegk}
 [h_K] = [\widetilde g_K]
 \ee
 and  the  inequality
\be\label{eq:step4}
M_r^*(\widetilde h_K) \leq C A^{4-2r} M_r^*(g)
\ee
from Proposition \ref{prop:rybicki2}.
 
 \item Apply the rolling-up operator $\Psi_A$ described in Proposition \ref{prop:PsiAk} to define
 $\overline h_K = \Psi_A(h_K)$. We have $\supp(\overline h_K) \subset K_A$.
 We have
 \be\label{eq:step6}
 M_r^*(\overline h_K) \leq CA^2\, M_r^*(h_K)
 \ee
 from Corollary \ref{cor:MrThetakA} and Proposition \ref{prop:rybicki82-estimates}.
 
 \item Apply a fragmentation of the second kind in $K_A$ in the directions $i = 1, \ldots, n$.
 We write $\overline a_n = a_n^3$ and get the fragmentation of $\overline h_K$,
 $$
 \overline h_K = \overline h_{K;1} \circ \cdots \circ \overline h_{K;\overline a_n}.
 $$
\emph{In each step of taking the conjugation by $\sigma_{i,t_i}$  for 
$i = 1, \cdots, n$}, the power of $A$ moves up by 2. Therefore we have
\be\label{eq:step7}
M_r^*(\overline h_K) \leq CA^{2n} M_r^*(h).
\ee
 \item Apply the operation of shifting supports of contactomorphisms in the $q_i$-directions
 by the translations $\tau_i$, $i=1, \ldots, n$. For each pair $K, \, i$, we define
 $\widetilde h_{K;i}$ instead of $\overline h_{K;i}$ with support
 $$
 \supp (\widetilde h_K) \subset I_A.
 $$
\emph{ All the norms of the latter map are the same as those of $\overline h_{K;i}$.}
 \item Take the product  and write 
 $$
 h = \prod_{K = 1}^{a_n} \prod_{i=1}^{\overline a_n}
 \widetilde h_{K;i}. 
 $$
 Then we have
 \be\label{eq:step9}
 M_r^*(h) \leq C A^{4-4r + 2n} M_r^*(g)
 \ee
 \item Take $u_h: = \mathscr G_A(h)$. Then 
 \be\label{eq:step10}
 \|u_h\|_{r+1} \leq C M_r^*(h)
 \ee
 from Proposition \ref{prop:derivative-uf} (1).
\end{enumerate}
Then we define the map
\be\label{eq:vartheta}
\vartheta(u): = u_h.
\ee
Combining the inequalities given in the above 10 steps, we have obtained
\be\label{eq:vartheta<}
\|\vartheta(u)\|_{r+1} \leq C A^{2(n + 2 - r)} \|u\|_{r+1}.
\ee
Therefore if $r > n+ 2$, we can choose $A > 0$ sufficiently large
(recalling that we also choose $\delta > 0$ and the Darboux-Weinstein chart
$\Phi_U: U \to V$ sufficiently small), we can make the inequality 
$$
 C A^{2(n + 2 - r)} < 1
 $$
 holds. This finishes the construction of the map 
 $\vartheta: \CL_r(\varepsilon,A) \to \CL_r(\varepsilon,A)$.
  \end{proof}

Once this lemma is established, Schauder-Tychonoff theorem implies that 
any such continuous map $\vartheta: \CL_r(\varepsilon,A) \to \CL_r(\varepsilon,A)$ 
carries a fixed point. The rest of the proof is the same as Rybicki's 
laid out in \cite[Section 9]{rybicki2}, especially the first half thereof, 
except that we again need to  incorporate the fact that the map $\vartheta$ itself
depends on the integer $a \geq 2$. Since we will fix $a$ in the following paragraph,
we just write $g_a = g$.

Let $u \in \CL_r(\varepsilon,A)$ be a fixed point of $\vartheta$, i.e., $\vartheta(u) = u$.
Denote by $f = \mathscr G_A^{-1}(u) \in \CU_1 \subset \Cont_c(\CW_k^{2n+1},\alpha_0)$ 
and  $u = u_f$. By definition of the map $\vartheta = \vartheta_{f_0}$ associated to $f_0 \in \Cont_c^r(\R^{2n+1},\alpha_0)$
defined by the above 10 steps, we obtain the following sequence of identities:
\beastar
[ff_0] & = & [g] = [g_1\cdots g_{a_n}] = [g_1]\cdots [g_{a_n}] 
= [\widetilde g_1]\cdots[\widetilde g_{a_n}]\\
& =& [h_1]\cdots [h_{a_n}] = [\overline h_1\cdots[\overline h_{a_n}] \\
& = & [\overline h_{11}]\cdots [\overline h_{a_n\overline a_n}] 
=  [\widetilde h_{11}]\cdots [\widetilde h_{a_n\overline a_n}] \\
& = & [\widetilde h_{11}\cdots\widetilde h_{a_n\overline a_n}] =[h] = [f].
\eeastar
Here the 5th equality follows from 
\eqref{eq:hk=tildegk} and the 8th equality from 
Proposition \ref{prop:main} (3). The last equality is a consequence 
of the definition $f = \mathscr G_A^{-1}(u)$
for the fixed point $u$ of the map $\vartheta$ by the standing hypothesis
$\vartheta(u) = u$. For by definition of $\vartheta$, we also have
$\vartheta(u) = \mathscr G_A(h)$. Since $\mathscr G_A$ is a bijective map,
this implies $h = f$. All other equalities are either trivial or consequences of 
Lemma \ref{lem:rybicki2} and Proposition \ref{prop:PsiAk}

Therefore we have proved $[f_0] = e$ in $H_1(\Cont_c(\R^{2n+1},\alpha_0)_0)$. 
This completes the proof of Theorem \ref{thm:perfect} for $r > n+2$.

\section{Proofs of Theorem \ref{thm:hoelder-up} and Theorem \ref{thm:hoelder-down}}

As for the diffeomorphism case of Mather \cite{mather,mather2}, Theorem \ref{thm:perfect}
is a consequence of Theorem \ref{thm:hoelder-up} which involves the function $\alpha$ of 
modulus of continuity, e.g., the function $\alpha(x) = x^\beta$ for the H\"older regularity
$(k,\delta)$ with $0 < \delta \leq 1$. We refer readers to \cite{sanghyun} 
for a detailed study of the set of modulus of continuity which helps the authors thereof 
 systematically analyse  the threshold case $r = n+1$ in \cite{mather,mather2}.

\subsection{Proof of Theorem \ref{thm:hoelder-up} and Theorem \ref{thm:perfect} for $r = n+2$}

Finally, we explain how we can extend the proof of Theorem \ref{thm:perfect} to 
the H\"older regularity class $\Cont_c^{(r,\delta)}(\CW_k^{2n+1},\alpha_0)$ for all $(r,\delta)$ with 
$r = n+1$ and  $\frac12 < \delta \leq 1$. In fact, the proofs of the two theorems
do not make difference, observing that all the estimates performed in Part II can be equally carried out
for the H\"older class $(r,\delta)$ without change: The only change needed to make the following 
estimate
\be\label{eq:integer-hoelder}
\mu_{r,\delta}\left(\sigma_{1,t_1} g_K \sigma_{1,t_1}^{-1}\right) \leq C A^{2(1-r - 2\delta)} \mu_{r,\delta}^*(g_K).
\ee
(See \cite[p.518]{mather} for the similar change made to handle the case of $\Diff_c^{r,\delta}(M)_0$.)
Now this change will make Proposition \ref{prop:main} into one such that
the map
$$
\Psi_A: \Cont_{J_A}^{r,\delta}(\R^{2n+1},\alpha_0)_0 \cap \CU_4 \to \Cont_{K_A}^{r,\delta}(\R^{2n+1},\alpha_0)_0
$$
that satisfies the estimate
$$
M_r^*(\Psi_A(g)) \leq C K_r A^{2(1 -r - 2\delta +n)}  M_r^*(g)
+  P_{\chi,r}(M_{r-1}^*(g))
$$
which in turn gives rise to the same inequality of the map $\varepsilon = \varepsilon_{f_0}$.
This proves Theorem \ref{thm:hoelder-up} for the case of $r = n+1$ and $\frac12 <  \delta \leq 1$.
Finally the case for $r = n+2$ follows by the same argument of Mather \cite{mather} by
noticing the equality
$$
\Cont_c^{n+2}(\R^{2n+1},\alpha_0) = \bigcup_{0 \leq \delta < 1} \Cont_c^{(n+2,\delta)}(\R^{2n+1},\alpha_0).
$$

\subsection{Proof of Theorem \ref{thm:hoelder-down}}
\label{subsec:hoelder-down}

The case $r = n+1$ and  $1 \leq r+ \delta < n+ \frac32$ was previously proved by Tsuboi in \cite{tsuboi3}
and in particular for $1 \leq r \leq n+1$ for integer $r$.  His result follows from our 
proof by dualizing the construction similarly as Mather did for the diffeomorphism case.

Here are the key points of changes to be made in  the estimates for
this dual construction from the case of lower threshold are the following:
\begin{enumerate}
\item We just replace $A$ by $A^{-1}$ in the construction, which in particular reverse
the direction of the map $\Theta_A^{(k)}$ so that we now have the map
$$
\Theta_A^{(k)}:  \Cont_{J^k_A}(\CW_k^{2n+1},\alpha_0) \cap \CU_1 
\to \Cont_{K^k_A}(\CW_k^{2n+1},\alpha_0) \cap \CU_1.
$$
See Diagram \ref{eq:diagram}.
\item 
As Mather put it in \cite[Section 4, p.37]{mather2},
\emph{``.... estimate (1) is essentially a special case of (1) \cite[Section 6]{mather}. 
Here $\supp u \subset \text{\rm int} D_{i-1,A}$, whereas there, we have only the weaker
condition $\supp u \subset D_{i,A}$. This explains why we may omit $A$ from the right hand side
of the inequality here: the width of $D_{i-1,A}$ in the $i$th coordinate is $4$, while the width
of $D_{i,A}$ is $4A$."},  The outcome is that we do not need
the $A^2$ in \eqref{eq:step7} and so the corresponding equality becomes
$$
M_r^*(\overline h_K) \leq C M_r^*(h_K).
$$
\item Recall we have used the contact scaling map $\chi_{A^2} = \chi_A^2$ the norm of 
which is bounded by $A^4$ while the norm of 
its inverse is bounded by $A^{-2}$. This asymmetry  is responsible for the appearance of 
$2\delta$ for the case of lower threshold and $\delta$ for the case of upper threshold below.
\item We remind the readers that the domain of the map $\vartheta$ is 
$$
I_A = [-2,2] \times [-2,2]^n \times [-2A,2A]^n
$$ 
which plays the role of the reference space 
that normalizes the conformal factor of the front projection $[-2,2] \times [-2,2]^n$ throughout
the constructions.  
\end{enumerate}

The final outcome is that the inequality \eqref{eq:vartheta<} is then transformed by
$$
\|\vartheta(u)\|_{r +1,\delta} \leq C A^{2(-2 - r-  n)}
$$
on the $C^r$ space, and
$$
\|\vartheta(u)\|_{r +1,\delta} \leq C A^{2(-2 - r  - n  + 2 \delta)} 
$$
on $C^{r,\delta}$ space.
(See \cite[p.516]{mather} for the relevant H\"older 
estimates for the diffeomorphism case.) We need either $r < n+1$ or $r = n+1$ which precisely gives rise to
the bound for $\delta$  given by
$$
- 1 + 2 \delta < 0
$$
which shows that  the H\"older regularity $(n+1, \delta)$  be in the required range
stated in Theorem \ref{thm:hoelder-up}.

Combining the above all, we have
finished the proof of Theorem \ref{thm:hoelder-down}.

\appendix

\section{Proof of Corollary \eqref{cor:equivariance}: equivariant contactomorphisms}
\label{sec:equivariance}

In this section, we give the proof of Corollary \ref{cor:equivariance} for completeness' sake.
We state the corollary here.

\begin{cor} We have the expression
$$
\Phi_U^{-1}(t,x,X)  = (t + h_t(t,X), x + h_x(t,X), x + h_X(t,X)) \in \R^{2(2n+1)+1} \cong J^1\R^{2n+1}
$$
such that $h_t(0,0) = 0$, $h_x(0,0) = 0$ and $h_X(0,0) = 0$.
\end{cor}
\begin{proof} Recall that $\Phi_U$ is $(\CG_1,\CG_2)$-equivariant, i.e, $\Phi_U^{-1}=: \varphi$ satisfies
$$
\varphi(t,x + g, X) = (\varphi_t(t,x,X), g + \varphi_x(t,x,X), g + \varphi_X(t,x,X)), \quad 
\varphi(0,x,0) = (0,x,x)
$$
where we write $\varphi = (\varphi_r, \varphi_x, \varphi_X)$ componentwise. 
Then we obtain the following system of equations
$$
\begin{cases}
\phi_t(t,x+g, X) = \phi_t(t, x,X) \\
\phi_x(t,x+g,X) = \phi_x(t,x,X) + g \\
\phi_x(t,x+g,X) = \phi_X(t,x,X) + g
\end{cases}
$$
for all $t,\, x\, X$ and $g \in \R^{2n+1}$. In particular, by plugging $x = 0$ into the equations,
we obtain
$$
\phi_t(t, g, X) = \phi_t(t, 0,X), \, 
\phi_x(t, g,X) = \phi_x(t, 0,X) + g\,  
\phi_x(t,g,X) = \phi_X(t,,X) + g.
$$
Since $g$ is arbitrary, we can put $g = x$ and then set
$$
h_t(t,X) = \varphi_t(t,0,X), \, h_x(t,X) = \varphi_x(t,0,X), \, h_X(t,X) = \varphi(t,0,X).
$$
This, $varphi(0,x,0) = (0,x,x)$ and
$$
T\varphi|_{(0,x,0)} = \id :\R \oplus \R^{2n+1}_x \oplus \R^{2n+1}_X
\to \R \oplus \R^{2n+1}_x \oplus \R^{2n+1}_X
$$
imply that we can write $\varphi$ in the form of 
$$
\varphi(t,x,X) = (t + h_t(t,X), x + h_x(t,X), x + h_X(t,X))
$$
with $h_t(0,X) = 0, \, h_x(0,X) = 0, \, h_X(0,X) = 0$. Here we identity
$$
T_{(0,x,0)}(J^1\CW_k^m) \cong T_{(0,x,x)}(M_{\CW_k^m})
$$
by parallel translations on $\R^{2(2n+1)+ 1}$.
This finishes the proof.
\end{proof}


\begin{thebibliography}{CDGM}

\bibitem[Ar]{arnold} Arnold, V. I.,  Mathematical Methods of Classical Mechanics, Second edition,
Translated from the Russian by K. Vogtmann and A. Weinstein,  Graduate Texts in Mathematics, 60. Springer--Verlag,
New York, 1989. xvi+508 pp.

\bibitem[Ba1]{banyaga} Banyaga, A. {\em Sur la structure du groupe des difféomorphismes qui pr\'eservent une forme symplectique},
Comment. Math. Helv. \textbf{53} (1978), no. 2, 174--227.

 \bibitem[Ba2]{banyaga:book}  \bysame,  {\em The structure of classical diffeomorphism 
 groups}, Math. Appl., 400,
Kluwer Academic Publishers Group, Dordrecht, 1997, xii+197 pp.]
\bibitem[Bhu]{bhupal}
M.~Bhupal, \emph{A partial order on the group of contactomorphisms of
  $\mathbb{R}^{2n+1}$ via generating functions}, Turkish J. Math. \textbf{25}
  (2001), 125--235.

\bibitem[CKK]{sanghyun} Chang, Jaewon and Kim, Sang-hyun and  Koberda, Thomas, 
{\em Algebraic structure of diffeomorphism groups of one-manifolds}, 
arXiv e-print (2019), arXiv:1904.08793.

\bibitem[CHS]{CHS}   Cristofaro-Gardiner, Daniel and Humili\`ere, Vincent and  Seyfaddini, Sobhan,
{\em Proof of the simplicity conjecture},
Ann. of Math. (2) \textbf{199} (2024), no. 1, 181--257.

\bibitem[E1]{epstein:simplicity} Epstein, D. B. A., {\em The simplicity of certain groups of homeomorphisms},
Compositio Math. \textbf{22} (1970), 165--173.

\bibitem[E2]{epstein:commutators} \bysame, {\em 
 Commutators of  $C^\infty$-diffeomorphisms. Appendix to: ``A curious remark concerning the geometric transfer map''by John N. Mather,
  [Comment. Math. Helv. \textbf{59} (1984), no. 1, 86--110.},
Comment. Math. Helv. \textbf{59} (1984), no. 1, 111–122.

\bibitem[Ha]{haefliger} Haefliger, A., {\em Feuilletage sur les vari\'et\'es ouvertes}, Topology
\textbf{9} (1970), 183--194.

\bibitem[LOTV]{LOTV}
 Le, Hong Van and Oh, Y.-G. and Tortorella, Alfonso and Vitagliano, Luca,
{\em Deformations of coisotropic submanifolds in {J}acobi manifolds},
J. Symplectic Geom., \textbf{16}, 2018, no. 4, 1051--1116.

\bibitem[Ly]{lychagin} Lychagin, V. V. Sufficient orbits of a group of contact diffeomorphisms
Mat. USSR-Sb. (N.S.) \textbf{33} (1977), no. 2, 223--242.

\bibitem[Ma1]{mather} Mather, John N., {\em Commutators of diffeomorphisms},
Comment. Math. Helv. \textbf{49} (1974), 512--528.
 
\bibitem[Ma2]{mather2} \bysame {\em Commutators of diffeomorphisms. II},
Comment. Math. Helv. \textbf{50} (1975), 33--40.

\bibitem[Ma3]{mather3} \bysame {\em On the homology of Haefliger's classifying space},
in: C.I.M.E. Differential Topology, 1976, pp. 71--116.

\bibitem[Ma4]{mather4} \bysame {\em A curious remark concerning the geometric transfer map},
Comment. Math. Helv. \textbf{59} (1984), no. 1, 86--110.

\bibitem[MS1]{mueller-spaeth1} M\"uller, S. and  Spaeth, P., {\em Topological contact dynamics I: symplectization and applications 
of the energy-capacity inequality}, Adv. Geom. \textbf{15} (2015), no. 3, 349--380.
\bibitem[MS2]{mueller-spaeth2} \bysame, 
{\em Topological contact dynamics II: topological automorphisms, contact homeomorphisms, 
and non-smooth contact dynamical systems}, 
Trans. Amer. Math. Soc. \textbf{366} (2014), no. 9, 5009--5041.
\bibitem[MS3]{mueller-spaeth3} \bysame,
{\em Topological contact dynamics III: uniqueness of the topological Hamiltonian and 
$C^0$-rigidity of the geodesic flow},
J. Symplectic Geom. \textbf{14} (2016), no.1, 1--29.

\bibitem[Oh10]{oh:hameo2}  Oh, Y.-G. {\em 
The group of Hamiltonian homeomorphisms and continuous Hamiltonian flows},
Contemp. Math., 512 American Mathematical Society, Providence, RI, 2010, 149--177.

\bibitem[Oh21a]{oh:contacton-Legendrian-bdy}
\bysame, \emph{Contact {H}amiltonian dynamics and perturbed contact instantons
  with {L}egendrian boundary condition}, preprint, arXiv:2103.15390(v2).

\bibitem[Oh21b]{oh:entanglement1} Oh, Y.-G.,
\bysame, \emph{Geometry and analysis of contact instantons and entangement of
  {L}egendrian links {I}}, submitted, arXiv:2111.02597.

\bibitem[Oh22a]{oh:shelukhin-conjecture}
\bysame, \emph{Contact instantons, anti-contact involution and proof of
  {S}helukhin's conjecture}, submitted, arXiv:2212.03557.
  
\bibitem[OM]{oh:hameo1} Oh, Y.-G., M\"uller, S., {\em 
The group of Hamiltonian homeomorphisms and $C^0$-symplectic topology},
J. Symplectic Geom. \textbf{5} (2007), no.2, 167--219.  

\bibitem[OW]{oh-wang3} Oh, Y.-G, Wang, R., {\em 
Analysis of contact Cauchy-Riemann maps II: Canonical neighborhoods and exponential convergence for the Morse-Bott case},
Nagoya Math. J. 231 (2018), 128--223.

\bibitem[Ryb1]{rybicki1} Rybicki, T., {\em On foliated, Poisson and Hamiltonian diffeomorphisms},
Differential Geom. Appl. \textbf{15} (2001), no. 1, 33--46.

\bibitem[Ryb2]{rybicki2} \bysame, {\em Commutators of contactomorphisms}, 
Adv. in Math. \textbf{225}
(2010), 3291--3326.

\bibitem[Th]{thurston} Thurston, W., {\em On the structure of the group of
volume preserving diffeomorphisms}, unpublished.

\bibitem[T1]{tsuboi1} Tsuboi, T., {\em
On the homology of classifying spaces for foliated products},
Adv. Stud. Pure Math., 5
North-Holland Publishing Co., Amsterdam, 1985, 37--120.

\bibitem[T2]{tsuboi2} \bysame, {\em On the foliated products of class  $C^1$},
Ann. of Math. (2) \textbf{130} (1989), no. 2, 227--271.

\bibitem[T3]{tsuboi3} \bysame, {\em On the simplicity of the group of contactomorphisms},
Adv. Stud. Pure Math., 52, Mathematical Society of Japan, Tokyo, 2008, 491--504.

\end{thebibliography}
\end{document}